\documentclass[11pt]{amsart}
\pdfoutput=1

\usepackage{amssymb,amsmath,amsthm,verbatim}
\usepackage{colortbl}

\usepackage{array}

\usepackage{mathtools}  % amsmath with extensions
\usepackage{amsfonts}  % (otherwise \mathbb does nothing)
\usepackage{amsthm}
\usepackage{graphicx}
\graphicspath{ {images/} }

\newtheorem{theorem}{Theorem}[section]
\newtheorem{proposition}[theorem]{Proposition}

\newtheorem{lemma}[theorem]{Lemma}

\theoremstyle{definition}
\newtheorem*{definition}{Definition}

\begin{document}

\title{On the Growth of the Wallpaper Groups}
\author{Rostislav Grigorchuk, Cosmas Kravaris}
\date{\today}
\maketitle

\begin{abstract} 
We develop further Cannon's method of cone types for finding the growth function of a group, which can also be used to find the coordination sequences of certain infinite graphs. We then apply this method to compute the growth functions and series of the wallpaper groups (the 2 dimensional crystallographic groups). The paper has a number of illustrating colored figures and tables summarizing the results.
\end{abstract}

\newpage
\pagenumbering{arabic}

\section{Introduction}

Given a locally finite rooted graph $\Gamma = (V, E, o)$ where $o \in V$ is a distinguished vertex (and as usual, V is the set of vertices and E the set of edges), one can consider the sequence:
$$\gamma(n) = \gamma_{\;\Gamma,o}(n) = |\{v\in V : d(o,v)\leq n\}|,$$
where $d(u,v)$ is the combinatorial distance in $\Gamma$, i.e. the number of edges in the shortest path connecting u and v in $\Gamma$.
We consider the case when $\Gamma$ is an undirected graph, so $d(u,v)$ is a well defined metric.
When $\Gamma$ is an infinite graph, the growth rate of $\gamma(n)$ when $n \to \infty$ is an important characteristic of the graph. The function $\gamma(n)$ is called the \textbf{growth function} or \textbf{coordination sequence} of $\Gamma$, and it appears in 
the study of lattices (\cite{Conway_Sloane_Coordination_Seq} and \cite{Coordinationsequencesforlattices}), 
the study of  crystallographic objects such as Zeolites in chemistry (\cite{Brunner_Zeolite_Chemistry_Paper} and \cite{Zeolite_LTA_Structure_Chemistry_Paper}), 
and also in geometric group theory \cite{dalaHarpe:GGTbook}.

Although $\gamma$ depends on the root $o$, the rate of growth of $\gamma(n)$ does not depend on the choice of $o$. In many cases it is convenient to represent growth functions by their corresponding growth series
$$ \Gamma(z) = \sum_{n=0}^{\infty}\gamma(n)z^n.$$
This is especially useful when the series represent a rational function, $\Gamma(z)=P(z)/Q(z)$ (where P and Q are polynomials). In this case, the growth rate of $\gamma(n)$ is polynomial or exponential, depending on the positive root of $Q(z)$, $z_0$, which is closest to zero along with its multiplicity when $z_0 = 1$. If $P(z)$ and $Q(z)$ are known, then classical analysis allows us to represent $\gamma(n)$ by a finite set of polynomials or exponential functions.

The papers of Conway and Sloane \cite{Conway_Sloane_Coordination_Seq}, Baake and Grimm \cite{BaakeGrimm::CoordinationSeq}, Goodman-Strauss and Sloane \cite{Coloring_Book_Approach}, Bacher, de la Harpe and Venkov \cite{Bacher_deLaHarpe_Venkov_Paper} give numerous examples for computations of the growth series and show the non-triviality of this task even in relatively easy situations.

The use of growth functions in group theory comes via the classical notion of a Cayley graph since the growth function of a group $G$ w.r.t a given generating set $S= \{s_1,...s_m\}$ coincides with the growth function on the corresponding Cayley graph $\Gamma(G,S)$. 
For groups, the growth rate of $\gamma(n)$ can be polynomial, exponential or of intermediate type (between polynomial and exponential). 
The question about the existence of groups of intermediate growth was raised by Milnor \cite{milnor:problem} and answered by the first author \cite{grigorch:milnor83}. 
For such groups, the growth series is transcendental.

There are several important classes of groups for which it is known that $\Gamma(z)$ is rational for any system of generators. 
For instance, this is true for virtually abelian groups (i.e. groups containing an abelian subgroup of finite index) (\cite{Benson_Paper} and \cite{Klarner_Paper_2}) and for Gromov hyperbolic groups \cite{HyperbolicGroupsBook}. 
However, given a group $G$ with a system of generators $S$ it is usually not easy to compute $\Gamma(z)$, even when its rationality is known. 
In addition to their relation to growth, another reason for computing the growth series of groups is that, for certain groups 
(e.g. of Coxeter type \cite{Serre_Paper_EulerCharacteristic} 
and for closed surface groups \cite{Cannon_EulerCharacteristic_Paper})
$\Gamma(1) = 1/\chi(G)$, where $\chi (G)$ is the Euler characteristic of the group. See \cite{Brazil_Paper} for further discussion on this topic.

An important class of virtually abelian groups are the \textbf{crystallographic groups}. 
By a classical result of Bieberbach in 1912, it is known that for any dimension d=2,3,..., there are finitely many such groups up to group isomorphism. For instance, when d=2, there are 17 groups (proven by Fedorov in 1891) and when d=3 there are 219 many (up to group isomorphism). 
Planar crystallographic groups (which are also called \textbf{wallpaper groups}) have a natural elegance because of the beauty of their Cayley graphs, presented for instance, in the book of Coxeter and Moser \cite{Coxeter_Moser_Book}. 
Coxeter and Moser choose the most natural system of generators arising from their geometry. 
Nonetheless, there are still alternate systems of generators for which the study of properties (including spectral properties) is important in models of image reconstruction of the human eye \cite{Health_Book_with_Groups} (as was indicated by A. Agrachev).

The graphs presented in the book of Coxeter and Moser are oriented. In the oriented case, one can also define a pseudometric on the graph, and hence also a growth function. 
In \cite{Shutov_Paper}, Shutov presented the growth functions of all 20 oriented graphs (some groups are presented by Coxeter and Moser with more than one system of generators and relators), using the method of the article of Zhuravlev \cite{Zhuravlev_Paper}.
Unfortunately, the expressions for the oriented version of $\gamma(n)$ are claimed to be valid for $n \geq n_0$ for some $n_0$ not indicated in \cite{Shutov_Paper}. 
Also, details of the proofs are missing and the growth of $\Gamma$ based on its orientation does not always coincide with the growth of $\Gamma$ as an undirected graph.

In this article, we rigorously compute $\Gamma(z)$ for all 17 wallpaper groups (20 group presentations in total), but for the non-oriented Cayley graphs determined by the same generating sets as in \cite{Shutov_Paper}. 
In 9 out of the 20 cases, our growth functions coincide with the ones in Shutov's paper, since all the generators are involutions (i.e elements of order 2) and so the edges in the Cayley graph are not oriented. It turns out that when we ignore the orientation, we only have 7 different Cayley graphs.
For the convenience of the reader, we present tables for both the growth series and the growth functions.

Our computations are based on the method of "cone types", first suggested by J. Cannon in \cite{cannon:growth80} and \cite{Cannon:Paper}. One of the important results of Cannon states that the Cayley graph of a classical hyperbolic group for any system of generators has only finitely many cone types (the same result also holds more generally for Gromov hyperbolic groups with basically the same proof).
This implies that the computation of $\Gamma(z)$ can be reduced to solving a system of linear equations, provided that the classification of cone types is known.
There is no such general result for a virtually abelian group even though we know that $\Gamma(z)$ in this case will be rational \cite{Benson_Paper}.
But at least for the wallpaper groups we have finitely many cone types when we correctly define the notion of a cone type.
In sections 3.2-3.5 we introduce several variations of the notion of a cone type and finally arrive at the notion of an extended cone type. 
Then in section 3.6, we show that, in general, extended cone types yield a system equations for the $\Gamma(z)$ and in section 3.8 we provide a systematic approach on how to rigorously arrive at these equations.
In sections 4.1-4.7, we apply this method for the wallpaper groups (in particular, their seven different non oriented Cayley graphs)
We note that in some cases the number of extended cone types can be quite large (in the graph 4.6.12 it is 42) while their shape can be quite complicated. Our results are summarized in section 5.

\section{Preliminaries}

	\subsection{Cayley Graphs}
		We begin by giving a short account of what we mean by a Cayley graph.
		Let $G$ be a  group. A subset $ S \subset G $ is a \textbf{generating set} for $G$
		if every element $ g \in G $ can be expressed as the product of elements in $S$ and their inverses. That is, 
		$$ g = s_{i_1}^{\epsilon_1}  ...  s_{i_n}^{\epsilon_n}, \;\;\;\; (1) $$
		where $ s_{i_j} \in S$ (not necessarily distinct) and $ \epsilon_j \in \{+1,-1\}$  for each $ j = 1 ... n $
		A group is \textbf{finitely generated} if it has a finite generating set.
		From now on, we assume that G is finitely generated.
		
		Now, fixing a generating set $S$, an element $g \in G$ can be represented as in (1) in many different ways,
		and the smallest possible index $n\in \mathbb{N}$ in (1) is called the \textbf{length} of g, denoted by $ |g|_S $. 
		Observe that the length of g depends on our choice of the generating set $S$.
		
		Sometimes we will deal with sets $S \subset G $ which generate G as a \textbf{semigroup}.
		This means every $g \in G$ can be presented in the form
		$$ g = s_{i_1} ...  s_{i_n}, \;\;\;\; (2) $$
		where $ s_{i_j} \in S$ for each $ j = 1 ... n $.
		For example, $\{2,3\}$ is merely a group generating set of Z while $\{-2,3\}$ is a semigroup generated set of Z
		Also note that every group generating set $S$ can be converted into a semigroup generating set
		$ S \cup S^{-1} $ where $ S^{-1} = \{s^{-1}|s\in S\} $. 
		
		Given a group generating set $S \subset G$, the \textbf{(left) Cayley graph} $\Gamma_l = \Gamma_l (G,S)$
		is a directed labeled graph with set of vertices $ G $ and set of oriented edges $ e $ of the form
		$$ e = (g, sg), \;\;\;\; where\;\; g\in G, s\in S. \;\;\;\; (3)$$
		The direction of $e$ is from $g$ to $sg$, while the label of $e$ in (3) is the generator s.
		We shall omit the (G,S) part of
		the notation when the underlying group and generating set are obvious.
		We can similarly define the \textbf{right Cayley graph} $\Gamma_r = \Gamma_r (G,S)$, which has directed edges of the form 
		$e = (g, gs)$ similar to (3).
		
		Furthermore, the group $G$ naturally acts on the structure of $\Gamma_l$ (i.e. on $G$) by right multiplication. 
		Similarly, the action of left multiplication preserves the structure of the right Cayley graph $\Gamma_r$. 
		In fact, there is a natural isomorphism between left and right Cayley graphs; as a result, we will only work
		with left Cayley graph and write $\Gamma$ for $\Gamma_l$.

	\subsection{Directions and Labels}
		
		Our notation will be as follows: 
		From now on, we denote by
		$\Gamma = \Gamma(G,S)$, $\Gamma_{+} = \Gamma_{+}(G,S)$, and $\Gamma_{\#}^{+} = \Gamma_{\#}^{+}(G,S)$ the undirected unlabeled, directed unlabeled
		and directed labeled (left) Caley graphs of $ G $ w.r.t. $ S $ (respectively). 
		Then
		$ Aut \Gamma$ is the group of undirected unlabelled (graph) automorphisms of $ \Gamma$, 
		$ Aut \Gamma^{+}$ is the group of directed unlabelled automorphisms of $ \Gamma^{+}$, and  
		$ Aut \Gamma_{\#}^{+}$ is the group of directed labelled automorphisms of $ \Gamma_{\#}^{+}$.
		Under this notation, we can see that via the group action in the above paragraph, G can be viewed as a subgroup of $ Aut_{\#}^{+} \Gamma$.
		In fact, the group $ G $ is naturally isomorphic to $ Aut \Gamma_{\#}^{+} $. In many situations, it is reasonable to ignore the directions and labels and view $ G $ as a subgroup of the (often) much larger group $Aut \Gamma$.
		
		Observe that $\Gamma$ is $2 | S| $ regular (that is, each vertex is adjacent to $2 |S| $ vertices). Also note that if the identity element
		$1$ is not in $S$, then $\Gamma$ has no loops and if the elements of $S$ are all distinct and not of order 2, then $\Gamma$ has no
		double edges (but there could be situations where it is reasonable to allow repetitions of generators
		 or the presence of the identity in S).
		\begin{figure}[h]
		\begin{center}
			\includegraphics[scale=0.3]{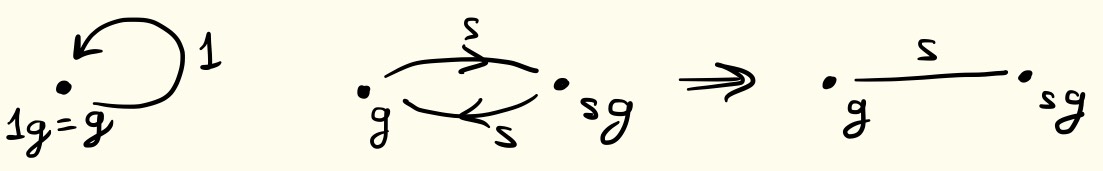}
			\caption{Loops and double edges.}
			\label{fig:2.2.1}
		\end{center}
		\end{figure}
		Another important point is treating generators of order 2. If $s \in S$ has order 2 (then $s = s^{-1}$), so it
		usually convenient to replace a pair of edges $(g,sg), (sg,g)$ with a single non-oriented edge. 
		Several examples of wallpaper groups considered in this article have generating elements of order 2,
		and we will follow the above convention. Finally, if $S = \{s_1, s_2, ... s_k, s_{k+1}, ..., s_m\}$ where all the $s_i$
		are distinct, $s_{k+1}, ..., s_m$ have order 2 while $s_1, s_2, ... s_k$ have order greater than 2,
		then $\Gamma(G,S)$ is a d-regular graph with $d = 2k + (m-k) = k+m$.	

	\subsection{Combinatorial Distance}
		
		For any given connected undirected graph $(V,E)$ with vertices $V$ and edges $E$, we define a \textbf{discrete metric} on $V$ by letting the distance between two vertices $u,v$ be the minimal combinatorial length of paths connecting $u$ to $v$ (the combinatorial length
		of a path p is the number of edges in p). In the case of the undirected graph $\Gamma$, we will denote this metric by $d_\Gamma$. It is not difficult to see that $d_\Gamma(g,h) = |g h^{-1}|$, (where $|\cdot|$ is the length
		of an element w.r.t. the generating set S), $(G,d_\Gamma)$ is a metric space, and $d_\Gamma$ is a right invariant metric (i.e. the action of G on the vertices via right multiplication consists of isometries w.r.t $d_\Gamma$). Note that since $\Gamma$ depends on the generating set $S$, so does the metric $d_\Gamma$.
		
		Sometimes, mathematicians (for instance Cannon in \cite{Cannon:Paper}) extend the metric on the whole graph $\Gamma$ 
		with the edges viewed as the unit interval with the standard metric. This turns $\Gamma$ into a
		path connected metric space.
		
		Another variation is considering a "distance" on the directed Caley graph $\Gamma^{+}$. For $g,h \in G$, the \textbf{discrete directed metric} 
		$d^{+}_\Gamma(g,h)$ is the minimal combinatorial length of an oriented path connecting g with h in $\Gamma^{+}$.
		When no such path exists, the distance in infinite. Observe that although $d^{+}_\Gamma$ satisfies the triangle inequality, 
		it is not always symmetric, and therefore $(G,d^{+}_\Gamma)$ is not a metric space in general.

	\subsection{Growth Function and Growth Series}

	Using the metric $d_\Gamma$ or the length function $|\cdot|$ we define the \textbf{cumulative (or volume) growth function} of $G \; w.r.t.\; S$:
	$$\gamma(n) :=  |\{g \in G : |g| \leq n \}| =  |\{g \in G : d_\Gamma(1,g) \leq n \}| \;\;\;\forall n \in \mathbb{N},$$
	and the \textbf{spherical growth function}:
	$$ \delta(n)  =  |\{g \in G : |g| = n \}| =  |\{g \in G : d_\Gamma(1,g) = n \}| \;\;\;\forall n \in \mathbb{N}.$$	
	Note that we have the obvious relation:
	$\gamma(n) = \sum_{k=0}^{n} \delta(k)$.\\
	We also define the corresponding generating functions of these sequences,
	the \textbf{cumulative (or volume) growth series}:
	$$ \Gamma(z)=\sum_{n=0}^{\infty}\gamma(n)z^{n} \;\;\;\forall z \in \mathbb{C}, $$
	and the \textbf{spherical growth series}:
	$$ \Delta(z) = \sum_{n=0}^{\infty} \delta(n) z^{n} \;\;\;\forall z \in \mathbb{C}.$$
	From the previous relation, we get
	$\Gamma(z) = \dfrac{\Delta(z)}{1-z}$.
	
	A variation of growth considered by Shutov in \cite{Shutov_Paper} uses the discrete directed metric $d^{+}_\Gamma$ instead $d_\Gamma$. The \textbf{directed cumulative growth function} is
	$$\gamma_{+}(n) =  |\{g \in G | d^{+}_\Gamma(1,g) \leq n \}|.$$
	Similarly we may define $\delta_{+}(n),\; \Gamma_{+}(z),\;\Delta_{+}(z)$, with the same relations holding as above.
	
	Note that all the growth functions and growth series depend on the choice of generating set $S$, but the rate of growth as defined by Milnor \cite{milnor:problem} does not depend on $S$.
	
	Finally, the cumulative growth series $\Gamma(z)$ have a radius of convergence
	$$ R =\dfrac{1}{  \limsup_{n \to \infty} \gamma(n)^{1/n}  },$$
	and hence determines an analytic function in a ball of radius $R$ centered at 0. We are mainly interested
	in the rationality of the growth series $\Gamma(z)$ and obtaining  polynomials $P(z)$ and $Q(z)$ such that
	$$ \Gamma(z) = \dfrac{P(z)}{Q(z)}.$$

	\subsection{The Wallpaper Groups and their Cayley Graphs} 

	A wallpaper group (or planar crystallographic group) is a discrete subgroup
	of isometries of the Eucledian plane that contains 2 linearly independent translations.
	To be more precise, a subgroup $G$ of the group of isometries of $\mathbb{R}^2$ is a  \textbf{wallpaper group} 
	if there exists a compact set $W \subset \mathbb{R}^2$, called a fundamental domain, such that the following hold:
	$$\cup_{g \in G}\;gW = \mathbb{R}^2 \;\;\; and \;\;\; \forall g \neq h \in G \;\;\; Area( gW \cap hW) = 0\;,$$
	where "Area" denotes the Lebesgue measure on $\mathbb{R}^2$.\
	
	In 1891, Fedorov proved that there exist only 17 wallpaper groups. 
	Their presentations by generators and relators, and figures of the (oriented) Cayley graphs
	are presented, for instance, in the classic book by Coxeter and Moser \cite{Coxeter_Moser_Book}. 
	It turns out that, up to graph isomorphism, there are only 7 non-oriented Cayley graphs and these graphs can all be realized as tessellations of $\mathbb{R}^2$.
	We denote these 7 graphs by their Schläfli symbols $3^6, 4^4, 6^3, (3.6)^2, 4.8^2, 3.12^2$ and $4.6.12.$ 
	as shown in Figure \ref{fig:2.5.1}, and we denote the 17 wallpaper groups by their Hermann-Mauguin codes shown in the first column of Table \ref{table1}. That way, we follow classical notations used  in \cite{Coxeter_Moser_Book} and \cite{Coloring_Book_Approach} and other places.
	
	\begin{figure}[h]
		\begin{center}
			\includegraphics[scale=0.08]{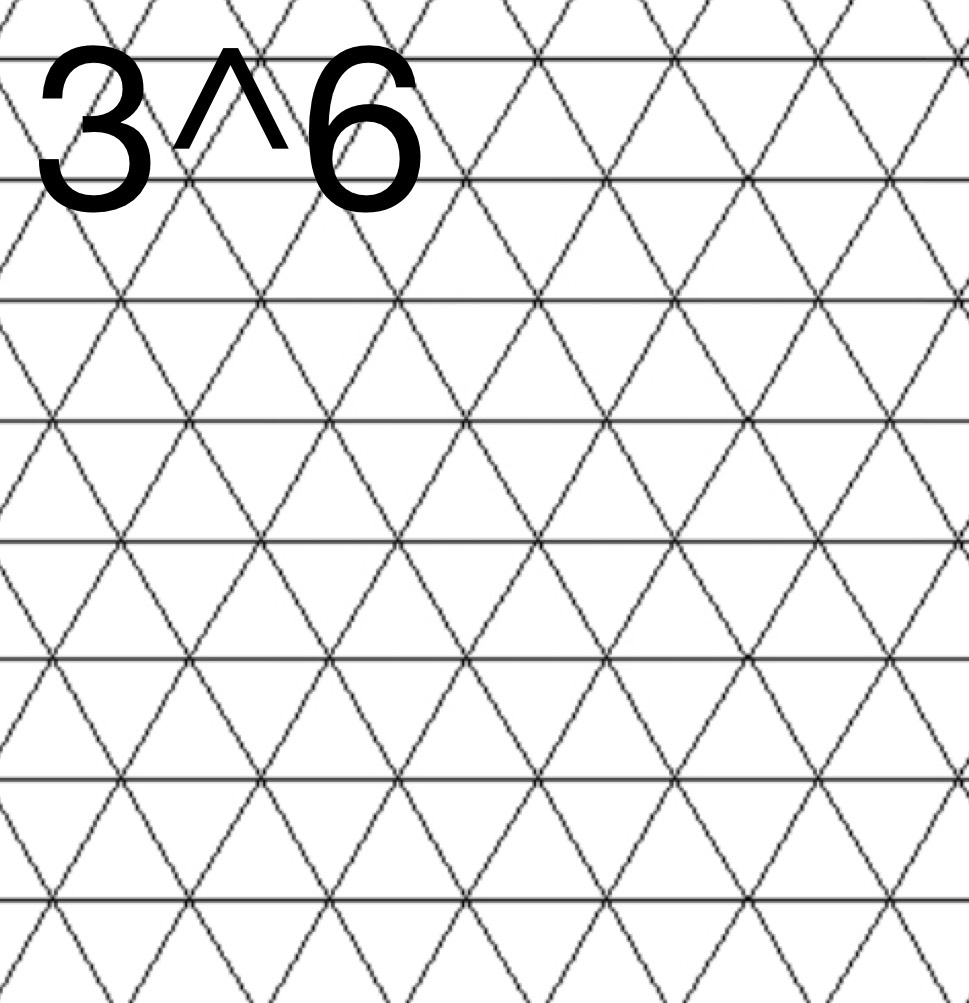}
			\includegraphics[scale=0.08]{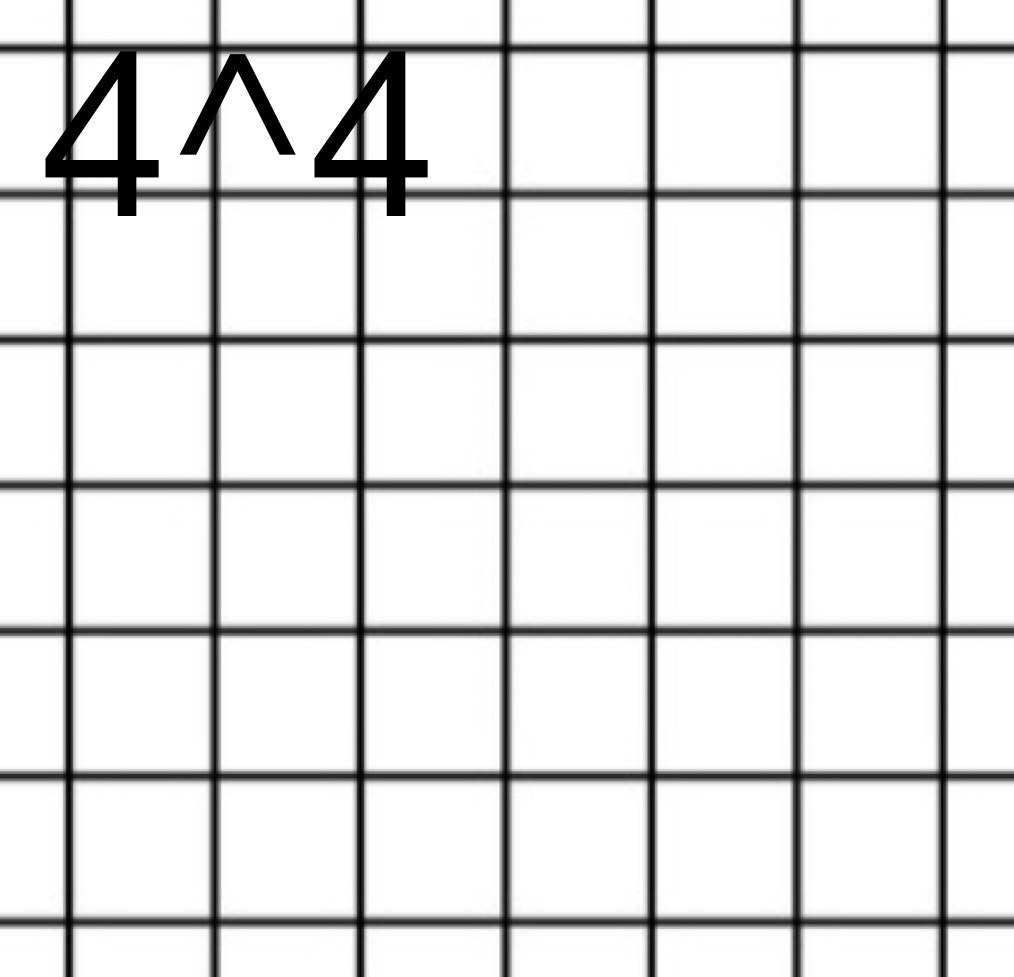}
			\includegraphics[scale=0.08]{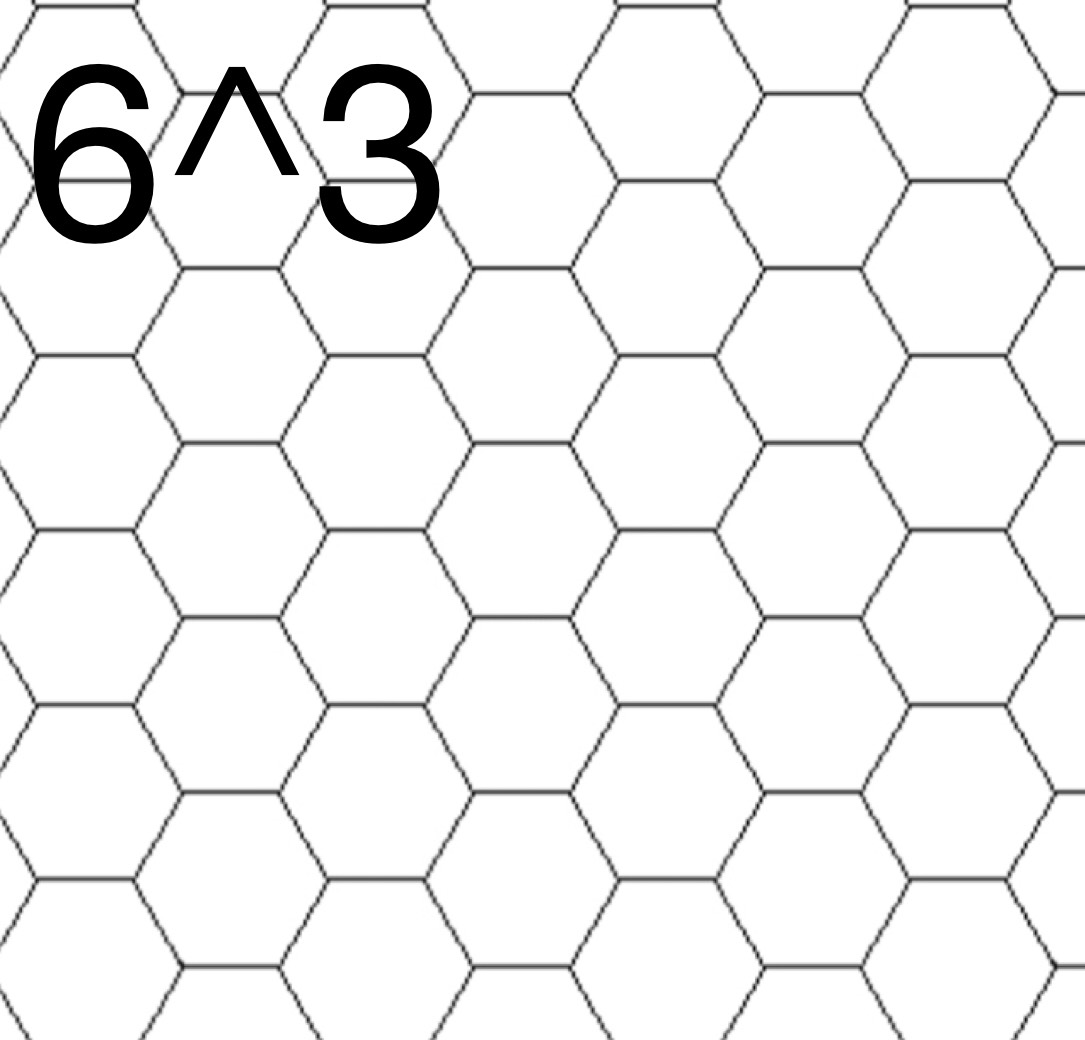}
			\includegraphics[scale=0.09]{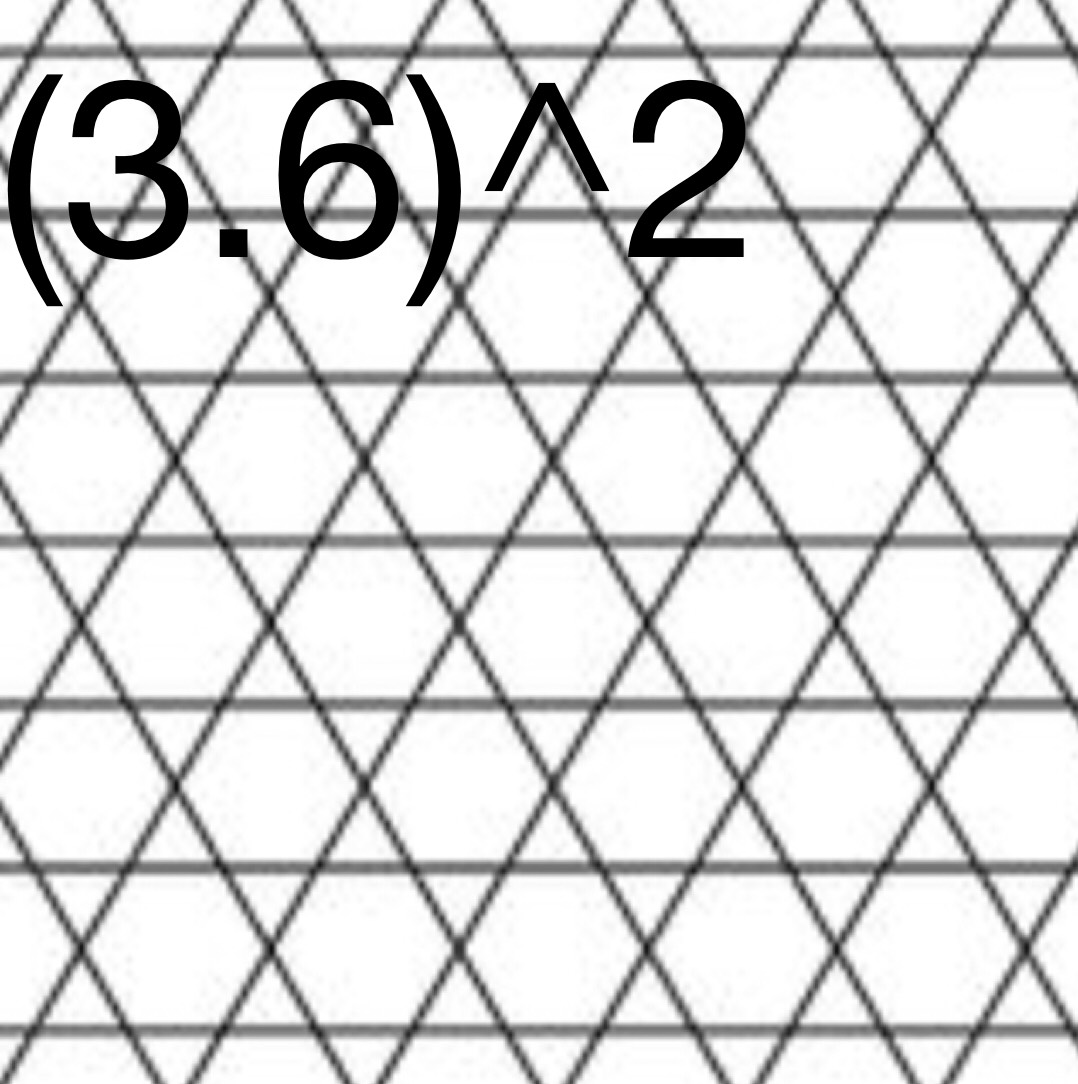}
			\includegraphics[scale=0.09]{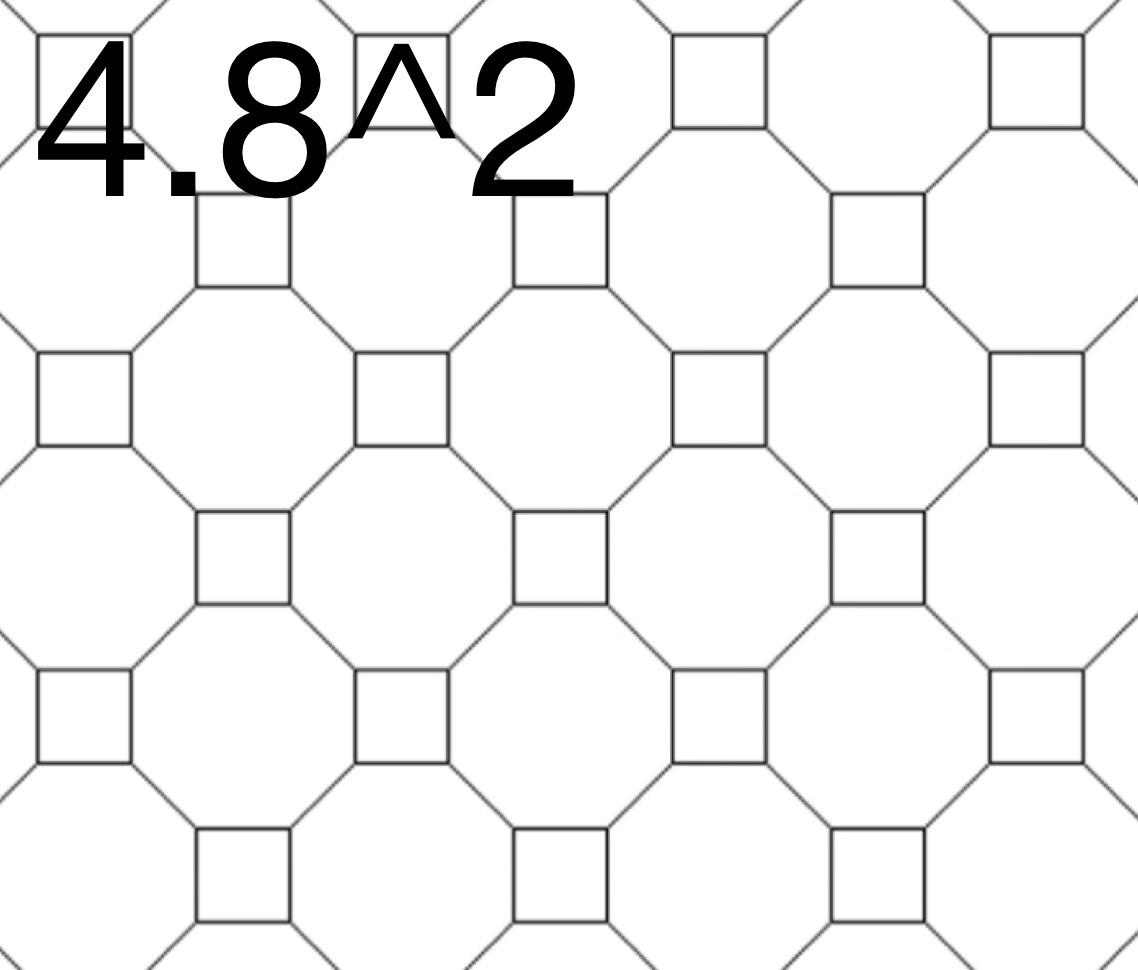}
			\includegraphics[scale=0.1]{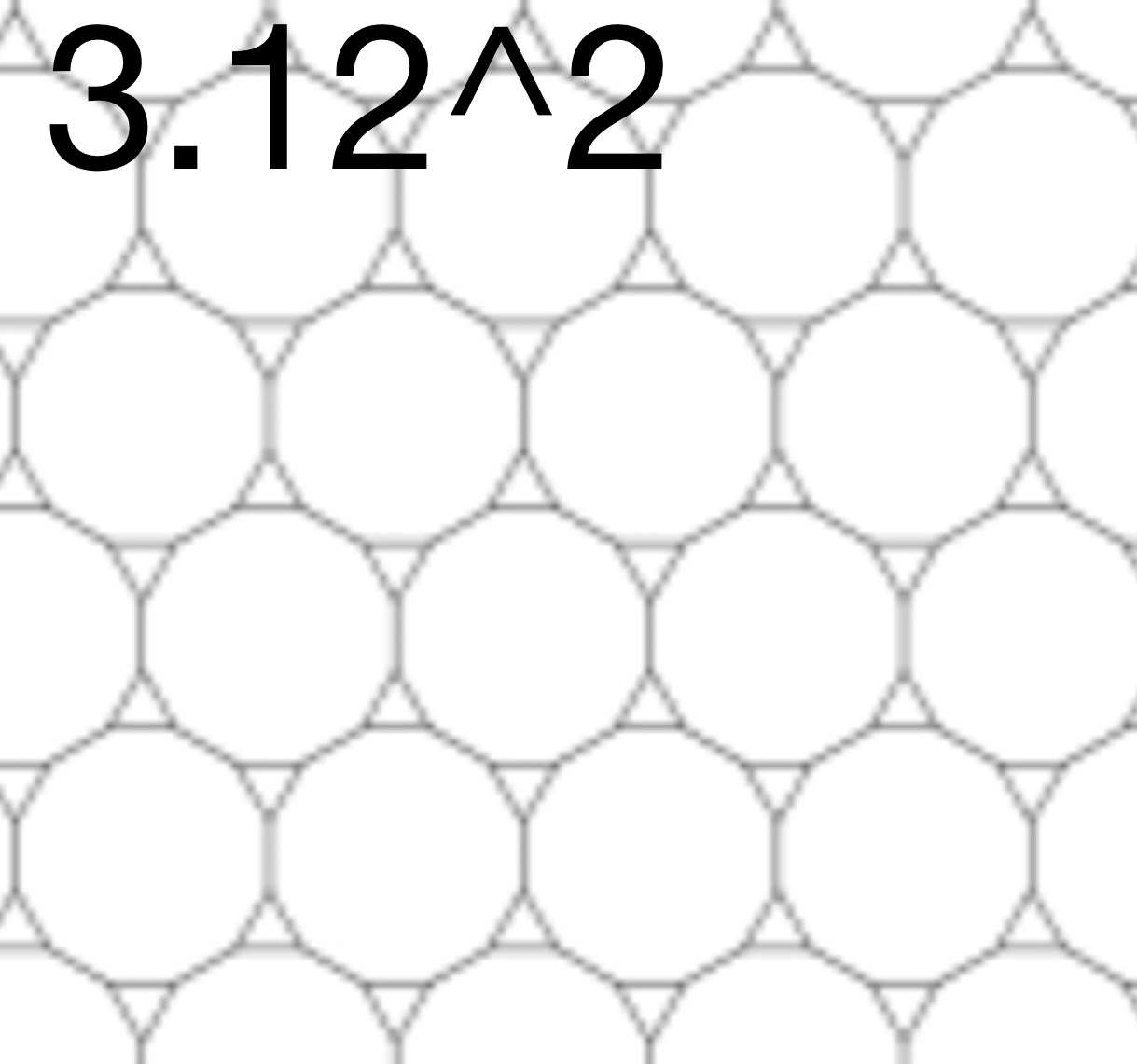}
			\includegraphics[scale=0.1]{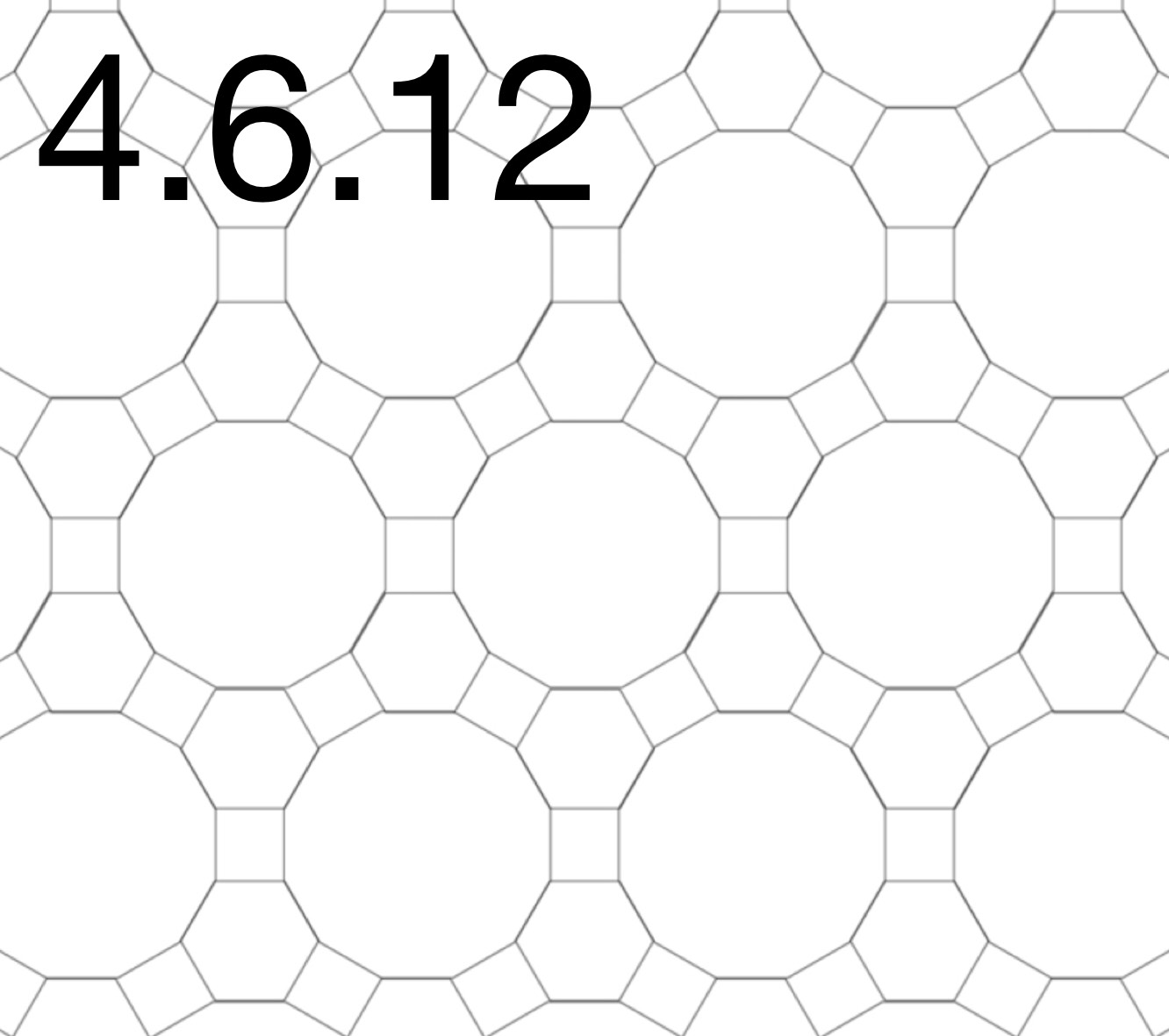}
			\caption{The Seven Cayley graphs of the wallpaper groups.}
			\label{fig:2.5.1}
		\end{center}
	\end{figure}
	We are going to calculate the spherical and cumulative growth functions and growth series of the wallpaper groups.
	Since the growth series of a group depends only on its non-oriented Cayley graph,
	we only need to consider the 7 graphs above. Then Table \ref{table1} shows how to obtain the growth series of each wallpaper group, as long as we know the growth series for each of the 7 graphs. We emphasize that for each of the 17 groups we use generating sets found in \cite{Coxeter_Moser_Book} to present a group via generators and relations. This choice is quite classical, but nonetheless depending on the situation there could be other reasonable systems of generators.
	
	\begin{table}[h]
	\begin{center}
	\begin{tabular}{ | c | c | m{20em} | c |}
	\hline
	\hline
	Group & Generators & Relations & Cayley Graph \\
	\hline
	\hline
	p1   	& $X, Y$ 				& $XY=YX$												& $4^4$\\
	\hline
	p1   	& $X,Y,Z$				& $XYZ = ZYX = 1$										& $3^6$\\
	\hline
	p2   	& $T_1,T_2,T_3$			&$T_1^2 = T_2^2 = T_3^2 = (T_1 T_2 T_3)^2=1$			& $6^3$\\
	\hline
	p2   	& $T_1,T_2,T_3,T_4$ 	&$T_1^2 = T_2^2 = T_3^2 = T_4^2 = T_1 T_2 T_3 T_4 = 1$	& $4^4$\\
	\hline
	pm  	& $Y,R,R'$ 				&$R^2 = R'^2 = 1, RY=YR, R'Y=YR'$						& $4^4$ \\
	\hline
	pg   	& $P,Q$					&$P^2 = Q^2$											& $4^4$ \\
	\hline
	cm   	& $P,R$ 				&$R^2=1,R P^2=P^2 R$									& $6^3$ \\
	\hline
	pmm 	&$R_1,R_2,R_3,R_4$		
	%& $\begin{array}{cc}
	%R_1^2 = R_2^2 = R_3^2 = R_4^2  = & _ \\
	%(R_1 R_2)^2 = (R_2 R_3)^2 = (R_3 R_4)^2 = (R_4 R_1)^2 = %1 & _ \end{array}$
	 &$R_1^2 = R_2^2 = R_3^2 = R_4^2  = (R_1 R_2)^2 = (R_2 R_3)^2 = (R_3 R_4)^2 = (R_4 R_1)^2 = 1$
	& $4^4$ \\
	\hline
	pmg 	& $R,T_1,T_2$ 			&$R^2=T_1^2=T_2^2=1, T_1 R T_1 = T_2 R T_2$				& $6^3$ \\
	\hline
	pgg 	& $P,O$					&$(PO)^2=(P^{-1}O)^2=1$									& $4^4$ \\
	\hline
	cmm 	& $R_1,R_2,T$ 			&$R_1^2 = R_2^2 =T^2=(R_1 R_2)^2 = (R_1 T R_2 T)^2 = 1$	& $4.8^2$ \\
	\hline
	p4  	& $S,T$ 				&$S^4 = T^2 = (ST)^4 = 1$								& $4.8^2$ \\
	\hline
	p4m 	& $R,R_1,R_2$ 			&$R^2=R_1^2=R_2^2=(RR_1)^4=(R_1R_2)^2=(R_2R)^4=1$		& $4.8^2$ \\
	\hline
	p4g 	& $R,S$					&$R^2=S^4=(RS^{-1}RS)^2=1$								& $4.8^2$ \\
	\hline
	p3 		& $S_1,S_2,S_3$ 		&$S_1^3=S_2^3=S_3^3=S_1S_2S_3=1$						& $3^6$ \\
	\hline
	p3 		& $S_1,S_2$ 			&$S_1^3=S_2^3=(S_1S_2)^3=1$								& $(3.6)^2$ \\
	\hline
	p31m 	& $R,S$ 				&$R^2=S^3=(RS^{-1}RS)^3=1$								& $3.12^2$ \\
	\hline
	p3m1 	& $R_1,R_2,R_3$ 		&$R_1^2=R_2^2=R_3^2=(R_1R_2)^3=(R_2R_3)^3=(R_3R_1)^3=1$	& $6^3$ \\
	\hline
	p6 		& $S,T$ 				&$S^3=T^2=(ST)^6=1$										& $3.12^2$ \\
	\hline
	p6m 	& $R, R_1,R_2$ 		&$R^2=R_1^2=R_2^2=(R_1R_2)^3=(R_2R)^2=(RR_1)^6=1$& $4.6.12$ \\
	\hline
	\hline
	\end{tabular}
	\end{center}
	\caption{The wallpaper groups given via generators and relations}
	\label{table1}
	\end{table}

\section{Cone Types}
	
	The idea of cone types is credited to Cannon (\cite{cannon:growth80} and \cite{Cannon:Paper}). We introduce modified versions
	of a cone and a cone type and use them in Section 4 to calculate the growth series and growth functions of the wallpaper groups.

	\subsection{Geodesic Words and Geodesic Growth}
	
		Once again, let $G$ be a group and $S$ be a group generating set of $G$. 
		Let $A := S \cup S^{-1}$ and consider $A$ as an alphabet.
		By $A^{*}$ we denote the set of all finite words (strings) over A, including the empty word $\o$.
		To every word $w=a_1... a_n \in A^*$ corresponds an element $\bar{w} \in G$ and a path in $\Gamma_{\#}^{+}(G,S)$
		that begins at $1 \in G$ and follows the directed labeled edges according to the letters $a_1, ... a_n$ 
		(if $a_i=s^{-1}$ for some $s \in S$ then we follow the corresponding edge in the opposite direction) 
		In addition, we have the special case $\bar{\o} = 1$.
		
		A word $w \in A$ is a \textbf{geodesic word} (or simply geodesic) 
		if its string length $|w|$ coincides the length of the element $\bar{w}$ it represents.
		Equivalently, the path represented by $w$, $p_w$ is one of the shortest paths connecting $1$ with $\bar{w} \in G$
		and is called a \textbf{geodesic path} in $\Gamma^{+}_{\#}$.
		Surely geodesic words are freely reduced (i.e.there is no occurrence of $s s^{-1}$ or $s^{-1}s$ for $s \in S$ in $w$; equivalently, there is no backtracking in $p_w$)
		
		Let $Geo(G,S) \subset A^*$ be the set of all geodesic words. We call $Geo(G,S)$ the language of geodesics.
		Define the \textbf{geodesic growth function} of $G$ w.r.t. $S$ to be the sequence: $\forall n \in \mathbb{N}$
		$$l(n) = |\{w \in Geo(G,S) | |w| = n\}| \;\;\;\forall n \in \mathbb{N},$$
		and the \textbf{geodesic growth series} to be:
		$$L(z) = \sum_{n=0}^{\infty} l(n) z^n \;\;\;\forall n \in \mathbb{C}.$$
		It is clear that $\delta(n) \leq l(n)$ for any n.

	\subsection{Word Cones}

	We begin with the definition of a word cone following \cite{Automatic_Group_Theory_Book}.
	Let $g \in G$ and choose $u \in Geo(G,S)$ such that $\bar{u}=g$.
	The \textbf{word cone} $Cone(g)$ of $g$ is the set of all words
	$w$ in our alphabet $A$ such that the concatenated word $uw$
	is a geodesic, i.e.
	$$ Cone(g) \; := \; \{ w \in A^* | uw \in Geo(G,S) \} \; = \; \{ w \in Geo(G,S) | uw \in Geo(G,S) \}. $$ 
	It is clear from the definition that $Cone(g)$ does \textit{not}
	depend on the choice of $u$. Next, two elements $g, h \in G$
	are said to have the same \textbf{word cone type} if
	$Cone(g) = Cone(h)$.
	
	For every geodesic word $w \in Geod(G,S)$ corresponds a geodesic path $p_w$ from $1$ to $\bar{w}$, which we may translate (since right multiplication is an automorphism in the left Cayley graph $\Gamma_{\#}^{+}$) to a geodesic path from $g$ to $g \bar{w}$. As a result, we may geometrically visualize $Cone(g)$ as the set of all geodesics starting from $g$ which extend any geodesic from $1$ to $g$. For an example, consider the wallpaper group p1 with presentation $<X,Y,Z|XYZ=ZYX=1>$ whose Cayley graph is $3^6$, and look at $g=\bar{ZX^{-1}}$ (see figure \ref{fig:3.2.1} below). Then the geodesic word $w=X^{-2}Z$ belongs to the word cone of $g$. In fact the geodesic word $w'=X^{-1}ZX^{-1}Z$ is also in $Cone(g)$. Observe that even though $\bar{w}=\bar{w}'$, $w$ and $w'$ represent two different elements of the word cone.
	
	\begin{figure}[h]
		\begin{center}
			\includegraphics[scale=0.17]{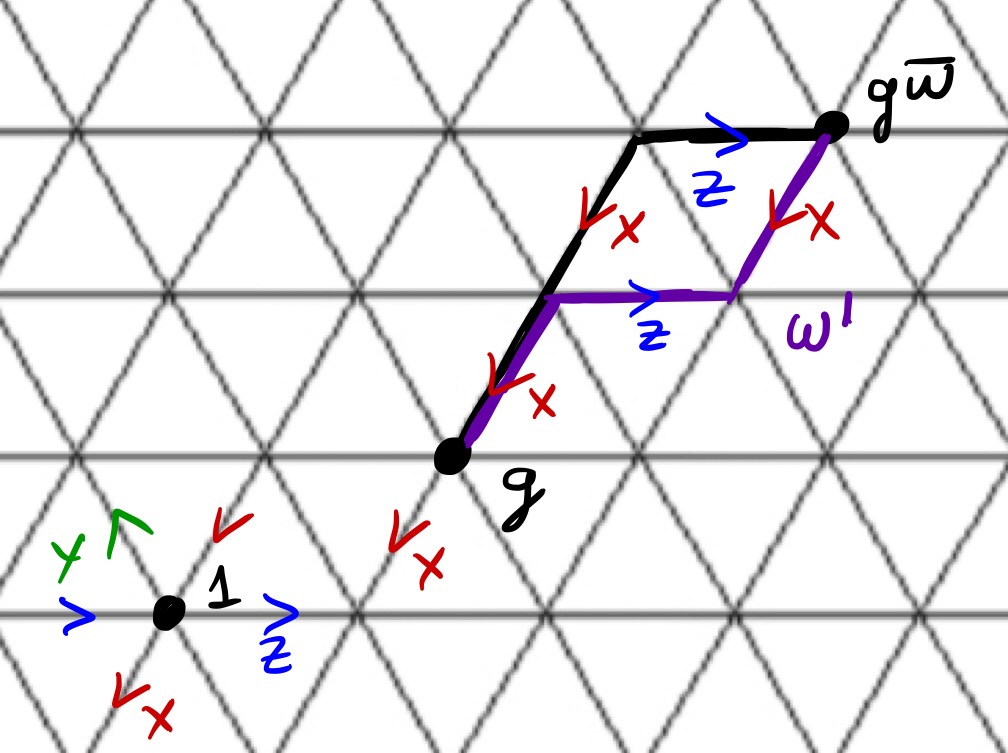}
		\end{center}
		\caption{Example of elements of a word cone.}
		\label{fig:3.2.1}
	\end{figure}
	
	The definition of a word cone turns out
	to be useful in the study of groups with automatic structure as explained in \cite{Automatic_Group_Theory_Book}.
	We should note that in automatic group theory, word cones are referred to simply as cones. 
	For the purposes of this paper, we mainly focus on a version of cones and cone types which is closer to the original definition proposed by Cannon \cite{Cannon:Paper}. We will call this 
	version	(graph) cones and (graph) cone types.

	\subsection{Graph Cones}
	
	For $g,h \in G$, we say that $g \leq h$ when there exists a geodesic path on $\Gamma$ from the identity element $1$ to $h$ which passes through g. It is not hard to see that $(G,\leq)$ is a partial order on $G$.
	
	We now define graph cones. For each $g\in G$, the \textbf{(graph) cone} $C(g)$ is a sub-graph of $\Gamma$ whose vertices are
	$$ \{ h \in G | g \leq h \} = \{h \in G | \exists \; geodesic \; path \; from \; 1 \; to \; h \; passing \; through \; g\}, $$
	and whose edges are all the edges of $\Gamma$ which connect the vertices of $C(g)$.
	We call vertex $g$ the \textbf{root vertex} of the (graph) cone $C(g)$.
	For the rest of this paper, we shall refer to graph cones simply as \textbf{cones}. Figure \ref{fig:3.3.1} shows three examples of cones in the graphs $3^6$, $(3.6)^2$ and $4.6.12$.
	
	\begin{figure}[h]
	\begin{center}
		\includegraphics[scale=0.1]{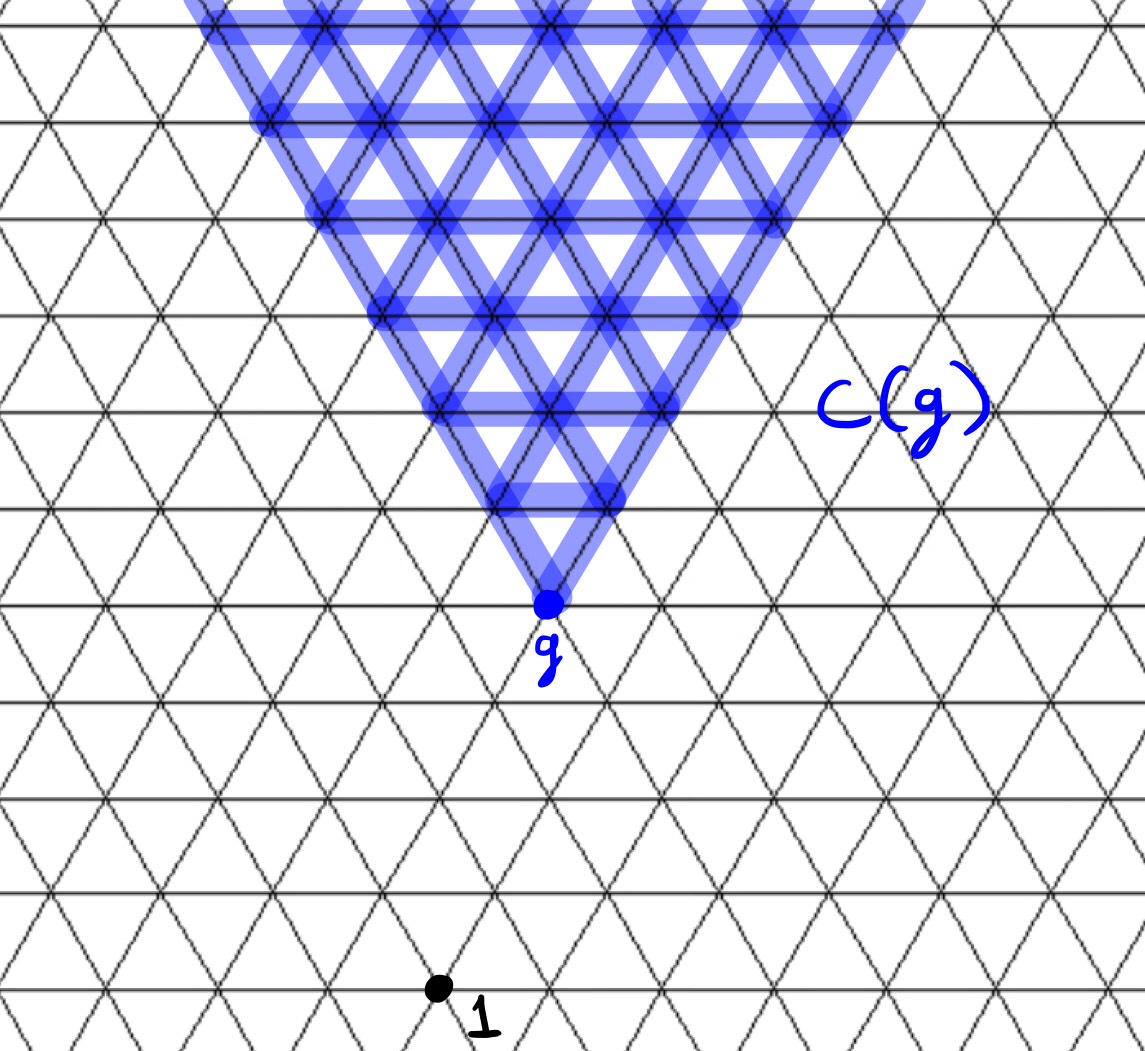}
		\includegraphics[scale=0.1]{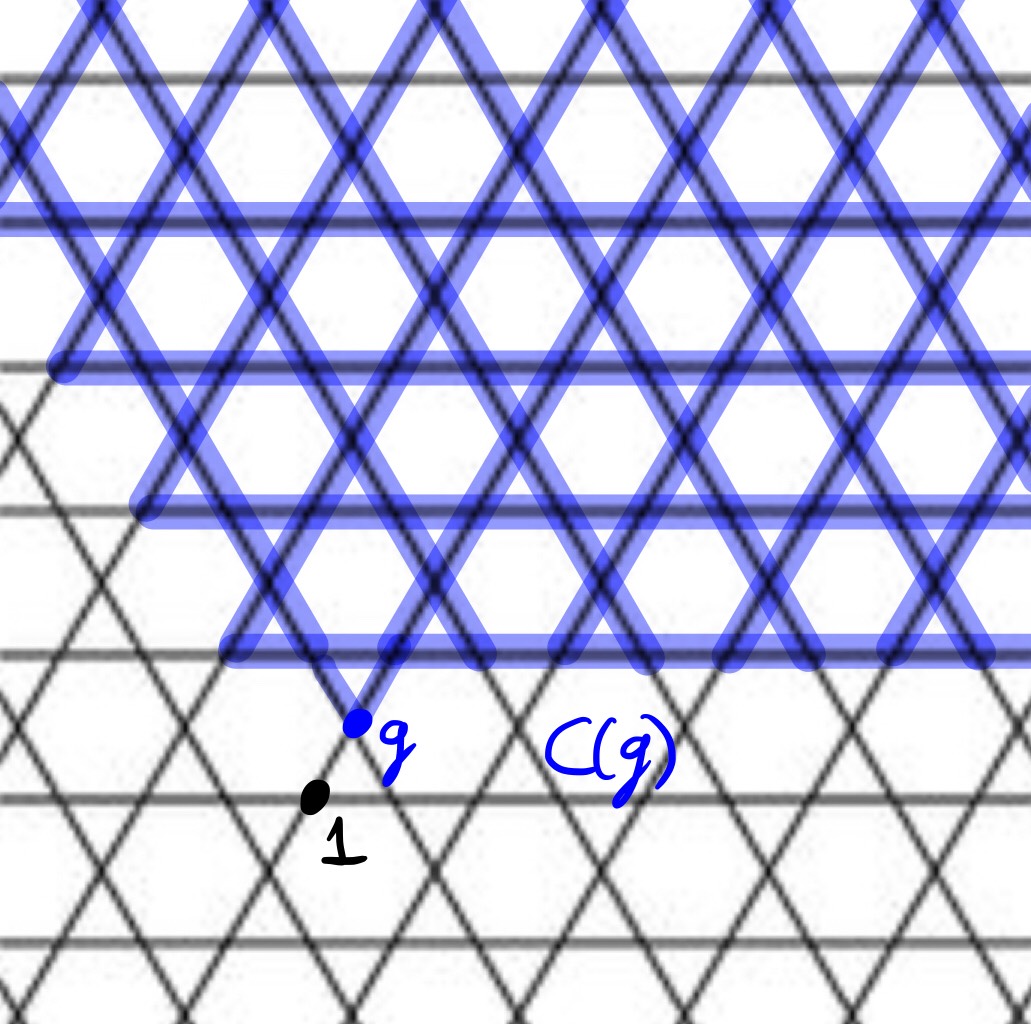}
		\includegraphics[scale=0.1]{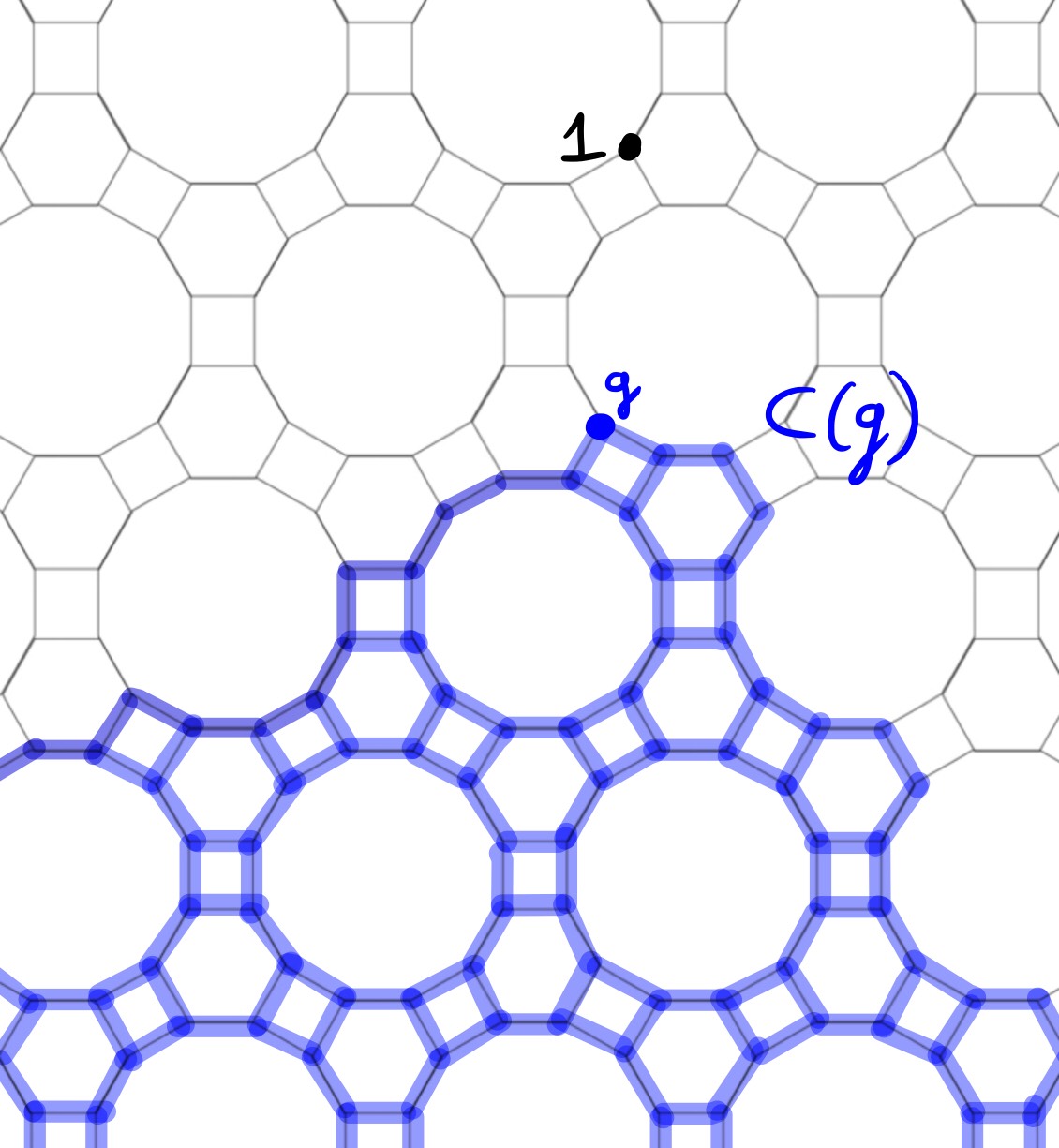}
		\caption{Examples of graph cones.}
		\label{fig:3.3.1}
	\end{center}
	\end{figure}

	Next, we say that two cones $C(g)$, $C(h)$ are \textbf{weakly equivalent} (written as $C(g) \sim C(h)$) when there is a graph isomorphism $T: C(g) \rightarrow C(h)$ (i.e. T is a 1-1 correspondence of vertices to vertices, edges to edges and preserves the graph structure) sending the root vertex $g$ to the root vertex $h$. So the \textbf{(graph) cone type} of $g$ is the equivalence class (formed by the above notion of weak equivalence) in which $C(g)$ belongs to.
	The notion of cone types will be useful in obtaining recursive relations for the geodesic growth of G. 

	\subsection{The Tangent Edges}
	
	Next we partition all the edges of the non-oriented unlabeled Cayley graph $\Gamma$ into two sets. We say that an edge $e=(g,h)$ ($g,h\in G$) is \textbf{tangent} if $|g|=|h|$, and is \textbf{normal} otherwise, i.e. when $||g|-|h||=1$. 
	Furthermore, we write \\
	$A_g$ for the set of all 
	edges (without orientation) that are attached to $g$, \\
	$A^t_g$ for the set of all tangent edges in $A_g$, \\
	$A^+_g$ for the set of all normal edges $(g,h)$ in $A_g$ with $|h|=|g|+1$ and \\
	$A^-_g$ for the set of all normal edges $(g,h)$ in $A_g$ with $|g|=|h|+1$ \\
	The edges in $A^+_g$ are often called the \textbf{norm increasing edges} starting at $g$, and \textbf{norm decreasing edges} for those in $A^-_g$.
	
	It is important to remark that if $\Gamma$ is a bipartite graph (meaning that the graph has no closed paths of odd length) then $\Gamma$ has no tangent edges (for if $(g,h)\in A^t_g$, then the loop starting at $1$ going, via a geodesic, to $g$ then to $h$ and next going back to $1$ via a geodesic has length $2|g|+1$ and is thus an odd length closed path).  
	
	Out of the 7 Cayley graphs that we will study, 4 are bipartite: $4^4$, $6^3$, $4.8^2$ and $4.6.12$.
	Figure \ref{fig:3.4.1} sketches with color the tangent(with red) and normal(with blue) edges of each of remaining 3 non-bipartite Cayley graphs: $3^6$, $(3.6)^2$ and $3.12^2$. The tangent edges may be  intuitively interpreted as being tangent to the discrete spheres centered at the identity.
	
	\begin{figure}[h]
	\begin{center}
		\includegraphics[scale=0.07]{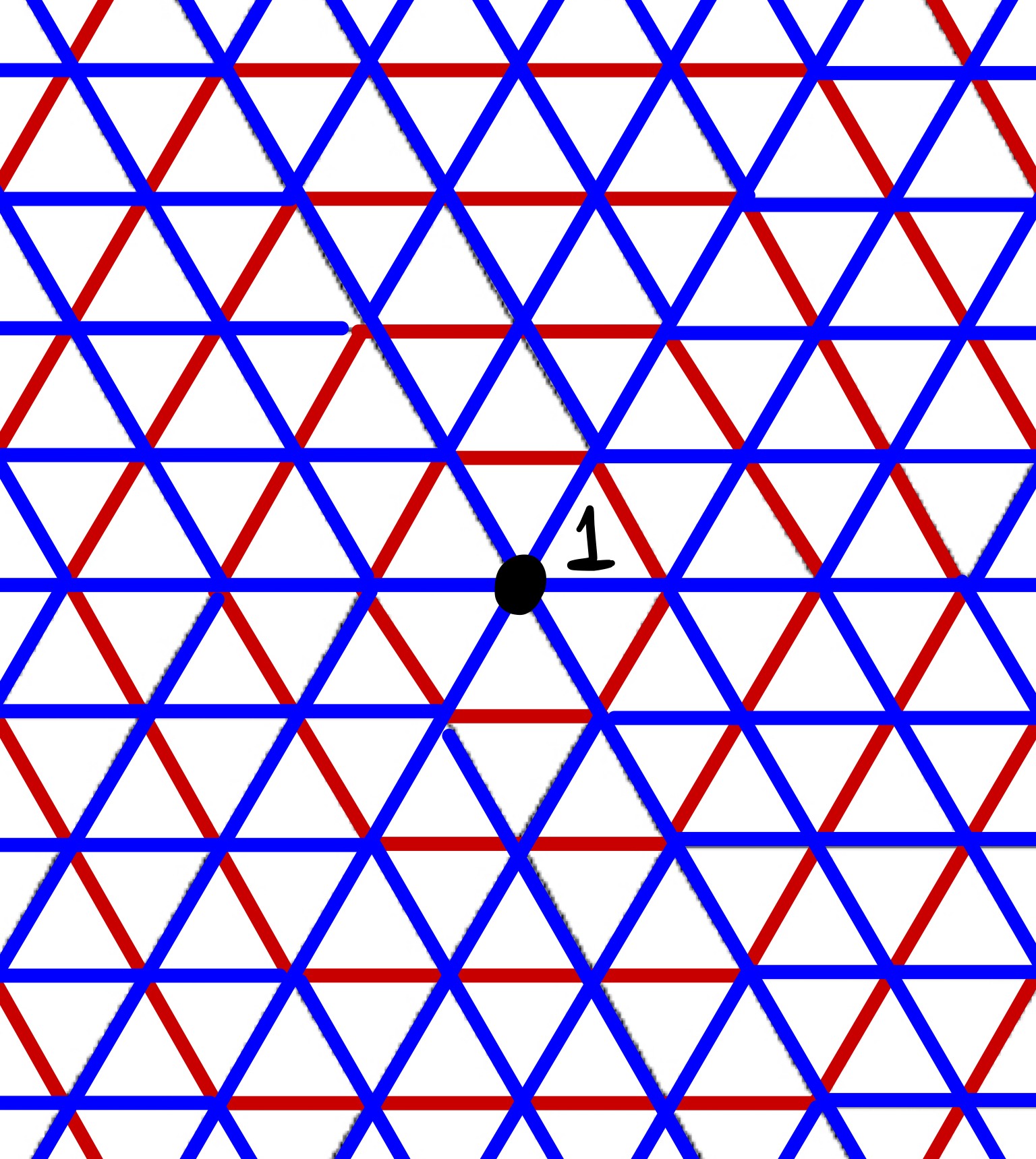}
		\includegraphics[scale=0.07]{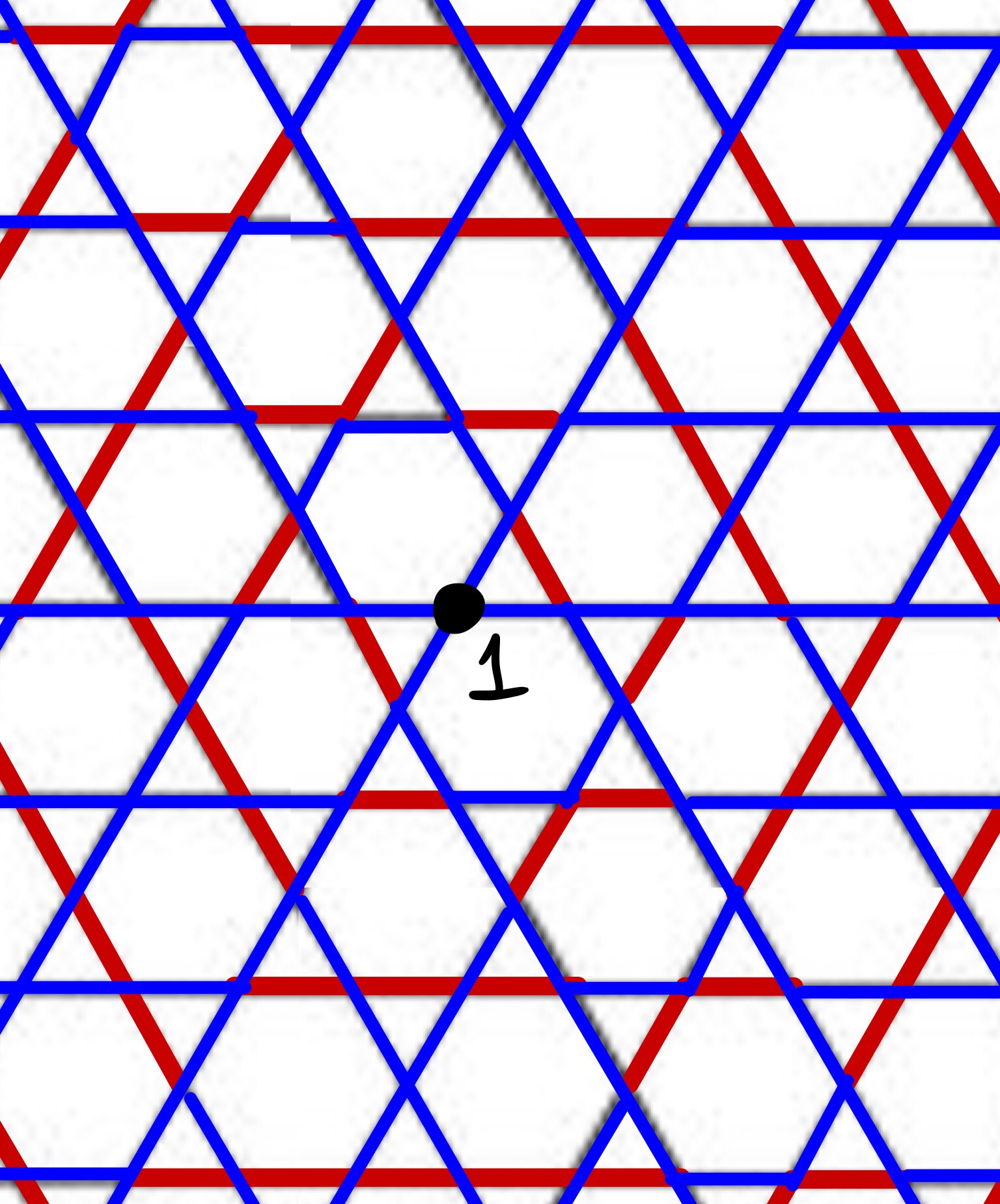}
		\includegraphics[scale=0.10]{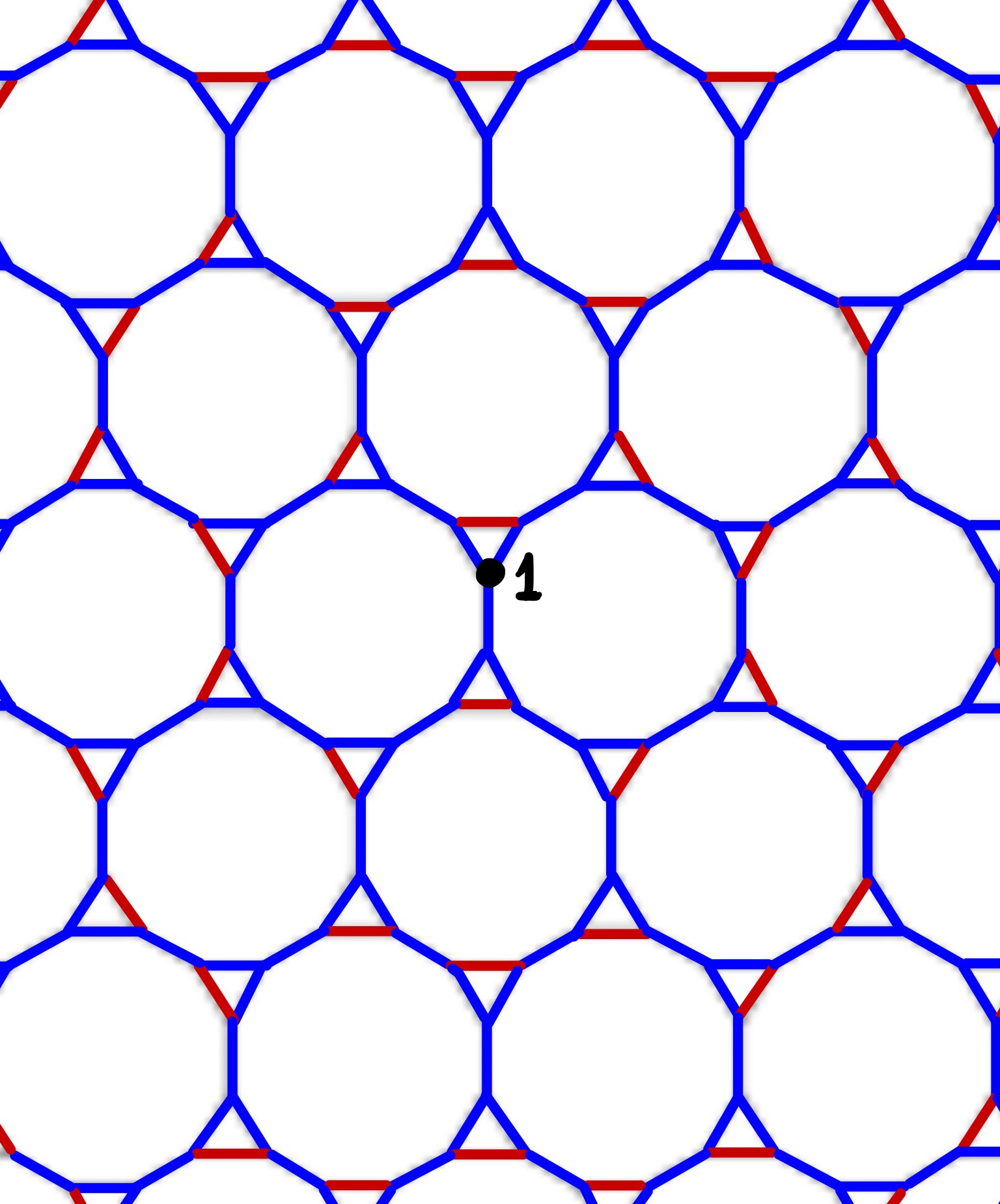}
		\caption{Examples of normal(blue) and tangent(red) edges.}
		\label{fig:3.4.1}
	\end{center}
	\end{figure}	 
	
	We next prove three basic lemmas that will be needed in subsequent sections.
	
	\begin{lemma}
	A path in $\Gamma$ from $1$ to $g \in G$ is a geodesic if and only if every edge on the path is norm increasing.	
	\end{lemma}

	\begin{proof}
	$\rightarrow:$
	If a path $p$ from $1$ to $g$ has a tangent or norm decreasing edge $(h,k)$ (i.e. the path goes first to h and then to k, and $|k|<|h|+1$), then we may take a geodesic from $1$ to $k$ which will have $|k|$ many edges and combine it with the sub-path of $p$ from $k$ to $g$. The result is a path from $1$ to $g$ having one or two less edges, and since it is shorter in length that $p$, $p$ cannot be a geodesic.
	\\ $\leftarrow$ Let $p$ be a path from $1$ to $g$ which goes through the points $1=x_0, x_1,x_2, ... x_n = g$ where $(x_i,x_{i+1}) \in A^+_{x_i}$ (i.e. is norm increasing) for every $i=1,2,..n-1$. Then by the definition of $A^+_{x_i}$, we have:
	$$|g|=|x_n|=|x_{n-1}|+1=|x_{n-2}|+1+1=...$$
	$$...=|x_1|+(1+1+...+1)(n-1\;times)=|x_0|+n=n.$$
	Since $|g|=n$, $p$ is a geodesic.
	\end{proof}
	
	\begin{lemma}
	For any $g \in G$, $k \in C(g)$,
	$$ d_{C(g)}(g,k) = |k|-|g| = d_{\Gamma}(g,k),$$
	where $d_{C(g)}$ is the discrete metric of the graph $C(g)$.
	\end{lemma}

	\begin{proof}
	By definition, take a geodesic path $p$ from $1$ to $g$ to $k$, so the sub-path $p'$ of this geodesic path from $g$ to $k$ is also a geodesic and has length $|k|-|g|$. From this, the second equality follows. \\
	For the first equality, since $C(g)$ is a su-graph of $\Gamma$, $|k|-|g| = d_{\Gamma}(g,k) \leq d_{C(g)}(g,k)$. We show that $p'$ lies in $C(g)$ from which we get $d_{C(g)}(g,k) \leq |k|-|g|$ and equality follows. \\
	By the definition of $C(g)$, it suffices to show the vertices of $p'$ are in $C(g)$. Let $x$ be a vertex of $p'$. Then $x$ is a vertex of $p$ so we consider the sub-path of $p$ from $1$ to $x$. $p''$ is then a geodesic path in $\Gamma$ from $1$ to $x$ which passes through $g$. By definition, $x \in C(g)$ and so $p'$ lies in $C(g)$.
	\end{proof}
	
	\begin{lemma}
	If $T:C(g) \rightarrow C(h)$ is a weak cone equivalence,\\ then T maps the tangent edges, norm increasing (oriented) edges and the norm decreasing (oriented) edges of $C(g) \subseteq \Gamma$ \\ to the tangent edges, norm increasing (oriented) edges and the norm decreasing (oriented) edges of $C(h) \subseteq \Gamma$ respectively.
	\end{lemma}

	\begin{proof}
	Let $(k_1,k_2)\in C(g)$ be a tangent edge in $C(g)$, so $|k_1|=|k_2|$. Then of course $(Tk_1,Tk_2)$ is an edge of $C(h)$. Since $T$ is a graph automorphism, the discrete graph metrics of $C(g)$ and $C(g)$ are equivalent with $T$ as an isometry. Using the above remark and Lemma 3.2, we have:
	$$|Tk_1|-|h| = d_{C(h)}(h,Tk_1) = d_{C(g)}(g,k_1) = |k_1|-|g|$$ $$ = |k_2|-|g| = d_{C(g)}(g,k_2) = d_{C(h)}(h,Tk_2) = |Tk_2|-|h|.$$
	Hence $|Tk_1|=|Tk_2|$, i.e. $(Tk_1,Tk_2)$ is tangent. The norm increasing/decreasing cases are identical to this one.
	\end{proof}
	
	As we will see later, in order to write down well defined recursive relations for the spherical growth using cone types, the cone equivalence will at least need to preserve $|A^t_k|$, the number of half edges attached to any element of the cone $k\in C(g)$. Lemma 3.3 tells us that at least $|A^t_k \cap C(g)|$ is preserved.
	
	Cannon's original definition considered the same notion of a cone but restricted cone equivalence only through graph isomorphisms $T:C(g) \rightarrow C(h)$ given by the right multiplication map $T(k):=hg^{-1}k$ (where $k\in G$). In addition, Cannon required that the above map $T:G \rightarrow G$ must send the tangent edges of any $k\in C(g)$ to the tangent edges of $Tk\in C(h)$. Unfortunately, for the wallpaper groups, this definition has two major difficulties. First, in certain complicated graphs, like $3.(12)^2$, it is difficult to distinguish which edges attached to C(g) are tangent, unless these edges are inside C(g) (then one can use Lemma 3.3). The second problem is a bit more practical. If we restrict to graph isomorphisms given by left multiplication, then we end up with a very large number of cone types, and hence our systems of equations becomes too large, tiresome and vulnerable to human error. The above observations force us to consider a more general notion of cones and cone types, strict enough to preserve tangent edges and yield well defined relations, but flexible enough to make calculations feasible.

	\subsection{Extended Graph Cones}
	
	%As discussed in the previous section, we define a different notion of a cone. 
	The \textbf{extended (graph) cone} of $g$, $\bar{C}(g)$ will be another sub-graph of  $\Gamma$, with the same vertices and edges as C(g) along with the additional vertices
	$$\{ k\in G | \exists h\in C(g), (h,k)\in A^t_h \},$$	
	and the additional edges
	$$\{ (k,h)\in \Gamma | h\in C(g), (h,k)\in A^t_h \}.$$	
	
	Figure \ref{fig:3.5.1} shows the extended cones for the same vertices considered in figure \ref{fig:3.3.1}. Observe that the extended cone of the vertex in the graph $4.6.12$ is the same graph as the cone at the same vertex. This occurs because $4.6.12$ is bipartite; Bipartite graphs have no tangent edges and therefore extended cones are the same as cones.
	
	\begin{figure}[h]
	\begin{center}
		\includegraphics[scale=0.12]{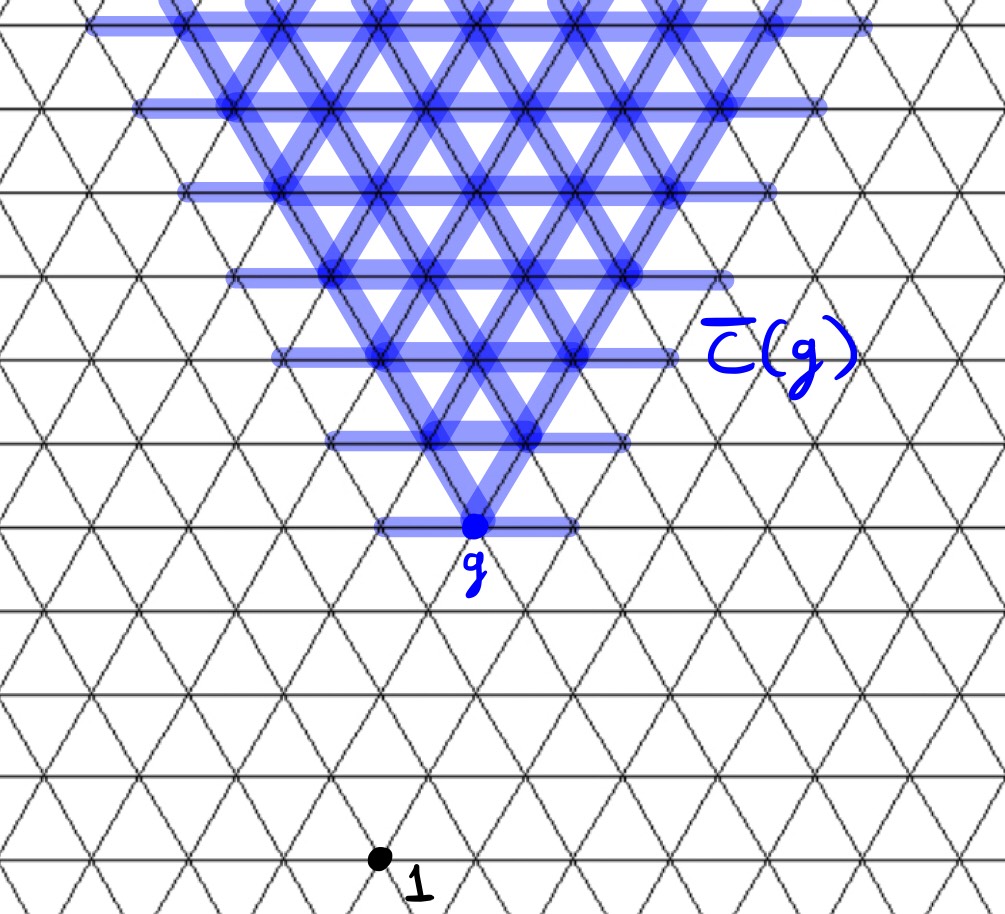}
		\includegraphics[scale=0.12]{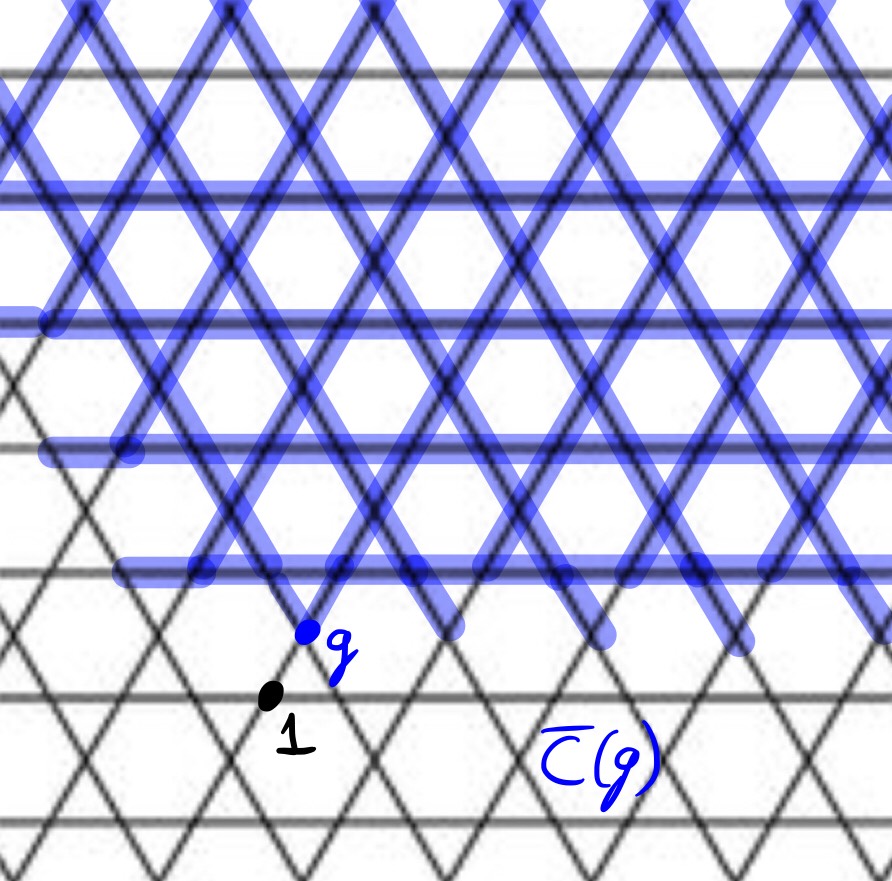}
		\includegraphics[scale=0.12]{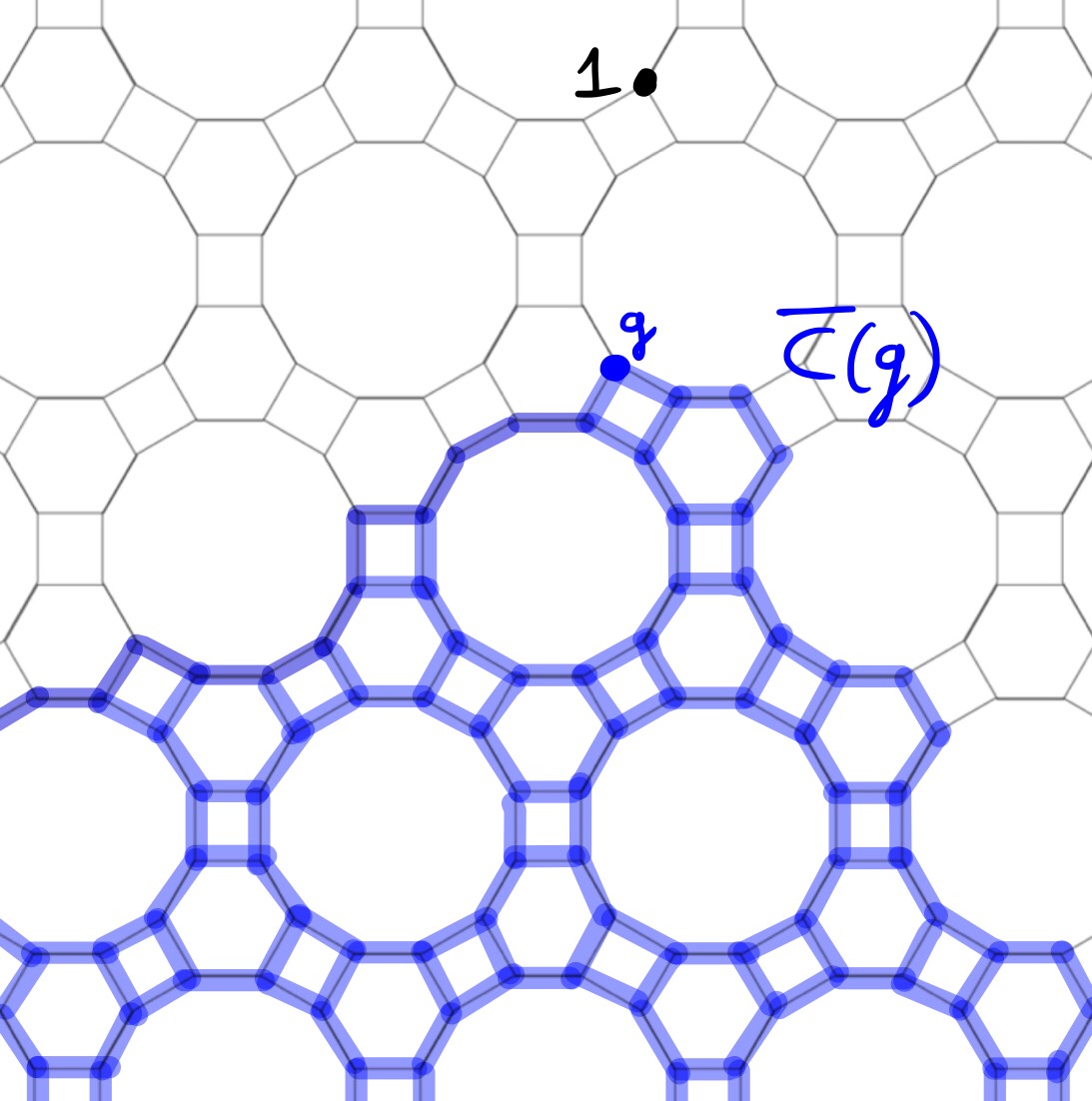}
		\caption{Examples of extended graph cones.}
		\label{fig:3.5.1}
	\end{center}
	\end{figure}
	
	As an amusing analogy, we may view $\bar{C}(g)$ as a cone with "antennas", which extend our cone to capture information about the surroundings. It is 	essentially this information that allows us to obtain a well defined system of recursive relations from which we find the spherical growth.	
	
	Next, we say that two extended cones $\bar{C}(g)$, $\bar{C}(h)$ are \textbf{strongly equivalent} (written as $\bar{C}(g) \sim \bar{C}(h)$) if there exists a graph isomorphism $T:\bar{C}(g) \rightarrow \bar{C}(h)$ such that T(g)=h and T(C(g))=C(h). The equivalence classes that are formed using the above equivalence relation are called \textbf{extended cone types}. It follows from the definition that $\bar{C}(g) \sim \bar{C}(h) \implies C(g) \sim C(h)$. We also have:	
	
	\begin{lemma}
	For any $g,h \in G$,
	$$ h\in C(g) \implies C(h) \subseteq C(g) \;and\; \bar{C}(h) \subseteq \bar{C}(g). $$
	\end{lemma}

	\begin{proof}
	If $k \in C(h)$ then we have two geodesics from $1$ to $g$ to $h$ and from $1$ to $h$ to $k$. We may then combine these two geodesics to get a new one from $1$ to $g$ to $h$ to $k$, implying that $k \in C(g)$. Next, if $l \in \bar{C}(h)\setminus C(h)$, then there is a tangent edge $(l,k)$ with $k \in C(h)$. But then since $k \in C(g)$, this tangent edge together with $l$ is also included in $\bar{C}(g)$.
	\end{proof}	
	
	Finally, it so happens that weak cone equivalence and strong cone equivalence are the same for the 7 Cayley graphs discussed in this paper. As already discussed before, this is also the case for any bipartite Cayley graph. In general, it is not known whether there exists a Cayley graph in which weak cone equivalence and strong cone equivalence are not the same.
	
	\subsection{Cannon's System of Equations}
		
	In this section, we generalize Theorem 7 from Cannon \cite{Cannon:Paper}.
	
	We define numbers $m(\cdot)$ and $n(\cdot ,\cdot)$, show that they are invariant under strong cone equivalence and then use them to deduce Cannon's system of equations for the spherical growth series.
	
	For each $g\in G$ let $m(\bar{C}(g))$ be number of norm decreasing edges of $g$, i.e.: 
	$$m(\bar{C}(g)) := |A^-_g| = |A_g|-|A^t_g|-|A^+_g| = \{(g,h) \in \Gamma : |g|=|h|+1\}.$$
	Next, for each $g,h,\in G$, let $n(\bar{C}(g),\bar{C}(h))$ be the number of edges $(g,h')$ in $C(g)$ with $\bar{C}(h') ~ \bar{C}(h)$, that is:
	$$ n(\bar{C}(g),\bar{C}(h)) := \{(g,h')\in A^+_g| \bar{C}(h) \sim \bar{C}(h') \}.$$
	
	\begin{lemma}
	For any $g_1,g_2,h_1,h_2 \in G$,
	$$ \bar{C}(g_1) \sim \bar{C}(g_2) \;and\; \bar{C}(h_1) \sim \bar{C}(h_2) \implies $$ $$ m(\bar{C}(g_1))=m(\bar{C}(g_2)) \;and \; n(\bar{C}(g_1),\bar{C}(h_1)) = n(\bar{C}(g_2),\bar{C}(h_2)). $$
	\end{lemma}

	\begin{proof}
	Let $T:\bar{C}(g_1) \rightarrow \bar{C}(g_2)$ be the corresponding strong cone equivalence.
	\\For the $m(\cdot)$ values, observe that $A^+_{g_i} \cup A^t_{g_i} = A_{g_i} \cap \bar{C}(g_i)$ ($i=1,2$). Since $T$ is a graph automorphism,
	$$|A^+_{g_1}|+|A^t_{g_1}| = |A_{g_1} \cap \bar{C}(g_1)| = deg_{\bar{C}(g_1)}g_1 $$ $$ =  deg_{\bar{C}(g_2)}g_2 = |A_{g_2} \cap \bar{C}(g_2)| = |A^+_{g_2}|+|A^t_{g_2}|,$$
	where $deg_{\cdot}\cdot$ is the degree of a vertex w.r.t. the indicated graph, and so
	$$ m(\bar{C}(g_1))= |A_{g_1}|-|A^t_{g_1}|-|A^+_{g_1}| =  |A_{g_2}|-|A^t_{g_2}|-|A^+_{g_2}| = m(\bar{C}(g_2)). $$
	For the $n(\cdot,\cdot)$ values, first observe that $n(\bar{C}(g_1),\bar{C}(h_1)) = n(\bar{C}(g_1),\bar{C}(h_2))$ follows easily from the definition of $n(\cdot,\cdot)$and the transitivity of strong cone type equivalence. Hence it remains to show $n(\bar{C}(g_1),\bar{C}(h_2)) = n(\bar{C}(g_2),\bar{C}(h_2))$. 
	\\Let $k\in A^+_{g_1}$ such that $\bar{C}(k) \sim \bar{C}(h_2)$. We show that $\bar{C}(Tk) \sim \bar{C}(k)$ and T restricted to $\bar{C}(k)$ is the map for this strong cone equivalence. Since T is injective, looking at $A^+_{g_1}$, this will imply that $n(\bar{C}(g_1),\bar{C}(h_2)) \leq n(\bar{C}(g_2),\bar{C}(h_2))$.
	\\Obviously T restricted to $\bar{C}(k)$ is a graph automorphism onto its image and sends the root vertex of $\bar{C}(k)$ to the root vertex of $\bar{C}(T(k))$ (also, T is well defined on its whole domain by the Lemma 3.4). What remains to be shown is that $T[C(k)] = C(Tk)$ and $T[\bar{C}(k)] = \bar{C}(Tk)$.
	\\Let $x \in C(k)$. Then there exists a geodesic path from $1$ to $g_1$ to $k$ and a geodesic path from $1$ to $k$ to $x$ which we may combine into a geodesic path from $1$ to $g_1$ to $k$ to $x$. Denote this geodesic by $p$. By Lemma 3.1, $p$ consists of only norm increasing edges. Denote the sub-path of $p$ from $g_1$ to $x$ by $p'$. By Lemma 3.3 $Tp'$,the image of $p'$ under graph automorphism T on $C(g_1)$ also consists of norm increasing edges. Combining $Tp'$ with any geodesic path from $1$ to $g_2$, we get a geodesic path from $1$ to $h$ to $Tk$ to $Tx$, thus $Tx\in C(Tk)$ and $T[C(k)] \subseteq C(Tk)$
	\begin{figure}[h]
		\begin{center}
			\includegraphics[scale=0.20]{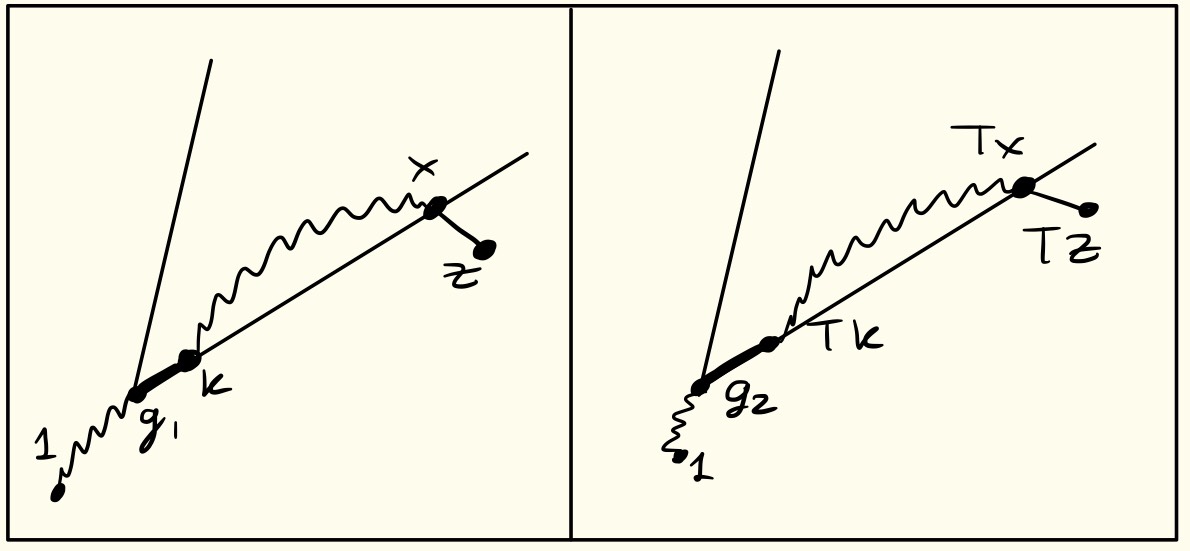}
			\caption{Picture illustrating the proof of Lemma 3.5}
			\label{fig:3.6.1}
		\end{center}
	\end{figure}
	\\
	Since T is bijective, we may apply exactly the same argument to $y\in C(Tk)$ obtaining $T^{-1}y\in C(k)$, and hence we get $T[C(k)] = C(Tk)$\\
	Next let $(x,z)\in \bar{C}(k)$ be a tangent edge with $x\in C(k)$ and $z \notin C(k)$. Then there are two cases, either $(x,z)\in C(g_1)$ or $(x,z)\notin C(g_1)$. If $(x,z)\in C(g_1)$, then by Lemma 3, $(Tx,Tz)$ is a tangent edge of $C(g_2)$. Since $Tx\in C(Tk)$, this means that $(Tx,Tz)\in \bar{C}(Tk)$ and $Tz \in \bar{C}(Tk)$. On the other hand, if $(x,z)\notin C(g_1)$ then  $(Tx,Tz)\notin C(g_2)$ (since T maps $C(g_1)$ to $C(g_2)$) and so $(Tx,Tz)$ must be tangent (by the definition of strong cone equivalence). Once again, since $Tx\in C(Tk)$, this means that $(Tx,Tz)\in \bar{C}(Tk)$ and $Tz \in \bar{C}(Tk)$. We get that $T[\bar{C}(k)] \subseteq \bar{C}(Tk)$. Applying the exact same argument with $T^{-1}$, we get equality.
	Finally, as noted in step 4, we have shown that $n(\bar{C}(g_1),\bar{C}(h_2)) \leq n(\bar{C}(g_2),\bar{C}(h_2))$. Applying the same argument but switching $g_1$ with $g_2$ we get the reverse inequality and thus $n(\bar{C}(g_1),\bar{C}(h_2)) = n(\bar{C}(g_2),\bar{C}(h_2))$ as we wanted.
	\end{proof}	
	
	Suppose that $\Gamma(G,S)$ has \textit{finitely many extended cone types}.
	Denote them by $T=\{t_0, t_1, .. t_q\}$ (where $q\in N$) so that $\bar{C}(1) \in t_0$.
	By the above lemma, we may choose representative extended cones $\bar{C}(1), \; \bar{C}(g_1),...\; \bar{C}(g_q)$ from each extended cone type, and obtain well defined numbers $m_j = m(t_j) := m(\bar{C}(g_j))$ and $n_{i,j} = n(t_i, t_j) := n(\bar{C}(g_i),\bar{C}(g_j))$ for all $t_i, t_j \in T$.
	
	\begin{theorem}
	Let G be a group with finite generating set S such that $\Gamma(G,S)$ has
	finitely many extended cone types. Then the growth series $\Delta(z)$ is rational and
	its presentation in the form of a ratio $\dfrac{P(z)}{Q(z)}$ of two polynomials
	can be computed using the data $((t_j)_j, (m_j)_j, (n_{i,j})_{i,j})$ about the set of extended cone types.
	In fact, $\Delta(z) = \sum_{j=0}^q f_j(z)$, where $(f_j (z))_{j=0}^q$ are obtained by solving the linear system of equations:
	$$ f_0(z) = 1, \;\;\; f_j(z) =  \dfrac{z}{m_j} \sum_{i=0}^q n_{i,j} f_i(z) \;\; \forall j \neq 0.$$
	\end{theorem}
	
	\begin{proof}
	Let $a_{n,i}$ denote the number of elements $g \in G$ s.t. $|g|=n$ and $\bar{C}(g) \in t_i$. 
		Then $$ a_{n,0}= \begin{array}{cc}
					1 & n=0 \\
					0 & n>0\\
					\end{array} \;\;\;\; and \;\;\; a_{0,j} = 0 \;\; for\; j \neq 0,$$
		along with the following recursive equations:
		$$ a_{n,j} = \dfrac{1}{m_j} \sum_{i=0}^q n_{i,j} a_{n-1,i} \;\;\; for \; n>0 \; j \neq 0.$$
		which follow easily from the definitions of $m_j$ and $n_{i,j}$ by the famous "double-counting" technique from enumerative combinatorics:
		\\We count $|\{(v,w)\in \Gamma : |v|+1=|w|=n,\; \bar{C}(w) \in t_j\}|$ in two different ways. First, for each vertex $v$ of length $n-1$, we count the number of vertices $w$ with length $n$ which are adjacent to $v$ and have cone type $t_j$. If $v$ has cone type $t_i$, then for each $v$ we count exactly $n_{i,j}$-many such vertices $w$. We have exactly $a_{n-1,i}$ such vertices $v$ with cone type $t_i$, therefore our count yields $\sum_{i=0}^q n_{i,j}a_{n-1,i}$. On the other hand, we get the exact same count if we consider all vertices $w$ with length $n$ which have cone type $t_j$ and count the number of vertices $v$ of length $n-1$ which are adjacent to $w$. By definition, the number of such vertices $v$ is $m_j$ and the number of vertices $w$ with length $n$ which have cone type $t_j$ is $a_{n,j}$, so our second count yields $m_j a_{n,j}$. Equating our two counts gives us the recursive relation.
		\\Next define $\forall j = 0 ... q$
		$$ f_j(z) = \sum_{n=0}^\infty a_{n,j}z^n.$$
		Then the previous relations become:
		$$ f_0(z) = 1, \;\;\; f_j(z) =  \dfrac{z}{m_j} \sum_{i=0}^q n_{i,j} f_i(z) \;\; j \neq 0.$$
		Therefore, we can use the data $(T, (m_j)_j, (n_{i,j})_{i,j})$ to compute the functions $(f_j)_j$ as solutions to the linear system of equations above. Every group element belongs to exactly one extended cone type and so
		$$\delta(n)=\sum_{j=o}^q a_{n,j}\;\;\implies \;\; \Delta(z) = \sum_{j=0}^q f_j(z),$$
		from which we obtain the spherical growth series $\Delta(z).$ 
		Finally, since the coefficients in the linear system of the $f_j$s only involve rational functions,
		the $f_j$s will be rational and hence our growth series $\Delta(z)$ will also be rational.
	\end{proof}	

	Finally, we may wish to visually display the data $((t_j)_j, (m_j)_j, (n_{i,j})_{i,j})$ in a directed graph which we shall call the \textbf{extended cone type diagram}. Its vertices will be the extended cone types, and number of directed edges from an extended cone type $t_i$ to another $t_j$ will be the number $n(t_i,t_j)$. Furthermore, we may mark next to an extended cone type its associated $m(\cdot)$ value, omitting writing it when the value is 1. For each of the 7 cases that we consider in section 4, we write down the associated extended cone type diagram.

	\subsection{Equations for the Geodesic Growth}

	In this section we present a method for finding the geodesic growth of a group using the notions of (graph) cones and weak cone equivalence. The method is similar to Cannon's equations, however, its origins are from the DSV method for computing the growth of context free languages (named after Delest, Schuetzenberger,and Viennot) \cite{ChomskySchutzenbergerDSVMethod}. This section will not be needed in the subsequent sections.
	
	As we have already seen in Section 3.1, to every geodesic word $w \in Geo(G,S)$ corresponds a geodesic path from $1$ to $\bar{w}$, and thus the geodesic growth of a group can be expressed as the number of geodesics starting from $1$ of a given length. This remark motivates the following:
	
	\begin{definition}
	For $g\in G$, the \textbf{geodesic growth function} of a graph cone $C(g)$ is
	$$ l_{C(g)}(n) = |\{w \in Cone(g): |w|=n\}|, $$
	i.e. the number of geodesics in $C(g)$ starting at $g$ with length $n$.
	\end{definition}
	
	It easily follows that if two cones $C(g_1)$ and $C(g_2)$ are weakly equivalent, then $l_{C(g_1)}(n) = l_{C(g_2)}(n)$.
	
	Suppose for the rest of this section that $\Gamma$ has finitely many cone types $\{t_0, t_1, ... t_q\}$. Via the previous remark, we can associate to each cone type $t_j$ its \textbf{geodesic growth function} $l_j(n)$.
	We can then also define the corresponding \textbf{geodesic growth series} $L_j(z) = \sum_{n=0}^{\infty} l_j(n)z^n$.
	
	Next we consider the analogue of the n-values in section 3.7, this time using cones instead of extended cones. For each $g,h \in G$, let $k(C(g),C(h))$ be the number of vertices in $A_+^g$ whose cones are weakly equivalent to $C(h)$.
	
	In a similar manner as in Lemma 3.5, it follows that 
	$$ C(g_1) \sim C(g_2) \;and\; C(h_1) \sim C(h_2) \implies $$ $$k(C(g_1),C(h_1)) = k(C(g_2),C(h_2)). $$	
	This property allows us to obtain well defined numbers $(k_{i,j})_{i,j}$ which can be used to compute the geodesic growth of a group with finitely many cone types:
	
	\begin{theorem}
	Let G be a group with finite generating set S such that $\Gamma(G,S)$ has
	finitely many cone types. Then the geodesic growth series $L(z)$ is rational and
	its presentation in the form of a ratio $\dfrac{P(z)}{Q(z)}$ of two polynomials
	can be computed using the data $((t_j)_j, (k_{i,j})_{i,j})$ about the set of cone types.
	\end{theorem}
	
	\begin{proof}
	The following recursive relation follows easily:
	$$l_i(n) = \sum_{j=0}^q k_{i,j}\;l(n-1).$$
	Looking at the corresponding growth series, this relation becomes:
	$$L_i(z) = 1 + z \cdot \sum_{j=0}^q k_{i,j}\;L_j(z). $$
	These relations hold for each $i = 0...q$ hence we have a systems of linear equations with 1 variable z which means that the solution $L(z)=L_0(z)$ will be a rational function in z.
	\end{proof}	
	
	\subsection{Describing Cones and Extended Cones}

	So far, we have discussed the use of cones and extended cones in determining the growth of groups.
	However, in order to apply these methods to any Cayley graph, one needs to be able to rigorously find the vertices and edges that comprise a given cone and extended cone. In this section, we provide a general procedure for this task, provided that one can find the cones and extended cones of $S \cup S^{-1}$ and one can translate and intersect subsets of the group easily. For the next proposition, recall that $\forall g \in G$, $A_g^+$ is the set of all $h \in G$ s.t. $(g,h)$ is an order increasing edge.
	
	\begin{proposition}
		For $g \in G$, $s \in S \cup S^{-1}$ and $sg \in A^+_g$,
		$$C(sg)\;=\; C(g) \; \cap \; C(s)\cdot g.$$
	\end{proposition}

	\begin{figure}[h]
		\begin{center}
			\includegraphics[scale=0.26]{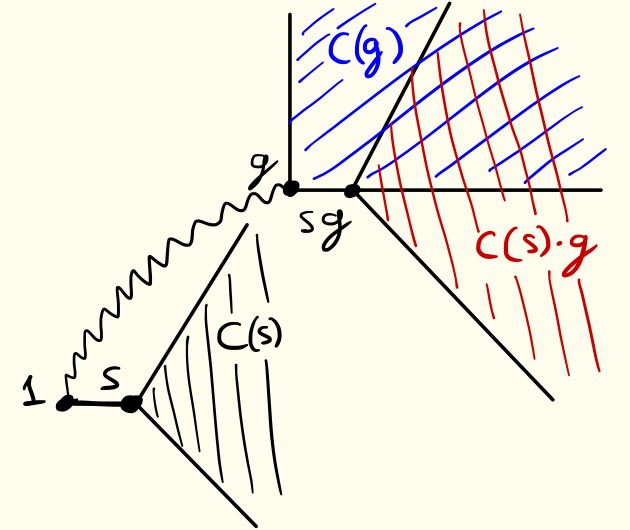}
			\caption{Illustration of Proposition 3.8}
			\label{fig:3.8.1}
		\end{center}
	\end{figure}

	\begin{proof}	 
	$\subseteq$: If $k \in C(sg)$, we get a geodesic from $1$ to $sg$ to $k$. Since $sg \in C(g)$, we also get a geodesic from $1$ to $g$ to $sg$. 
	We may then combine the parts of the above geodesics to construct a new geodesic path from $g$ to $sg$ to $k$. 
	Since right multiplication maps are isometries on $\Gamma$, we may translate the above path to get a geodesic from $gg^{-1} = 1$ to $sgg^{-1} = s$ to $kg^{-1}$, meaning that $kg^{-1} \in C(s)$, i.e. $k\in C(s)$.
	\\
	$\supseteq$: If $k\in C(g)\cap C(s)\cdot g$, we get a geodesic from $1$ to $s$ to $kg^{-1}$ which (via the right multiplication map) translates to a geodesic from $g$ to $sg$ to $k$. We also have a geodesic from $1$ to $g$ to $k$. Combining these two geodesics, we get another geodesic from $1$ to $g$ to $sg$ to $k$, therefore $k \in C(sg)$.
	\end{proof}
	
	\begin{proposition}
			For $g \in G$, $s \in S \cup S^{-1}$ s.t. $sg \in A^+_g$ and for $h \in C(sg)$, $k \not\in C(sg)$, \\
		$(h,k)$ is a tangent edge of $\bar{C}(sg)$
		precisely when one of two cases holds:
		\\ \textit{Case 1}: 
		$(h,k)$ is a tangent edge of $\bar{C}(g)$. 
		with $k\not\in C(g)$
		\\ \textit{Case 2}:
		$(hg^{-1},kg^{-1})$ is a tangent edge of $\bar{C}(s)$
		with $kg^{-1} \not\in C(s)$ but $k\in C(g)$.
	\end{proposition}

	\begin{figure}[h]
		\begin{center}
			\includegraphics[scale=0.26]{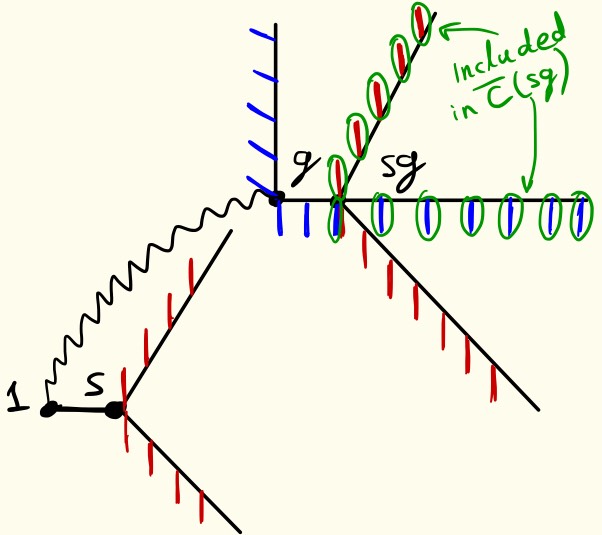}
			\caption{Illustration of Proposition 3.9}
			\label{fig:3.8.2}
		\end{center}
	\end{figure}

	\begin{proof}
	$\leftarrow$:
	If case 1 holds, then tautologically $(h,k)$ is a tangent edge of $\Gamma$ with $h\in C(sg)$ so by definition $(h,k)$ is in $\bar{C}(sg)$.\\
	If case 2 holds, there are two geodesics of equal length, one from $1$ to $s$ to $hg^{-1}$ and another from $1$ to $kg^{-1}$. These two geodesics translate to geodesics from $g$ to $sg$ to $h$ and from $g$ to $k$. \\
	Since $h \in C(sg)$, $h \in C(g)$ and together with the assumption that $k \in C(g)$, we may extend these two geodesics from $1$ to $g$ to $sg$ to $h$ and from $1$ to $g$ to $k$. Via our construction these two geodesics will still have the same length, ergo $(h,k)$ is tangent in $\Gamma$ and the claim follows.\\
	$\rightarrow$:
	We again split into two cases:
	\\Case 1: $k \not\in C(g)$ Then since $k \in C(sg) \subseteq C(g)$, by the definition of an extended cone type, we must have that $(h,k) \in \bar{C}(g)$.
	\\Case 2: $k \in C(g)$ we can a geodesics from $1$ to $g$ to $k$, and from $1$ to $g$ to $sg$ to $h$. Considering the subpaths from $g$ to $k$, and from $g$ to $sg$ to $h$ and translating them we get geodesics from $1$ to $kg^{-1}$ and from $1$ to $s$ to $hg^{-1}$. These paths will have the same length since $(h,k)$ is tangent ergo $(hg^{-1},kg^{-1})$ is a tangent edge in $\bar{C}(s)$. Finally, from Lemma 6, $k \in C(g)$ but $k \not\in C(sg)$ means that $k \not\in C(s) \cdot g$.
	\end{proof}

	The importance of above two propositions is immense. We have essentially reduced the problem of finding the cones and extended cones into:
	\\i) Finding the cones $C(s)$ where $s \in S \cup S^{-1}$.
	\\ii) Finding the edges that are in an extended cone $\bar{C}(s)$ but are not in $C(s)$ for $s \in S \cup S^{-1}$.	
	\\iii) Translating (via the group action of $G$ on $\Gamma$) and intersecting the cones in i) along with the edges in ii)
	\\In order to achieve these three goals, we really need to use specific properties of our 7 Cayley graphs, one key property being their embeddings into $\mathbb{R}^2$. Another proposition follows:
	
	\begin{proposition}
		Suppose $i:\Gamma \to \mathbb{R}^2$ is a graph embedding, $g \in \Gamma$, and there are sub-graphs $S_{in} \subset C(g)$, $S_{out} \subset \Gamma \setminus C(g)$ s.t.
		\\a) $i[S_{out}]$ splits $\mathbb{R}^2$ into 2 disconnected domains $D_1$ and $D'$. 
		\\b) $i[S_{in}]\subset D'$ and splits $D'$ into 2 disconnected domains $D_2$ and $D_3$.
		\\c) $1\in i^{-1}[D_1] \cup S_{out}$ but $g\not\in i^{-1}[D_1] \cup S_{out}$ 
		\\ Then $i^{-1}[D_3] \cup S_{in} \subseteq C(g) \subseteq \Gamma \setminus (i^{-1}[D_1] \cup S_{out})$
	\end{proposition}

	\begin{figure}[h]
		\begin{center}
			\includegraphics[scale=0.15]{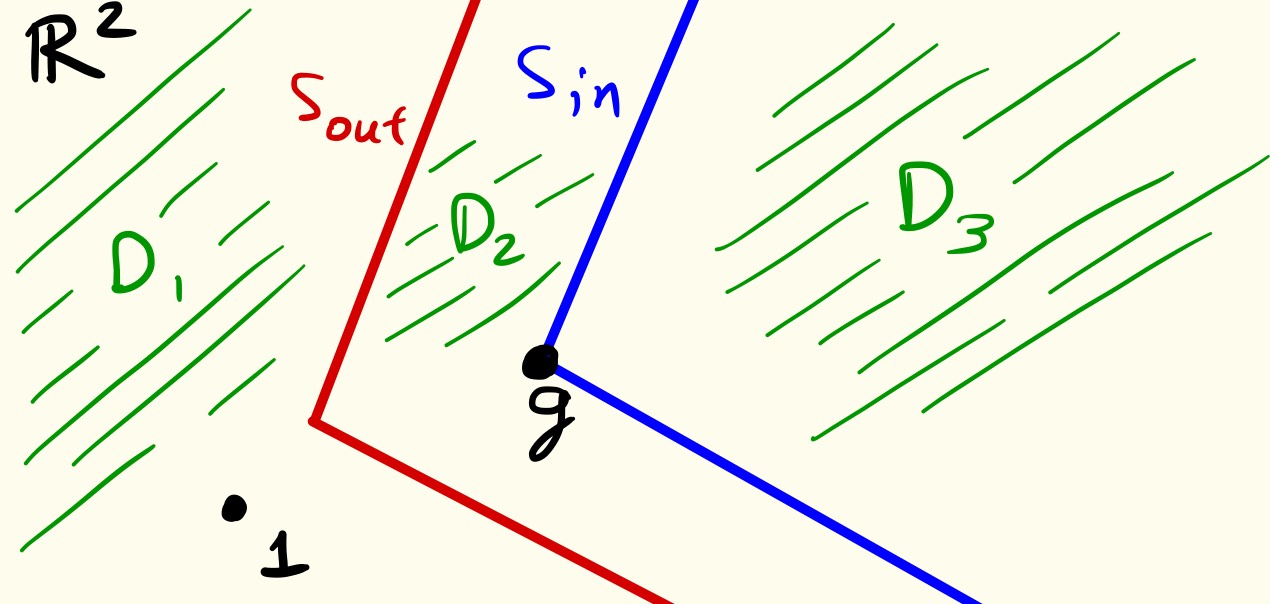}
			\caption{Picture illustrating Proposition 3.10}
			\label{fig:3.8.3}
		\end{center}
	\end{figure}	 
	
	\begin{proof}
	If $k \in i^{-1}[D_1] \cup S_{out}$, then any discrete path from $1$ to $g$ embeds to a continuous path on $\mathbb{R}^2$ which must cross $i[S_{out}]$ and hence the path on $\Gamma$ must pass through a vertex of $S_{out}$. Since $S_{out} \subset \Gamma \setminus C(g)$, we can always replace the subpath from $1$ to $g$ to $S_{out}$ with a strictly shorter one, and thus our initial path cannot be a geodesic.\\
	On the other hand, if $k \in i^{-1}[D_3] \cup S_{in}$, then any geodesic from $1$ to $k$ must pass through a vertex $h \in S_{in}$. But since $S_{in} \subset C(g)$, there is a geodesic from $1$ to $g$ to $h$, which we can combine with the previous geodesic to make a geodesic from $1$ to $g$ to $h$ to $k$, ergo $k \in C(g)$.
	\end{proof}

	The importance of Proposition 3.10 is that it lets us work with the "boundary" of our cones rather than the entire graphs. These "boundaries" will typically consist of 2 rays towards infinity and with a repeating periodic pattern, making them amenable to inductive arguments. Before we move on to applying these propositions to our 7 cases, we will need one more definition: a \textbf{geodesic ray} is an infinite sequence of vertices $\{g_i\}_{i=0}^\infty \subset \Gamma$ s.t. $\forall n \in \mathbb{N}$, $\{g_i\}_{i=0}^n$ is a geodesic from $g_0$ to $g_n$. 
	%Similarly we define geodesic lines $\{g_i\}_{-\infty}^\infty$
	
\section{The Seven Cases}

\subsection{The Triangular Tiling $3^6$}
		
	Recall that there are 17 wallpaper groups listed in Table 1 (3 of them listed twice with 2 different generating sets). Up to non-oriented unlabeled graph isomorphism, these 20 cases yield 7 different Caley graphs.
	We begin the analysis of the cone types with the graph $3^6$, shown in Figure \ref{fig:2.5.1}. For brevity, we will write "ext.cone" instead of "extended cone".
	
	Our first step is to find the cone and extended cone of a vertex $v$ of distance 1 from the identity. However, we first consider the following sequence of lines on $\Gamma$, $\{... l_{-2}, l_{-1}, l_{0}, l_{1}, l_{2} ... \}$ shown in Figure \ref{fig:4.1.1}. 
	Let $1, u \in l_0$, and consider a path from $1$ to $u$.
	\\
	Look at the largest $N$ such that the path crosses $l_{-N} \cup l_N$. 
	When $N>0$ we can replace this path with another one which does not cross $l_{-N} \cup l_N$ and has a shorter length than the original. 
	Thus any path from $1$ to $u$ which does not lie entirely on $l_0$ is not a geodesic.
	From this we can immediately conclude that $l_0 \subset \Gamma\setminus C(v)$. 
	\\
	Similarly, any path from $1$ to $w \in l_1$ which does not lie entirely on $l_0 \cup l_1$ is not a geodesic. 
	Using these two observations, we may compute the norms of vertices on $l_0 \cup l_1$ using induction to conclude that the rays $r_1$ and $r_2$ in Figure \ref{fig:4.1.1} are geodesic rays and that $r_2 \subseteq C(v)$.
	\begin{figure}[h]
		\begin{center}
			\includegraphics[scale=0.14]{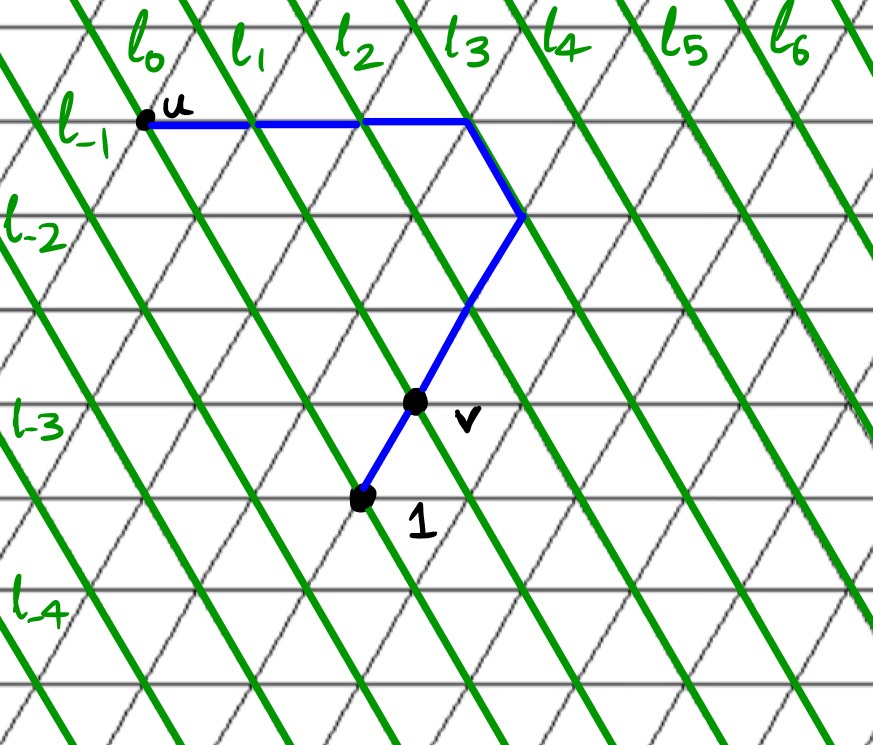}
			\includegraphics[scale=0.14]{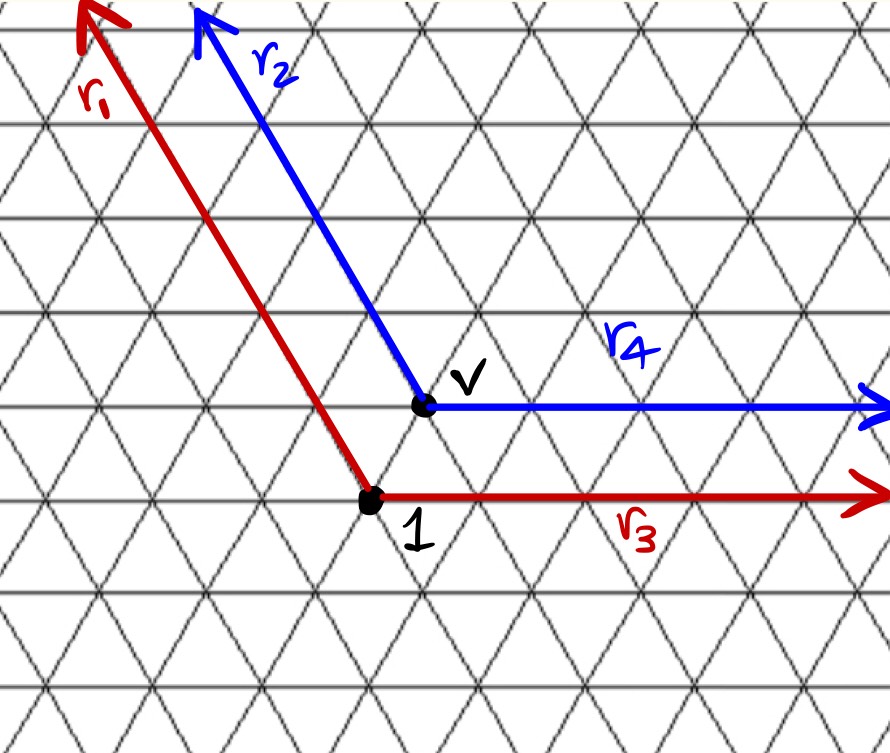}
			\includegraphics[scale=0.15]{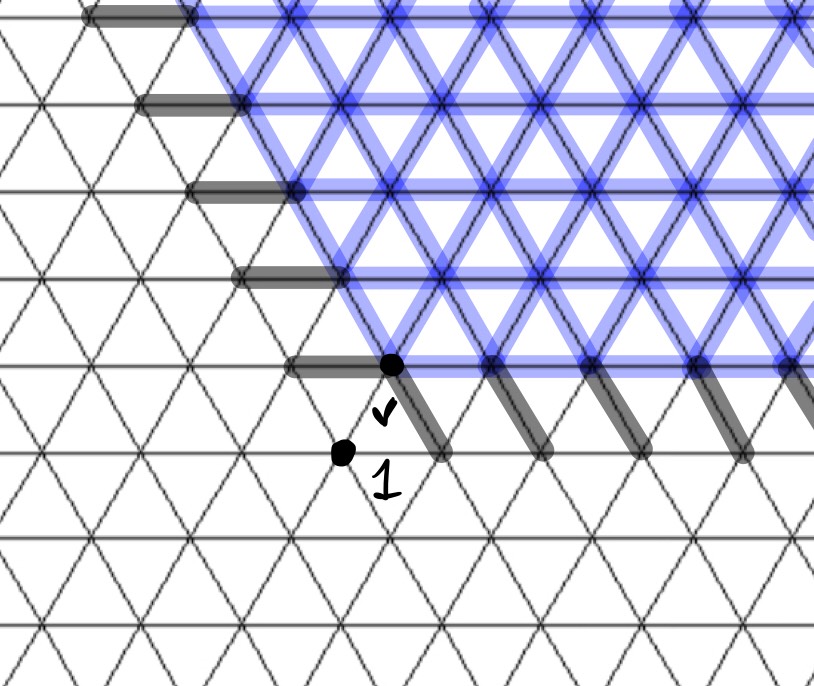}
			\caption{Finding the ext.cone of $v$ with $|v|$}
			\label{fig:4.1.1}
		\end{center}
	\end{figure}
 	Next consider the rays $r_3, r_4$, as shown in Figure \ref{fig:4.1.1}. 
 	From analogous arguments it follows that $r_3 \subset \Gamma \setminus C(v)$ and $r_4\subset C(v)$. 
	Therefore, letting $S_{in} := r_2 \cup r_4$ and $S_{out} := r_1 \cup r_3$, we can apply Proposition 3.10 and get C(v) shown in blue in Figure \ref{fig:4.1.1}.
	\\
	Next, to get the extended cone $\bar{C}(v)$, we need to add new edges and vertices through tangent edges.
	The vertices of $C(v)$ which are connected to vertices outside $C(v)$ are precisely $r_2 \cup r_4$, so we only need to look at these vertices.
	Since we computed the lengths of all the vertices  in $r_1 \cup r_2 \cup r_3 \cup r_4$ we can add any new vertices and tangent edges for our extended cone. 
	These additions are shown in Figure \ref{fig:4.1.1} in black.
	
	Similarly, we find all the 6 extended cones of the \textbf{1-sphere}, the vertices of distance 1 from the root vertex, shown in Figure \ref{fig:4.2.1}.
	\begin{figure}[h]
		\begin{center}
			\includegraphics[scale=0.12]{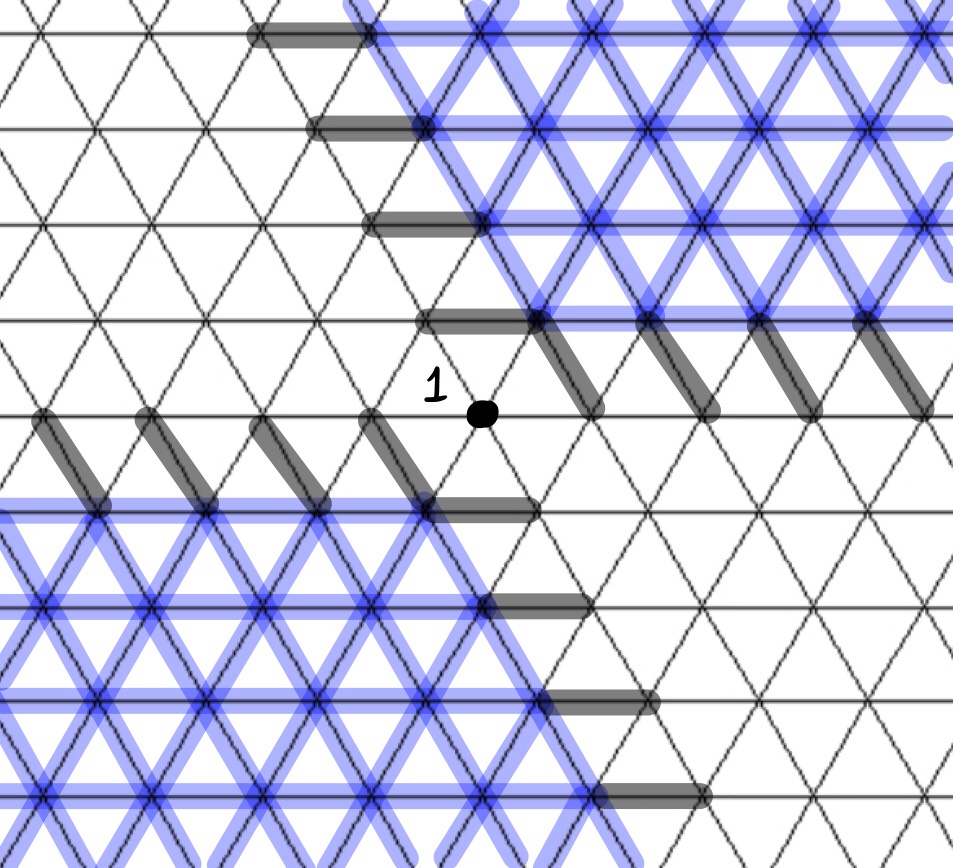}
			\includegraphics[scale=0.12]{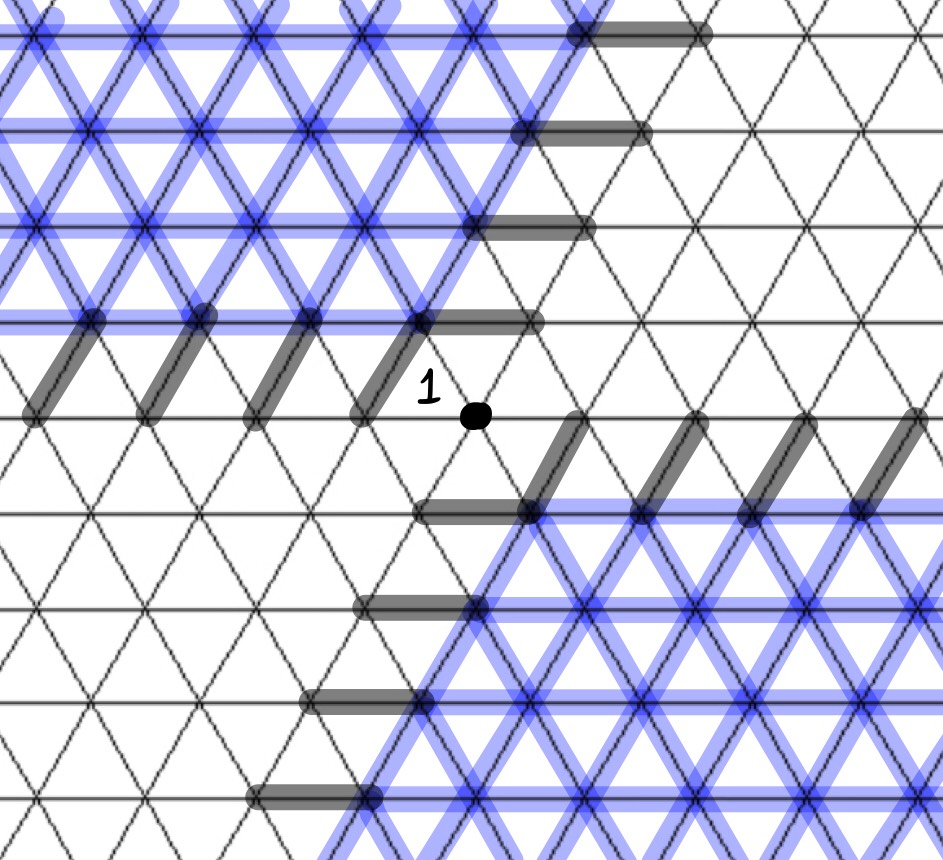}
			\includegraphics[scale=0.11]{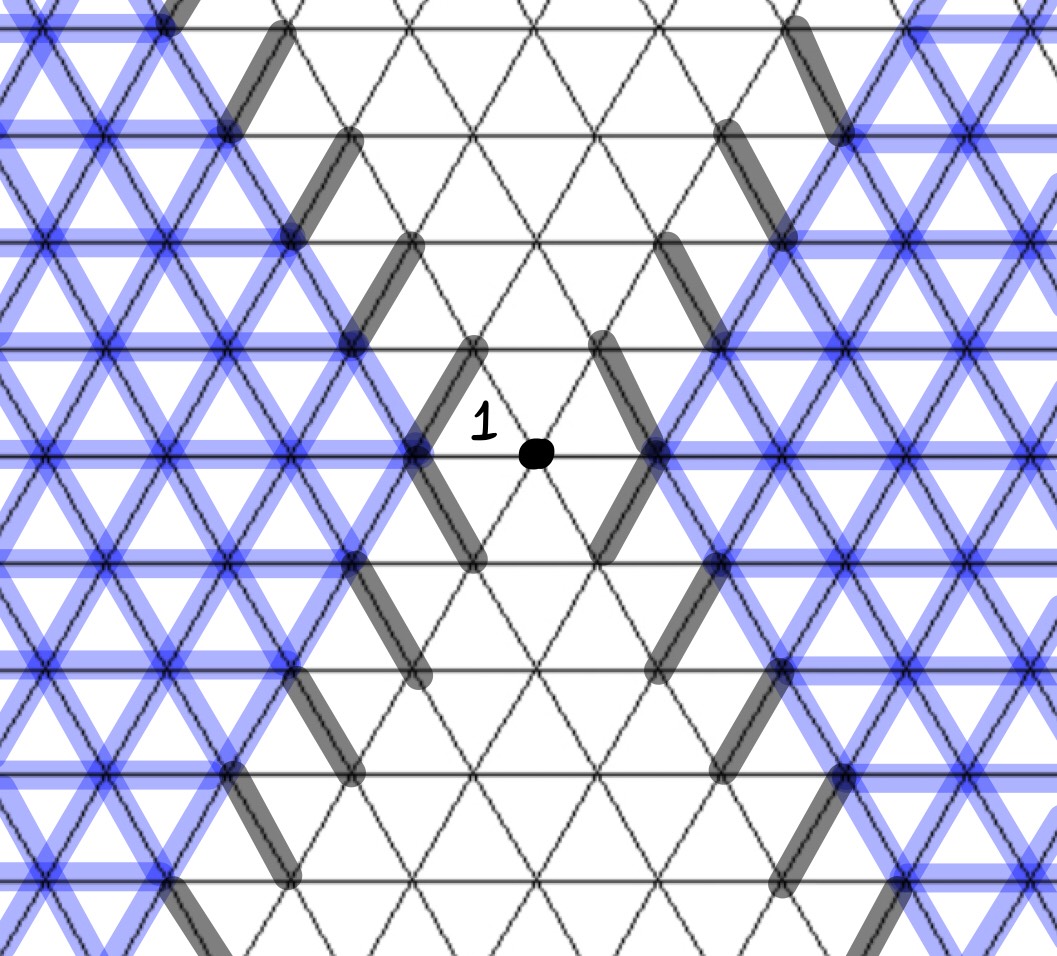}
			\caption{Finding the ext.cones of all vertices in the 1-sphere}
			\label{fig:4.1.2}
		\end{center}
	\end{figure}	 
	It is easy to see that the above extended cones are equivalent. We name their ext.cone type, $A$. Looking at the pictures in Figure \ref{fig:4.1.2}, we see that $m_A = 1$. Pick a representative cone with root vertex $v_1$ as shown in Figure \ref{fig:4.1.3}. 
	We wish to find all the ext.cones of the vertices in $A^+_{v_1}$. Doing so, we will get the numbers $(n_{A,i})_i$ needed for Cannon's equations.
	We start with one such vertex $v_2$. We apply Proposition 3.8 and see that $C(v_2)$ consists of all the vertices in the intersection shown in Figure \ref{fig:4.1.3}.
	
	Next,we compute this intersection in detail. Notice that the following 2 red rays belong to this intersection while the 2 black ones do not (see Figure \ref{fig:4.1.3})
	\begin{figure}[h]
		\begin{center}
			\includegraphics[scale=0.18]{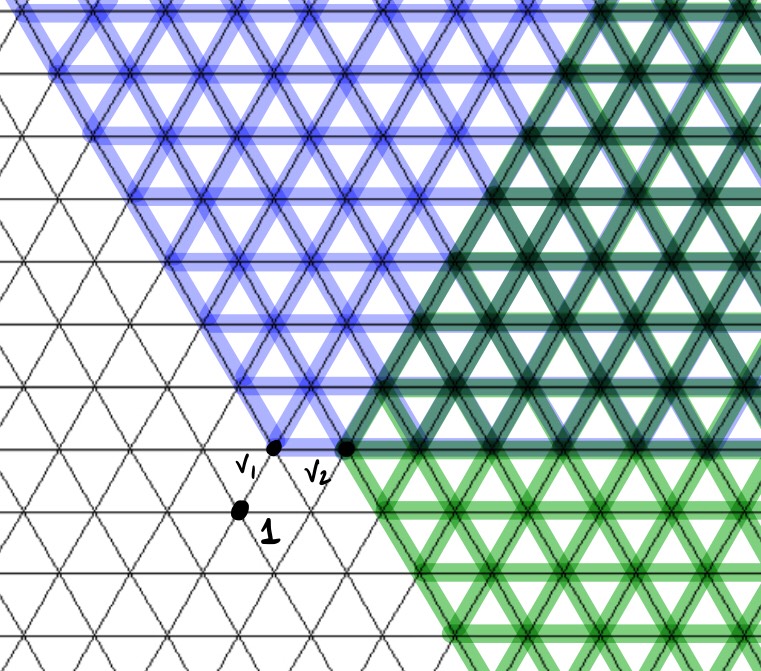}
			\includegraphics[scale=0.18]{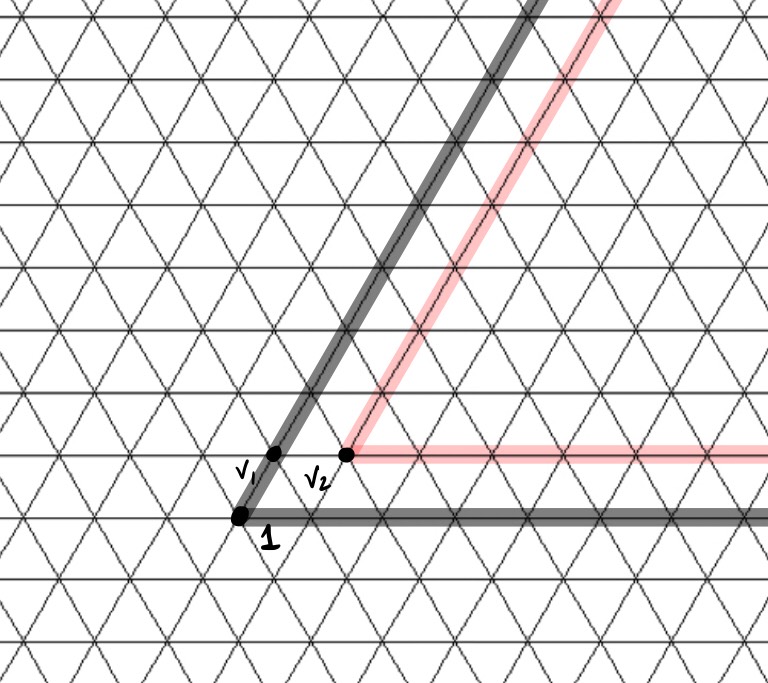}
			\caption{Applying Proposition 3.8}
			\label{fig:4.1.3}
		\end{center}
	\end{figure}		
	Using Proposition 3.10, we can find the cone $C(v_2)$ shown in Figure \ref{fig:4.1.4}. To find $\bar{C}(v_2)$, we use Proposition 3.9. Shown in Figure \ref{fig:4.1.4} are the additional edges (in black) that need to be added to $C(v_2)$ to obtain $\bar{C}(v_2)$. The edges at the bottom side of the cone are added according to case 1 of Proposition 3.9, whereas the ones on the top side of the cone according to case 2.
	\begin{figure}[h]
		\begin{center}
			\includegraphics[scale=0.18]{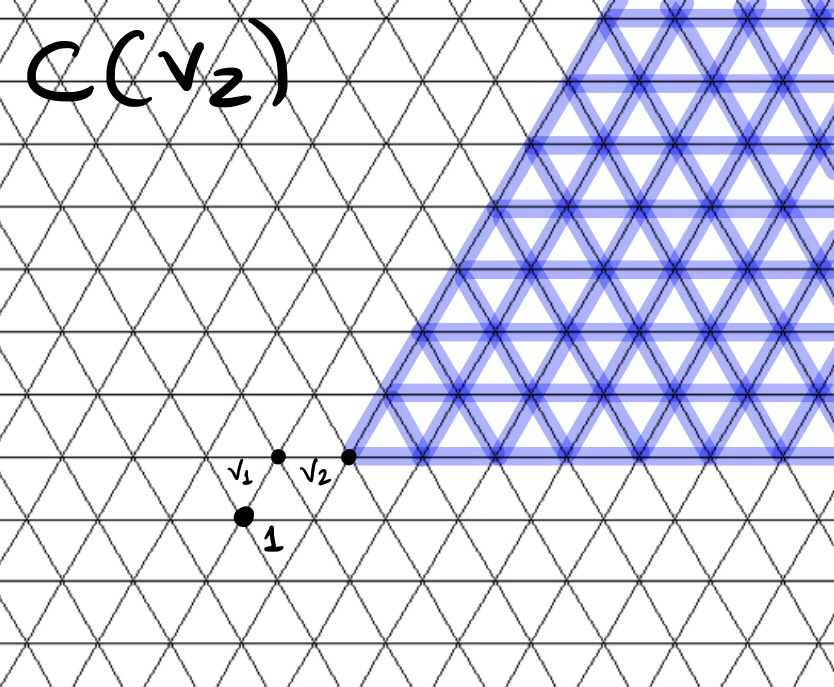}
			\includegraphics[scale=0.18]{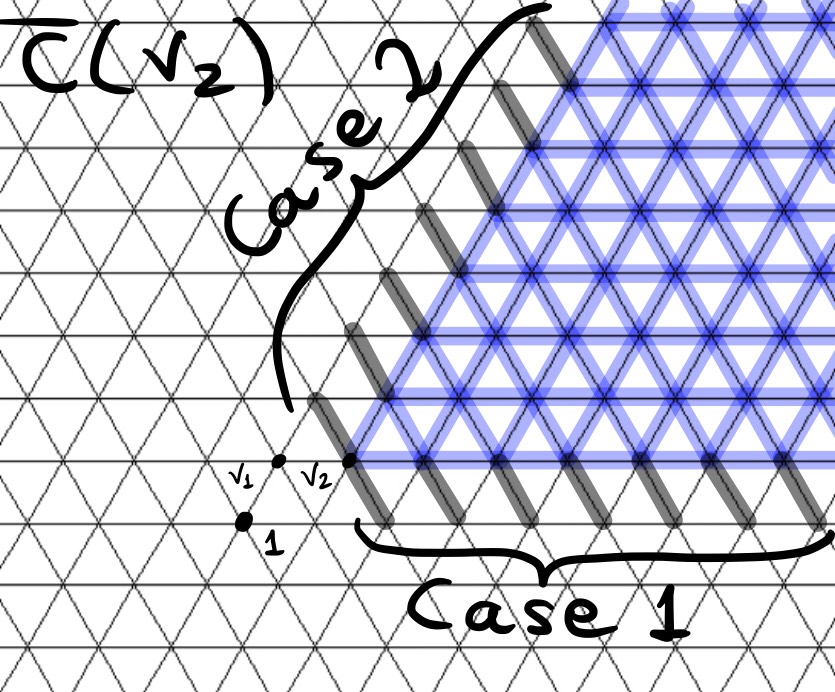}
			\caption{Applying Proposition 3.9}
			\label{fig:4.1.4}
		\end{center}
	\end{figure}	 
	
	Applying the same argument to the other points of $A^+_{v_1}$, we get two more ext.cones (See Figure \ref{fig:4.1.5}).
	\begin{figure}[h]
		\begin{center}
			\includegraphics[scale=0.085]{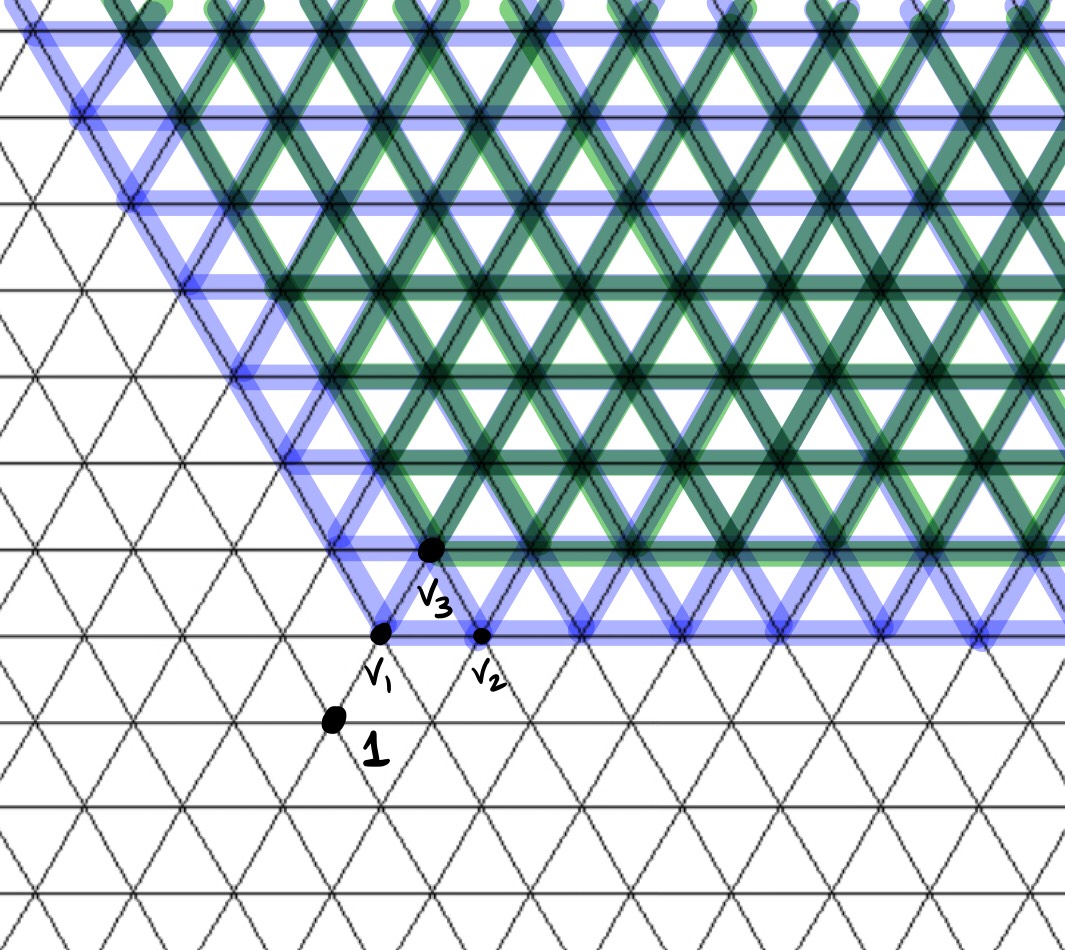}
			\includegraphics[scale=0.085]{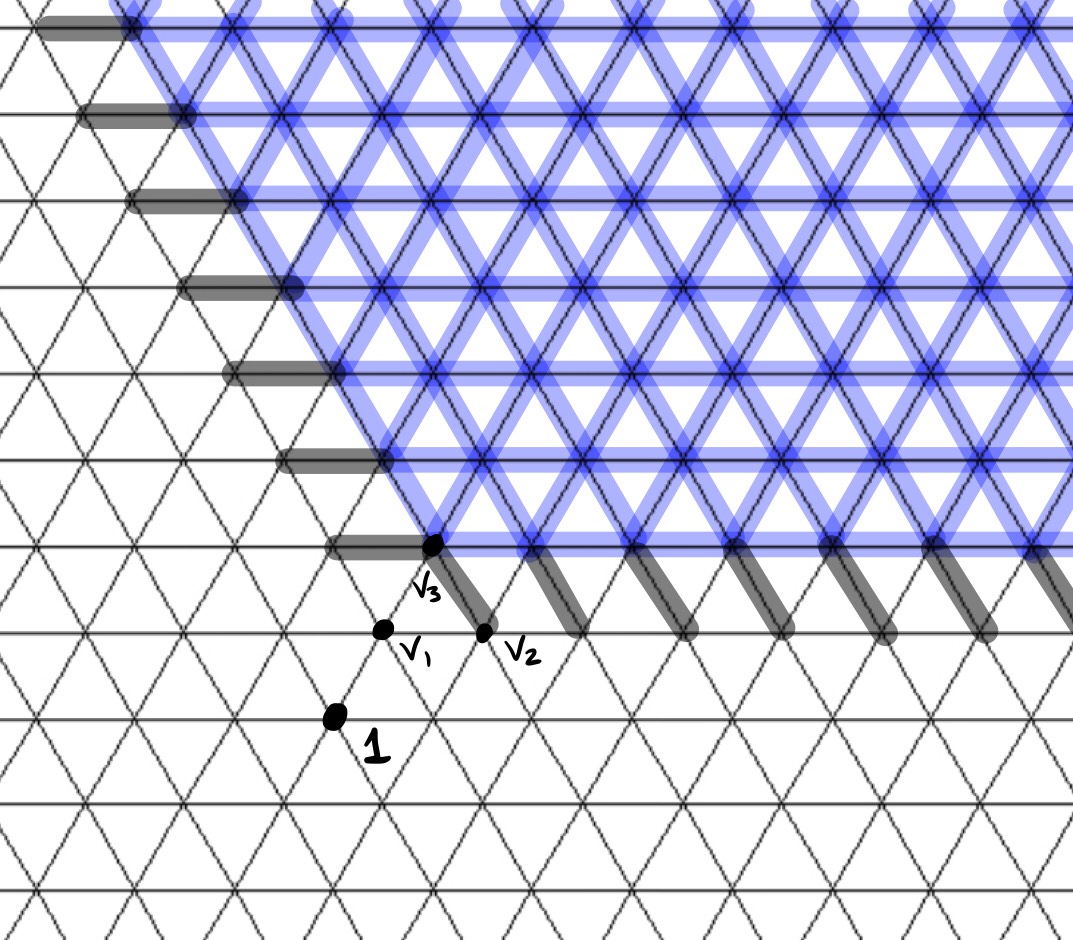}
			\includegraphics[scale=0.12]{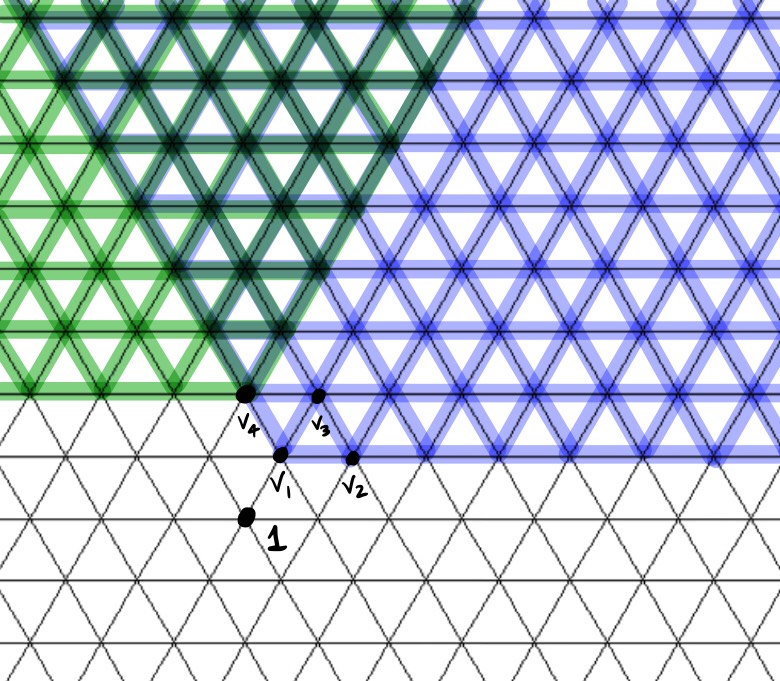}
			\includegraphics[scale=0.085]{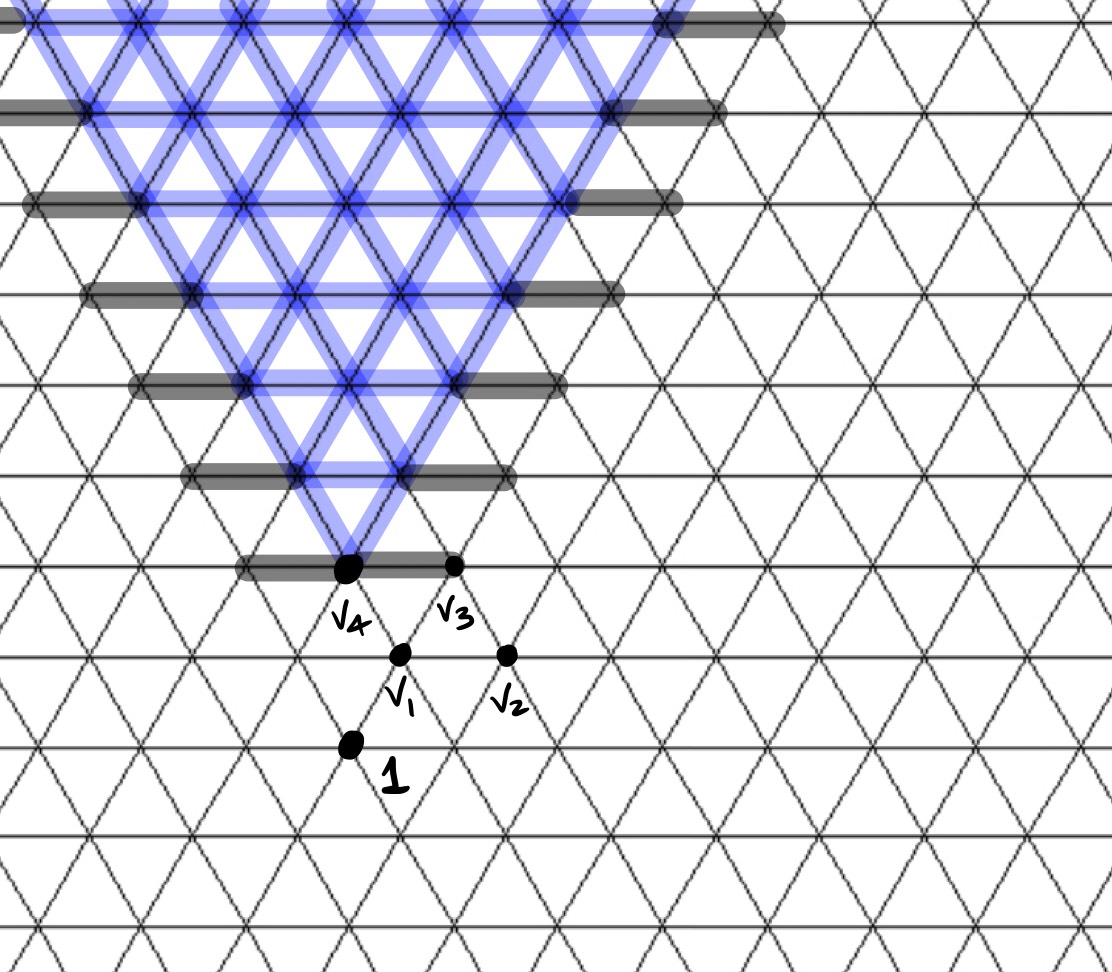}
			\caption{Finding the rest of the ext.cones of all vertices in $A^+_{v_1}$}
			\label{fig:4.1.5}			
		\end{center}
	\end{figure}		
	Observe that $\bar{C}(v_1) \sim \bar{C}(v_3)$ and $\bar{C}(v_2) \sim \bar{C}(v_4)$. Name the ext.cone type of the latter two ext.cones $B$. We then consider $\bar{C}(v_4)$ and find the ext.cones of the vertices in $A^+_{v_4}$, using the exact same approach as above (See Figure Figure \ref{fig:4.1.6}). 
	\begin{figure}[h]
		\begin{center}
			\includegraphics[scale=0.125]{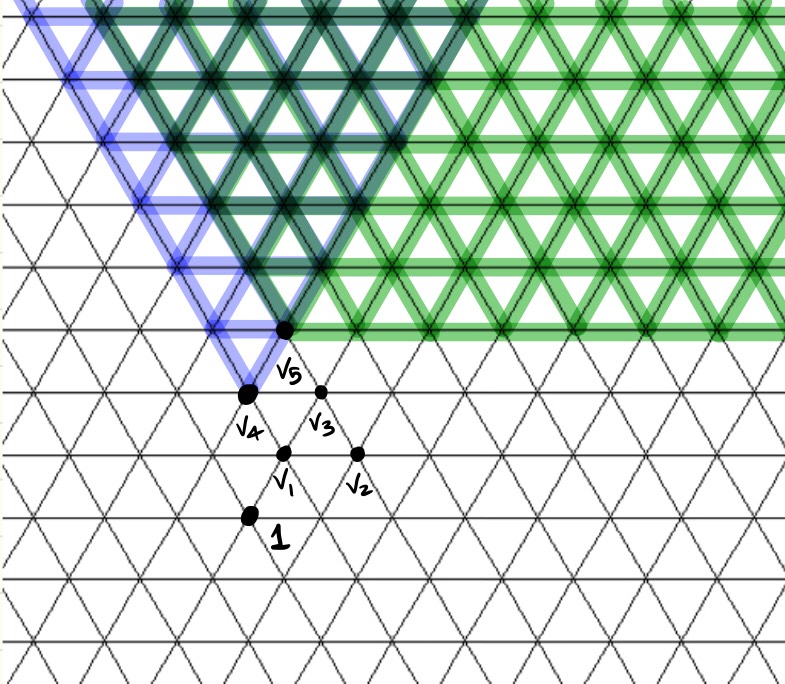}
			\includegraphics[scale=0.125]{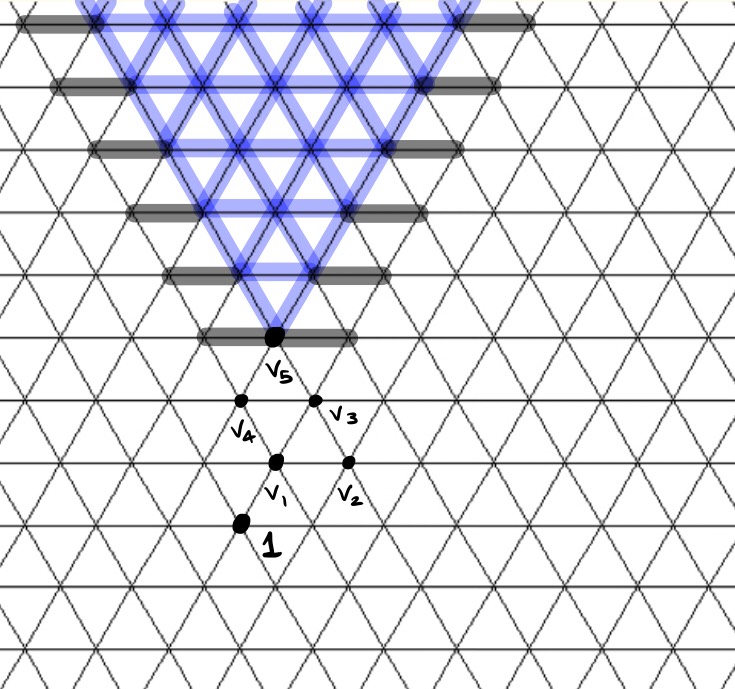}
			\includegraphics[scale=0.125]{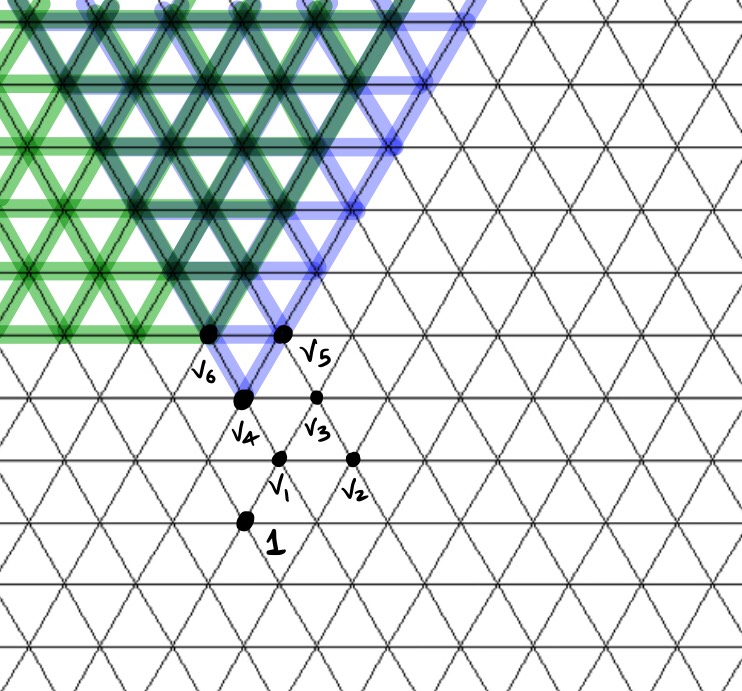}
			\includegraphics[scale=0.125]{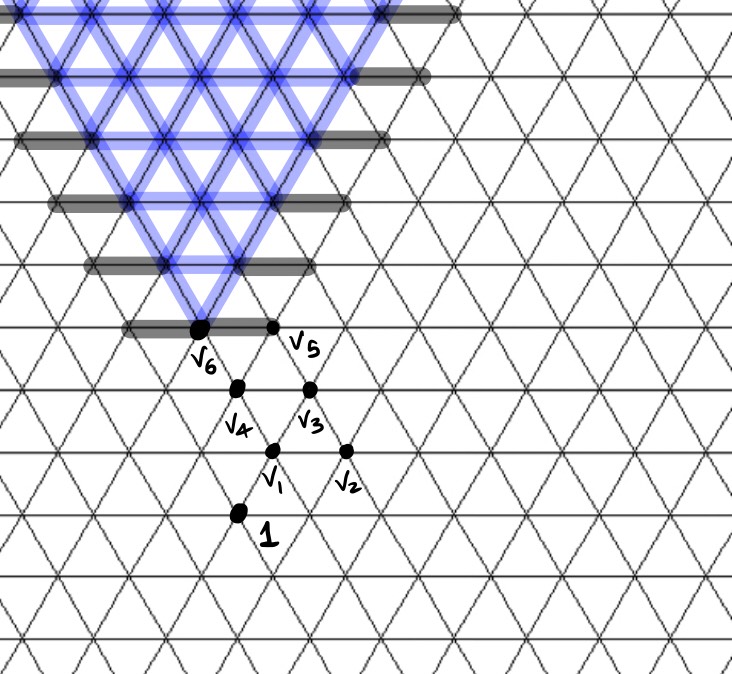}
			\caption{Finding the ext.cones of all vertices in $A^+_{v_4}$}
			\label{fig:4.1.6}
		\end{center}
	\end{figure}
	Observe that $\bar{C}(v_4) \sim \bar{C}(v_5) \sim \bar{C}(v_6)$, so all these points have cone type $B$.
	Since no new cone types emerged, we are done! 
	
	We have found all the extended cone types of $3^6$ and are now ready to get Cannon's equations. 
	Lets call the ext.cone type at the identity $e$.
	The m-values are easy to find: just take an ext.cone of a given type and look at how many edges starting at the root vertex are not tangent nor lead to an element in the cone; this will be the m-value. 
	For the n-values:
	The cone type $e$ appears only at the identity and leads to 6 ext.cones of type $A$. 
	Then, ext.cones of type $A$ always lead to 1 ext.cone of type $A$ and 2 ext.cones of type $B$. Finally, ext.cones of type $B$ leads to 2 ext.cones of type $B$. 
	We obtain the numbers $(m_j)_j$ and $(n_{i,j})_{i,j}$ that were defined at the beginning of section 3.6:
	$$ 
	\begin{matrix}
	m_e = 0	\\
	m_A = 1	\\
	m_B = 2
	\end{matrix}
	\;\;\;\; and \;\;\;\;
	\begin{matrix}
	n_{i,j}	& e	& A	& B	\\
	e		& 0	& 0	& 0	\\
	A		& 6	& 1	& 0	\\
	B		& 0	& 2	& 2	\\
	\end{matrix}
	\;\;.$$
	The above numbers can be summarized in an extended cone type diagram (see Figure \ref{fig:4.1.7}), which was defined at the end of section 3.6.
	\begin{figure}[h]
		\begin{center}
			\includegraphics[scale=0.2]{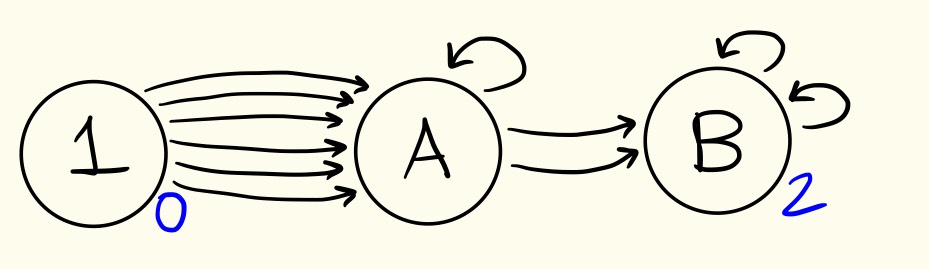}
			\caption{The ext.cone type diagram of $3^6$}
			\label{fig:4.1.7}
		\end{center}
	\end{figure}			
	Using the above numbers, we set up and solve Cannon's Equations using Theorem 1:
	$$ 
	\begin{pmatrix}
	1	& 0	 & 	0	\\
	-6z	& 1-z& z-1	\\
	0	& -2z& 2-2z
	\end{pmatrix}
	\cdot
	\begin{pmatrix}
	f_e\\
	f_A\\
	f_B
	\end{pmatrix}
	=
	\begin{pmatrix}
	1\\
	0\\
	0
	\end{pmatrix}
	.$$
	and obtain the spherical growth series:
	$$\Delta(z) = f_e(z) + f_A(z) + f_B(z) = \dfrac{z^2+4z+1}{(1-z)^2}.$$

	\subsection{The Square Tiling $4^4$}
	
	This case is even simpler than the $3^3$ graph. Since $4^4$ is bipartite, there are no tangent edges, thus cones and extended cone are the same, making our calculations easier (see remark at the beginning of section 3.4).
	\\
	Similar to the $3^6$ graph, we can easily find the cone of a vertex $v_1$ on the 1-sphere at the identity using geodesic rays $r_1$ and $r_2$ and applying the same arguments (see Figure \ref{fig:4.2.1}).
	\begin{figure}[h]
		\begin{center}
			\includegraphics[scale=0.18]{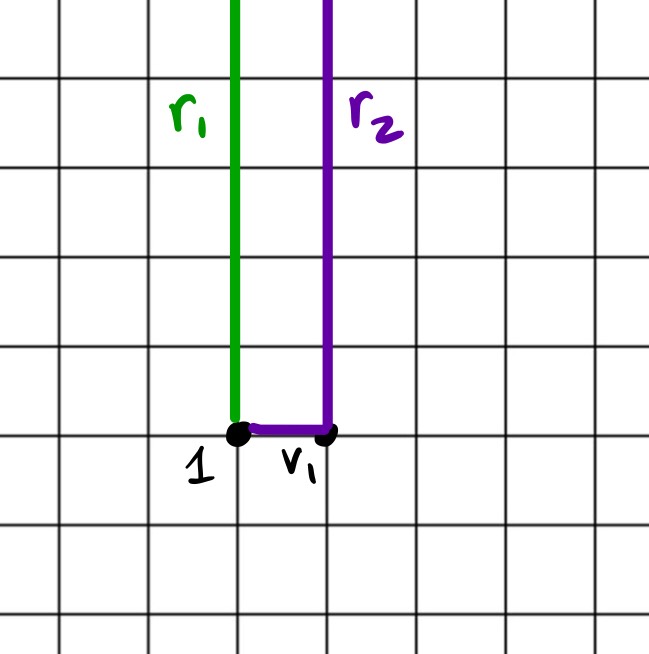}
			\includegraphics[scale=0.18]{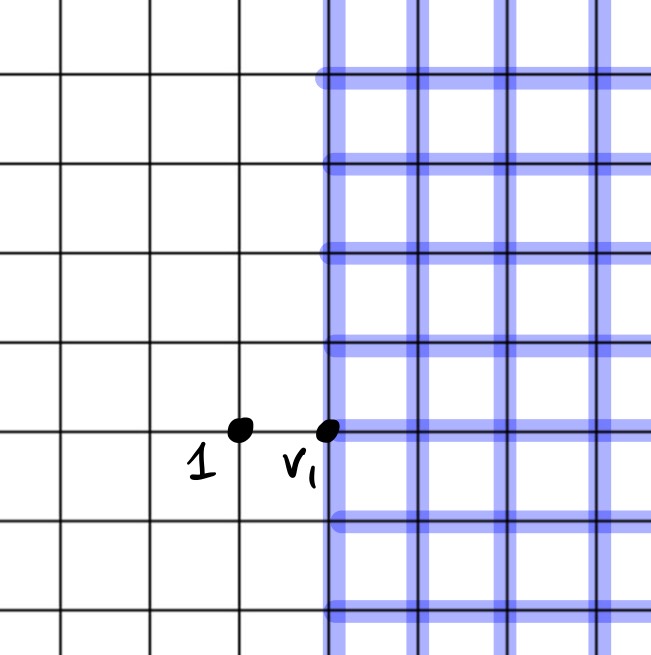}
			\caption{Getting an ext.cone at the 1-sphere}
			\label{fig:4.2.1}
		\end{center}
	\end{figure}
	Using Proposition 3.8, through translating and intersecting the cones at the 1-sphere, we can find the cone at the vertex $v_2 \in A^+_{v_1}$ (see Figure \ref{fig:4.2.2})
	\begin{figure}[h]
		\begin{center}
			\includegraphics[scale=0.18]{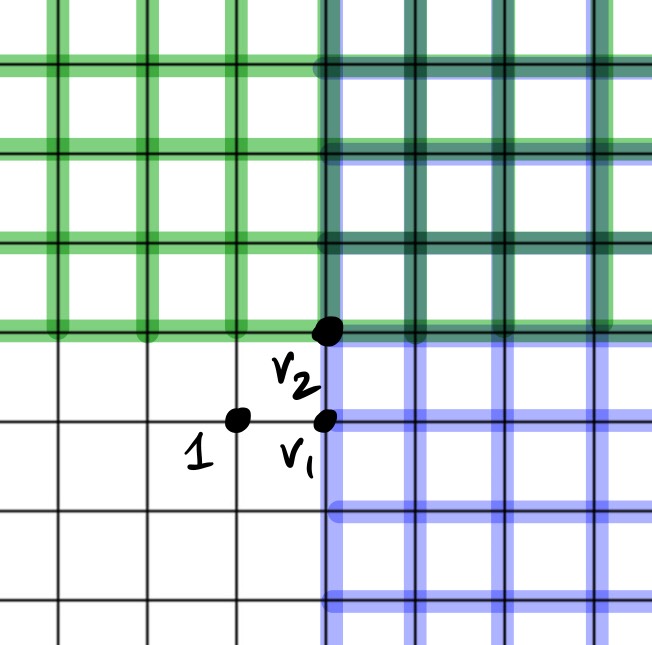}
			\includegraphics[scale=0.18]{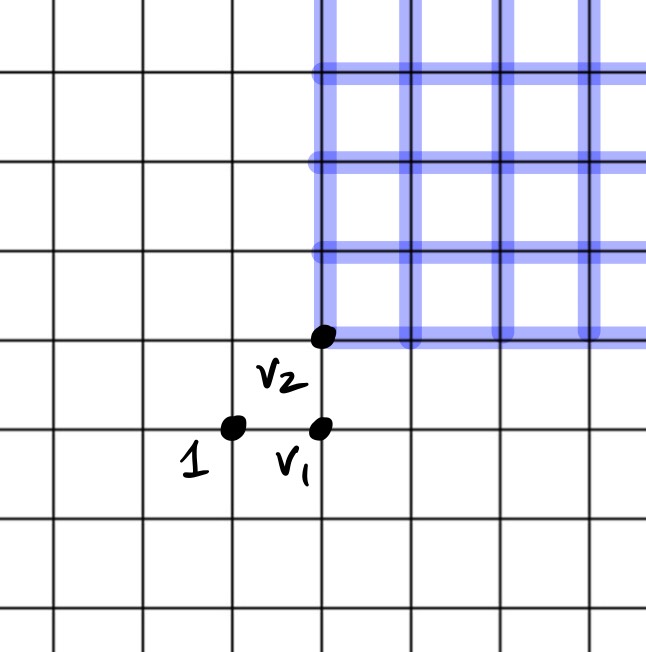}
			\includegraphics[scale=0.15]{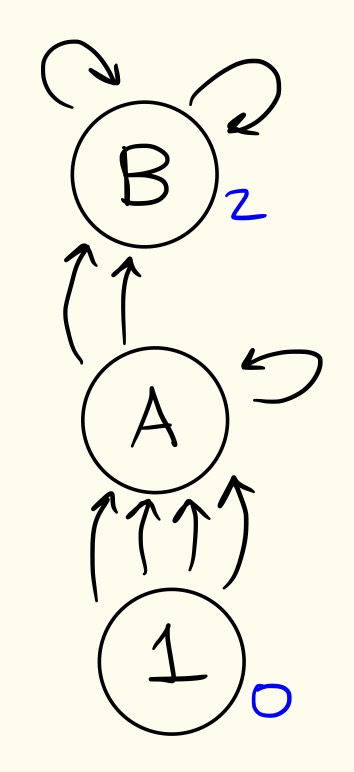}
			\caption{Getting an ext.cone in $A^+_{v_2}$ and the ext.cone type diagram of $4^4$}
			\label{fig:4.2.2}
		\end{center}
	\end{figure}	 
	\\
	Call the ext.cone type that $\bar{C}(v_1)$ belongs to $A$ and the ext.cone type that $\bar{C}(v_2)$ belongs to $B$. If we make all the needed calculations as in the $3^6$ graph, it turns out that ext.cone types $A$, $B$ and the ext.cone at the identity $1$ are all the ext.cone types of $4^4$ and the corresponding Cannon's numbers can be summarized in the extended cone types diagram found in Figure \ref{fig:4.2.2}.
	Using Cannon's system of equations we end up with:
	$$ \Delta(z) = \dfrac{(z+1)^2}{(z-1)^2}.$$
	
	\subsection{The Hexagonal Tiling $6^3$}
	
	Notice that $6^3$ is also bipartite, making cones and ext.cones the same. Likewise to the $3^6$ graph, we consider lines ...$l_{-3}$, $l_{-2}$, $l_{-1}$, $l_{0}$, $l_{1}$, $l_{2}$, $l_{3}$... as shown in Figure \ref{fig:4.3.1}. 
	For any two vertices $v, w \in l_0 \cup l_1$, we can replace any path from $v$ to $w$ that crosses vertices outside $l_0 \cup l_1$ with a shorter path that lies on $l_0 \cup l_1$. 
	This observation allows us to inductively compute the norms of all the vertices of $l_0 \cup l_1$ and show that the rays $r_1$, $r_2$, $r_3$ and $r_4$ in Figure \ref{fig:4.3.1} are all geodesic rays, so $r_3 \cup r_4 \subset C(v_1)$.
	Similarly any path from $1$ to $w \in l_0$ which does not lie on $l_0$ is not a geodesic, and therefore $r_1 \cup r_2 \subset \Gamma \backslash C(v_1)$.
	\begin{figure}[h]
		\begin{center}
			\includegraphics[scale=0.3]{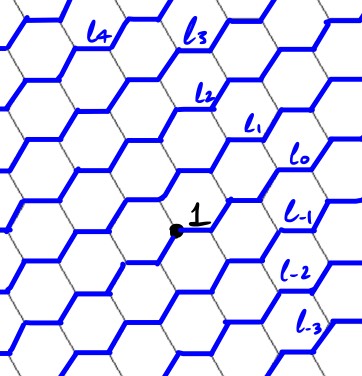}
			\includegraphics[scale=0.27]{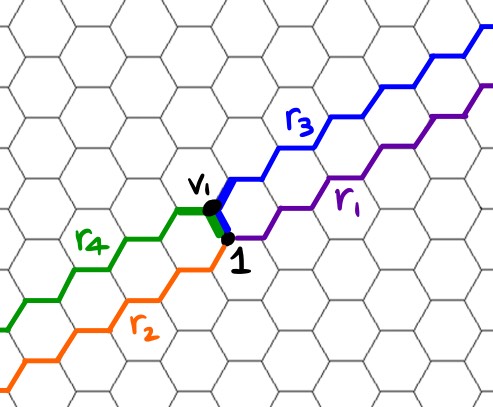}
			\caption{Geodesic rays in $6^3$}
			\label{fig:4.3.1}
		\end{center}
	\end{figure}	 
	\\
	Applying Proposition 3.10, we can compute the cone of $v_1$, an element on the 1-sphere. Through translating and intersecting cones via Proposition 3.8, we end up with 5 different ext.cone types, the ext.cone at the identity 1 and ext.cones A, B, C, D. Representatives $\bar{C}(v_1), \bar{C}(v_2), \bar{C}(v_3), \bar{C}(v_4)$ of each of the four ext. cone types $A, B, C, D$ respectively are shown in Figure \ref{fig:4.3.2}.
	
	\begin{figure}[h]
		\begin{center}
			\includegraphics[scale=0.18]{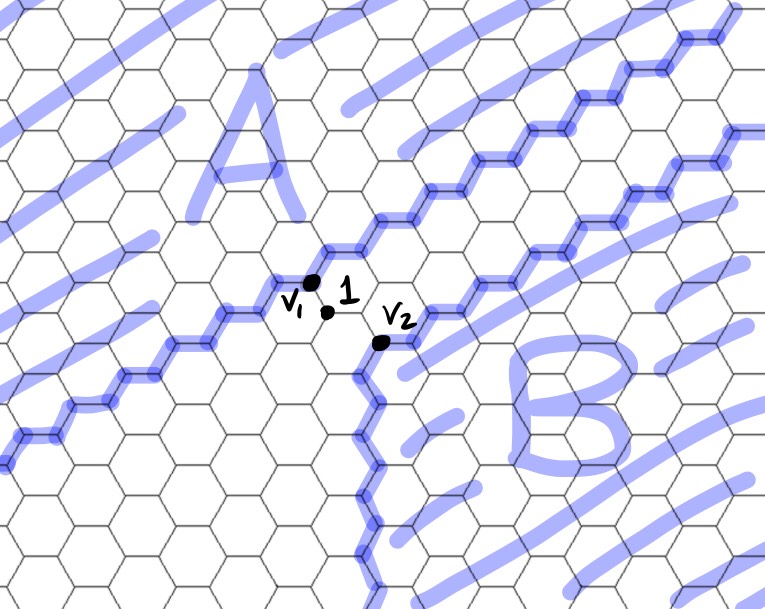}
			\includegraphics[scale=0.18]{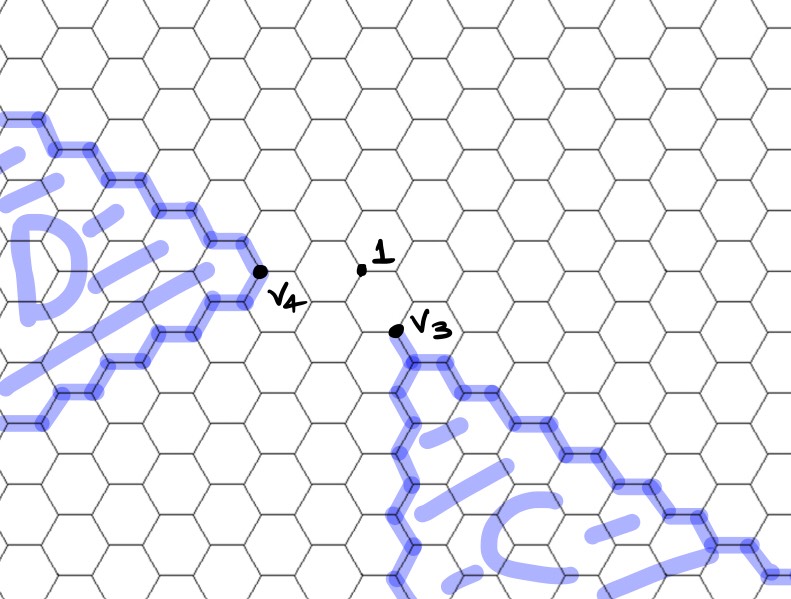}
			\includegraphics[scale=0.15]{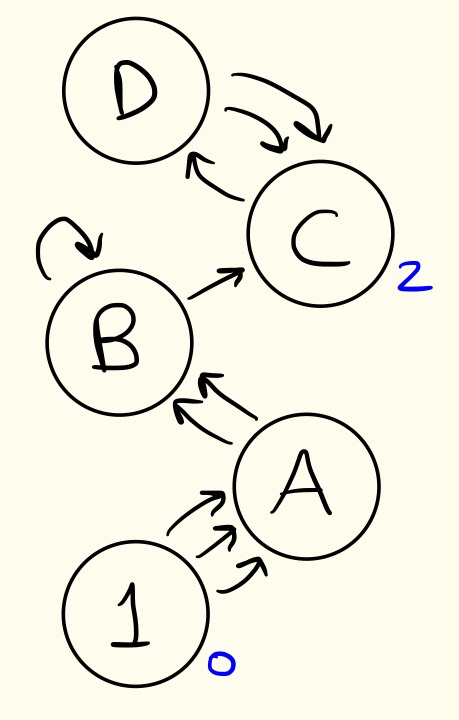}
			\caption{The ext.cone types of and the ext.cone type diagram of $6^3$}
			\label{fig:4.3.2}
		\end{center}
	\end{figure}

	The corresponding ext.cone type diagram (Figure \ref{fig:4.3.2}) is then used get the spherical growth series:
	$$ \Delta(z) = \dfrac{z^2 + z + 1}{(z-1)^2}.$$	

	\subsection{The $(3.6)^2$ Tiling}
	
	%This case is a little bit harder and trickier, since our graph is not bipartite. We have 8 ext.cone types in total.
	We begin, as in the previous cases, with constructing lines $(l_k)_{k\in Z}$ and $(l'_k)_{k\in Z}$ as shown in Figure \ref{fig:4.4.1}.  
	\begin{figure}[h]
		\begin{center}
			\includegraphics[scale=0.17]{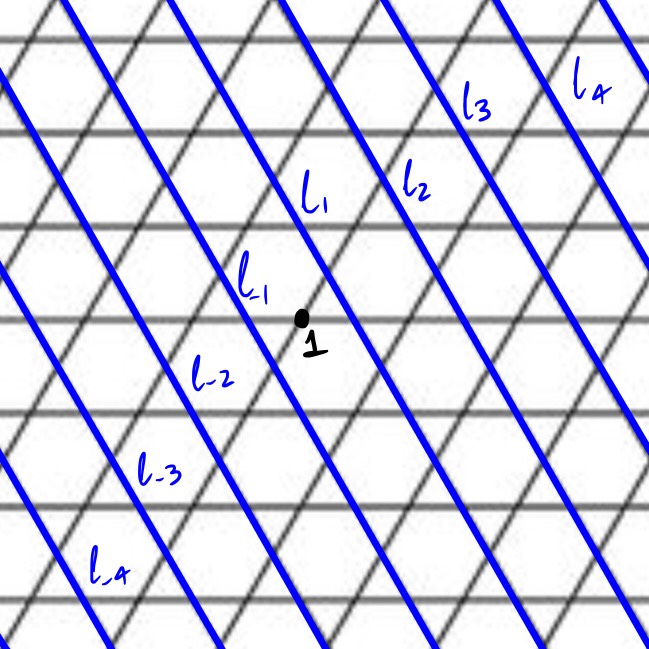}
			\includegraphics[scale=0.17]{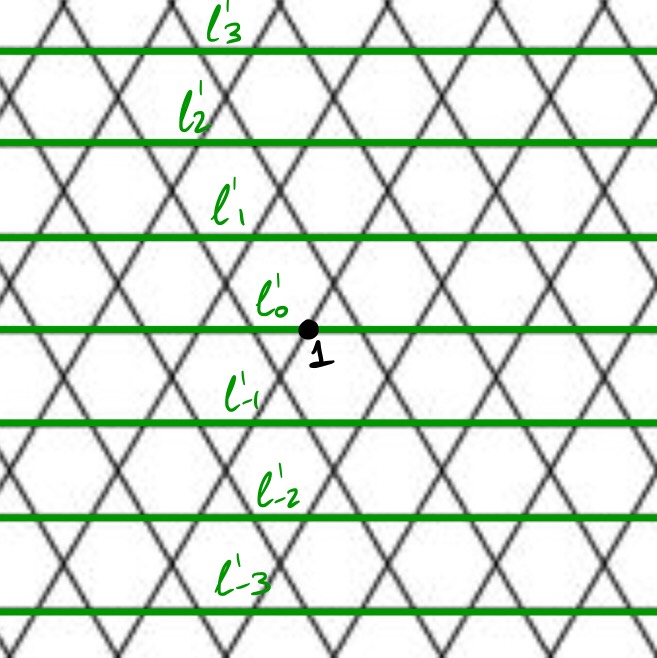}
			\includegraphics[scale=0.17]{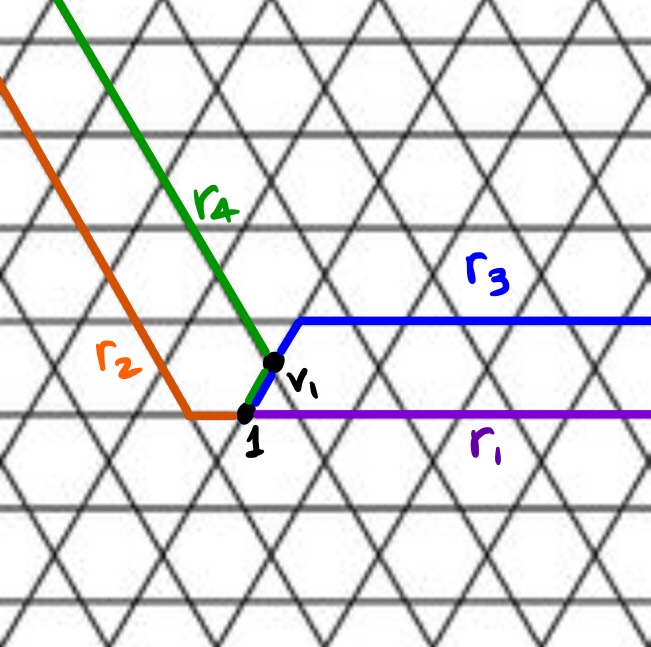}
			\caption{Geodesic rays in $(3.6)^2$}
			\label{fig:4.4.1}
		\end{center}
	\end{figure}	 
	Similarly to previous arguments, we can conclude that the $r_1$, $r_2$, $r_3$ and $r_4$ in Figure \ref{fig:4.2.2} are geodesic rays.
	Using an inductive argument, we can compute the norms of all the vertices that lie on or between $r_1$ and $r_3$ and on or between $r_2$ and $r_4$. 
	We then get (via Proposition 3.10) the extended cone of a vertex $v_1$ adjacent to $1$, whose ext.cone type we shall call $A$. Similarly we obtain all the other ext.cones at the 1-sphere at $1$ and it turns out that they are all strongly equivalent. 
	\begin{figure}[h]
		\begin{center}
			\includegraphics[scale=0.17]{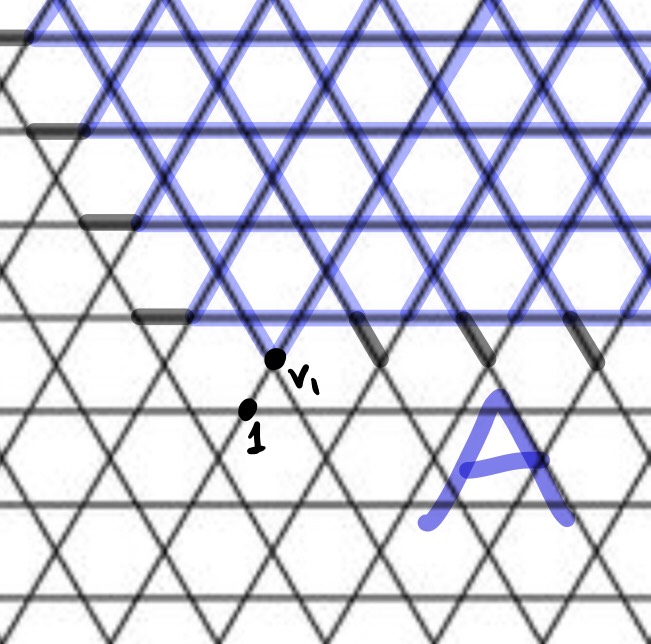}
			\includegraphics[scale=0.17]{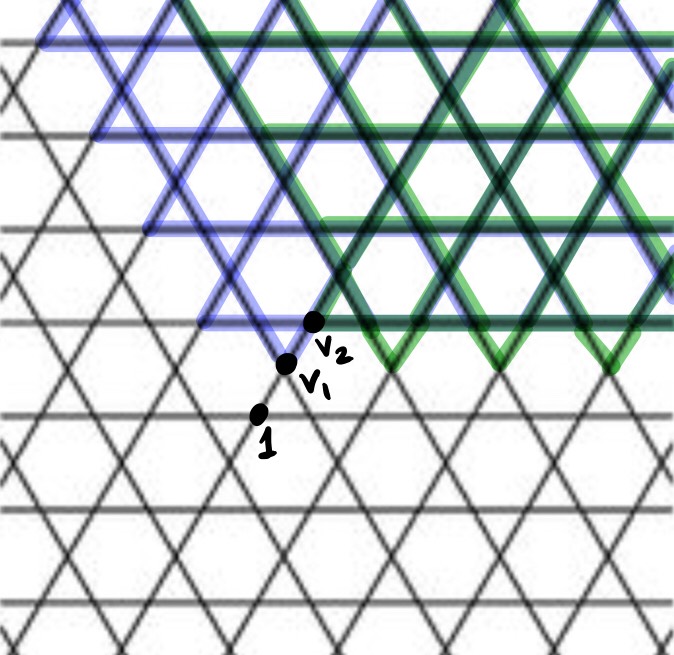}
			\includegraphics[scale=0.17]{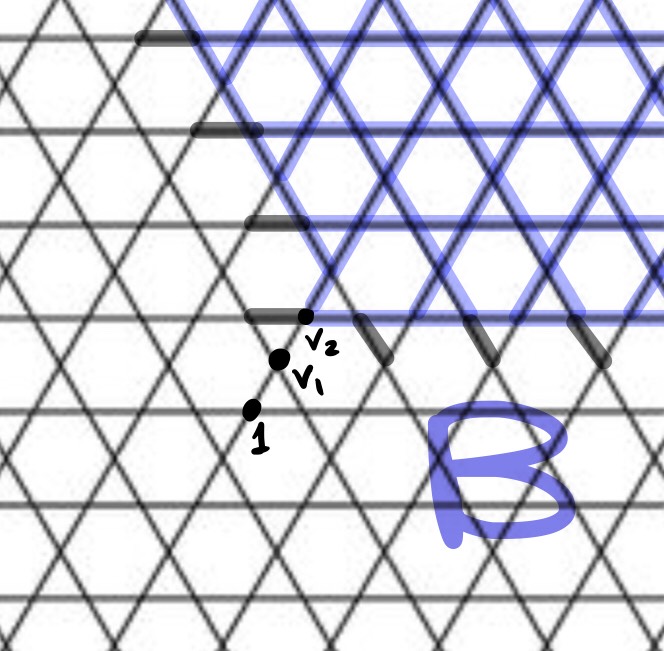}
			\caption{Obtaining ext.cone types A and B}
			\label{fig:4.4.2}
		\end{center}
	\end{figure}	 		
	Next, we apply Propositions 3.1 and 3.2 to find the extended cone type, which we will call $B$, of a vertex $v_2 \in A^+_{v_1}$ (see Figure \ref{fig:4.4.2}).
			
	Proceeding in the same manner, we get in total 8 ext.cone types. Figures \ref{fig:4.4.3} and \ref{fig:4.4.4} depict representative ext.cones of the remaining 5 ext.cones types, while Figure \ref{fig:4.4.4} also shows the extended cone type diagram of this graph. It is noteworthy to remark that an ext.cone of type $F$ is a finite graph with only 3 vertices.
	\begin{figure}[h]
		\begin{center}
			\includegraphics[scale=0.15]{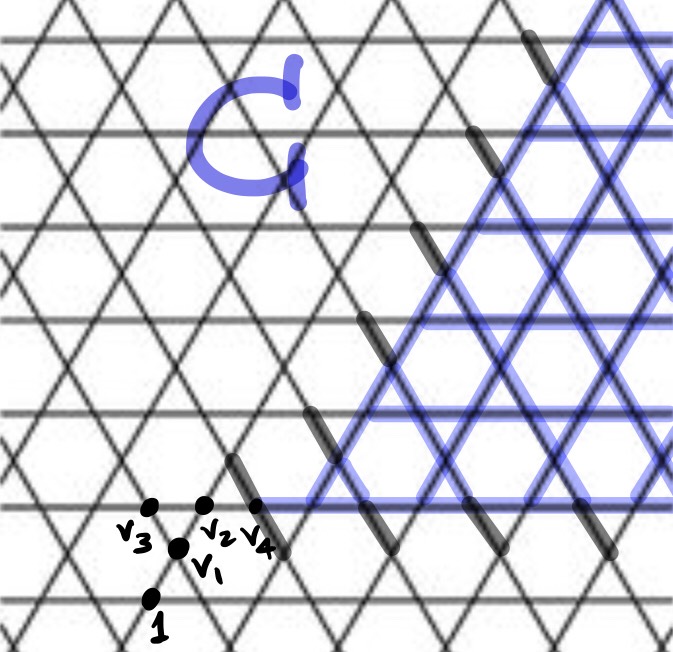}
			\includegraphics[scale=0.15]{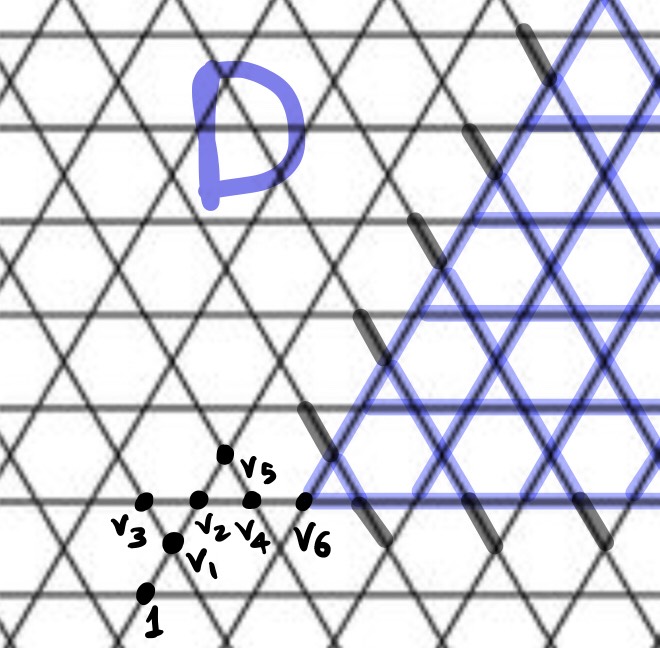}
			\includegraphics[scale=0.15]{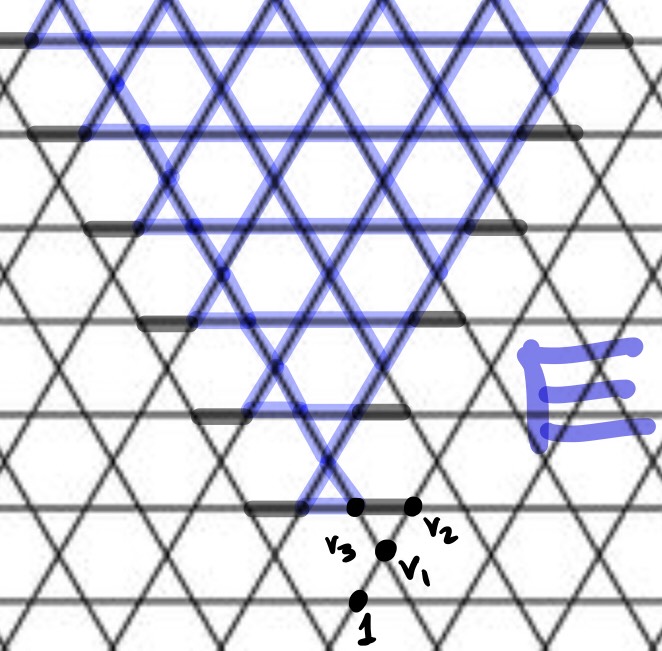}
			\caption{Ext.cone types C, D and E}
			\label{fig:4.4.3}
		\end{center}
	\end{figure}			
	\begin{figure}[h]
		\begin{center}			
			\includegraphics[scale=0.15]{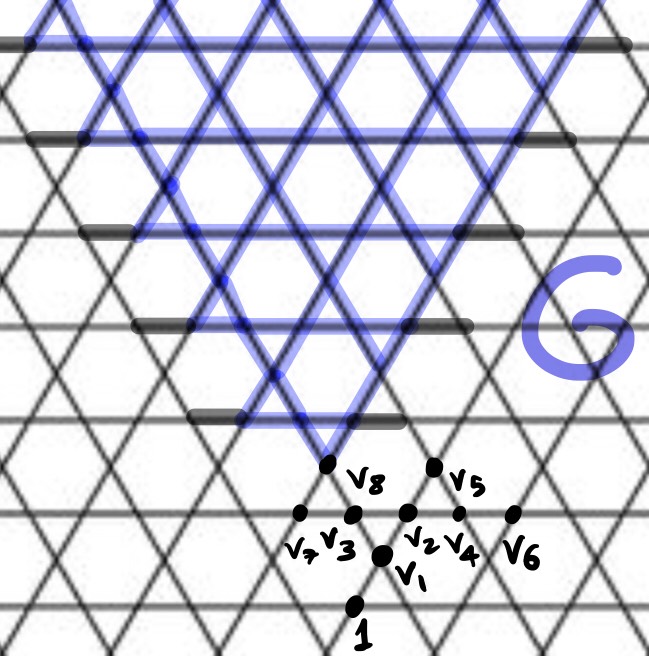}
			\includegraphics[scale=0.15]{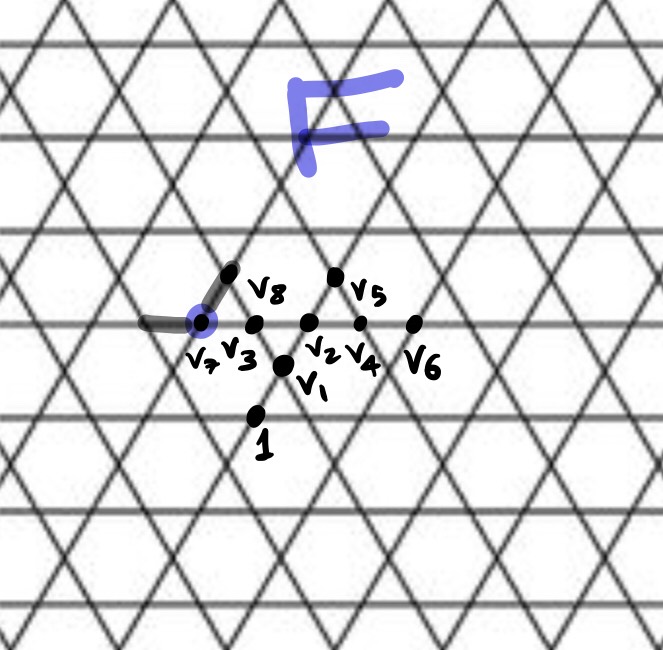}
			\includegraphics[scale=0.12]{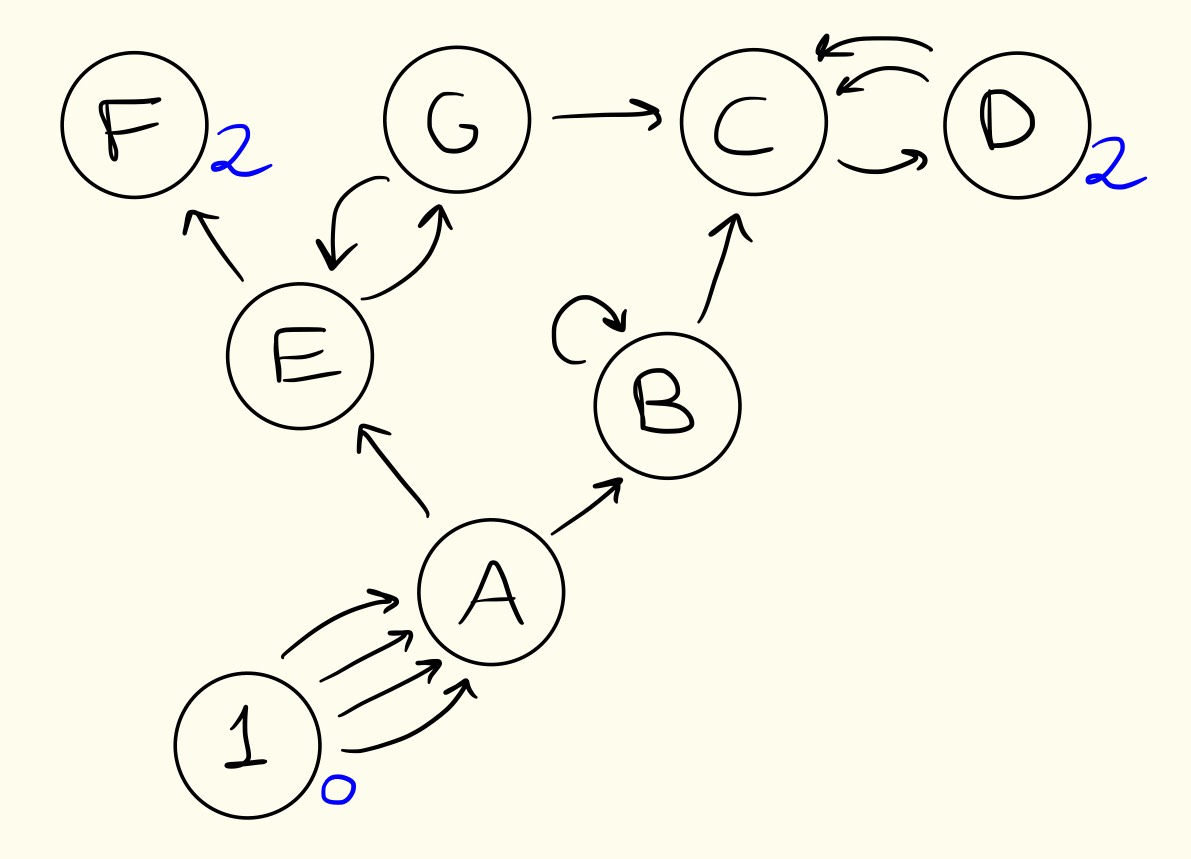}
			\caption{Ext.cone types E and F, and the ext.cone type diagram of $(3.6)^2$}
			\label{fig:4.4.4}
		\end{center}
	\end{figure}
	After solving Cannon's system of equations, we get the spherical growth series:
	$$ \Delta(z) = \dfrac{-z^5+3z^4+6z^3+6z^2+4z+1}{(1-z^2)^2}.$$
	
	\subsection{The $4.8^2$ Tiling}

	Once again, since our graph is bipartite, we do not have to worry about extended cones. We construct lines $(l_k)_{k\in \mathbb{Z}}$ as shown in Figure \ref{fig:4.5.1} to conclude (via similar arguments) that the rays $r_1, r_2, r_3, r_4$ are geodesic rays and hence we obtain the cone with cone type $I1$ (also in Figure \ref{fig:4.5.1}).
	\begin{figure}[h]
		\begin{center}
			\includegraphics[scale=0.127]{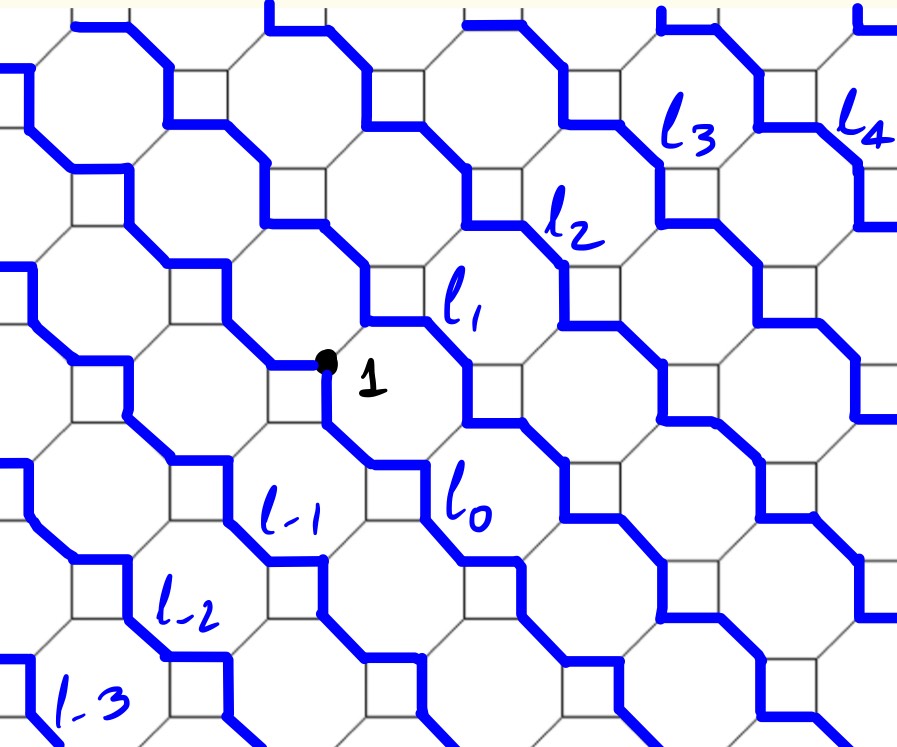}
			\includegraphics[scale=0.127]{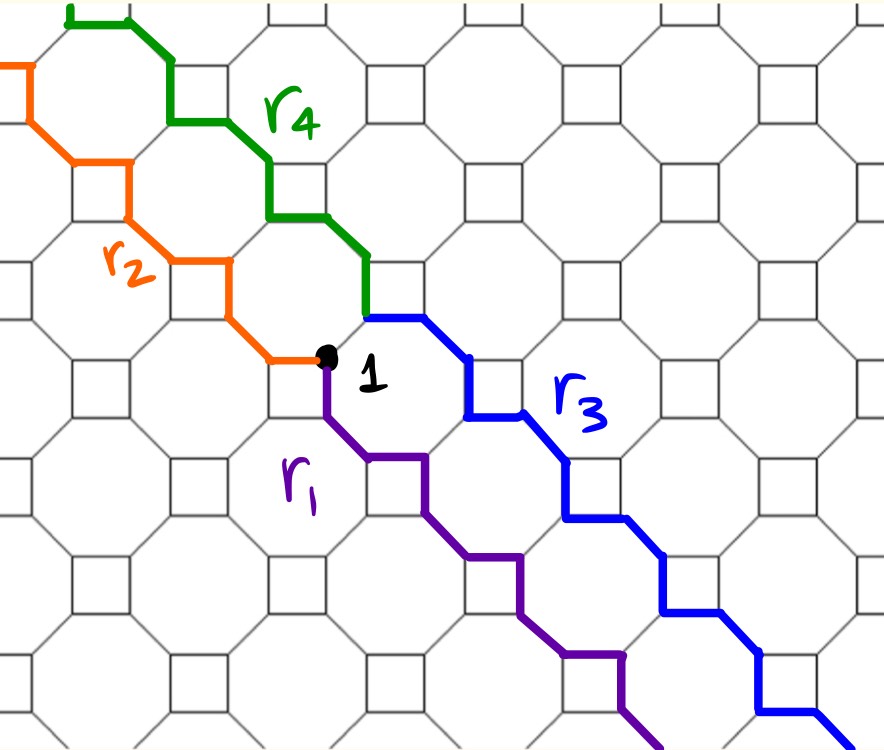}
			\includegraphics[scale=0.127]{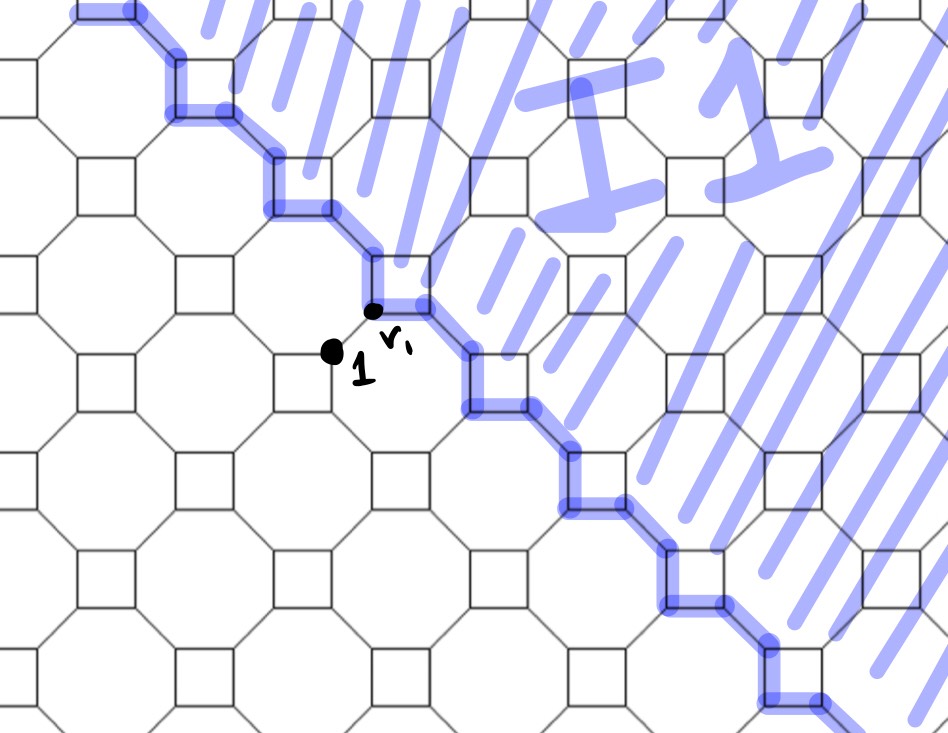}
			\caption{Obtaining cone type $I1$}
			\label{fig:4.5.1}
		\end{center}
	\end{figure}	 
	\begin{figure}[h]
		\begin{center}
			\includegraphics[scale=0.127]{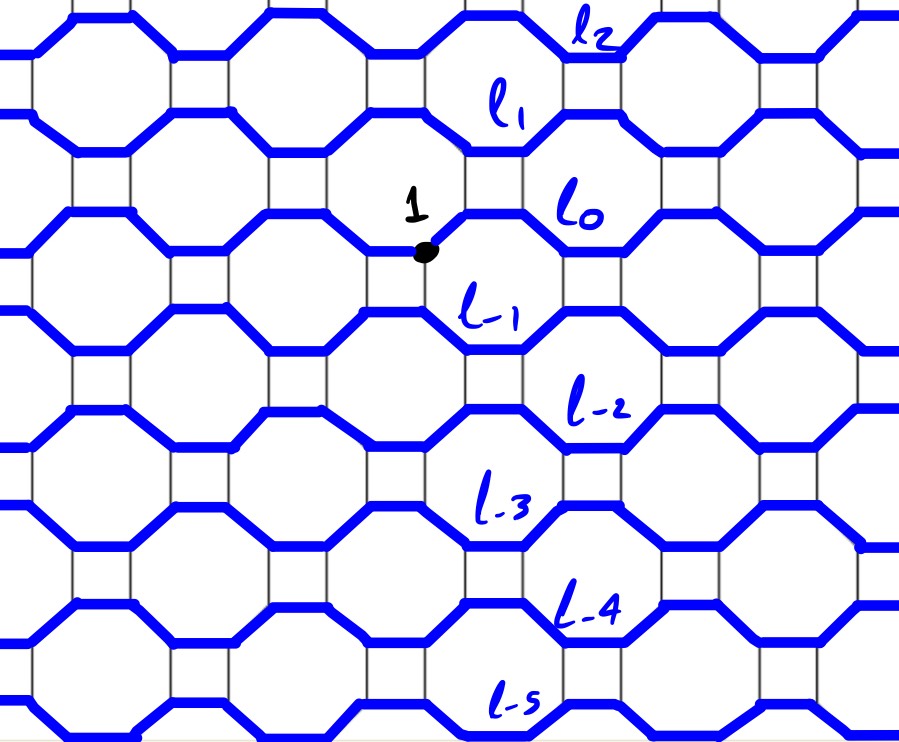}
			\includegraphics[scale=0.127]{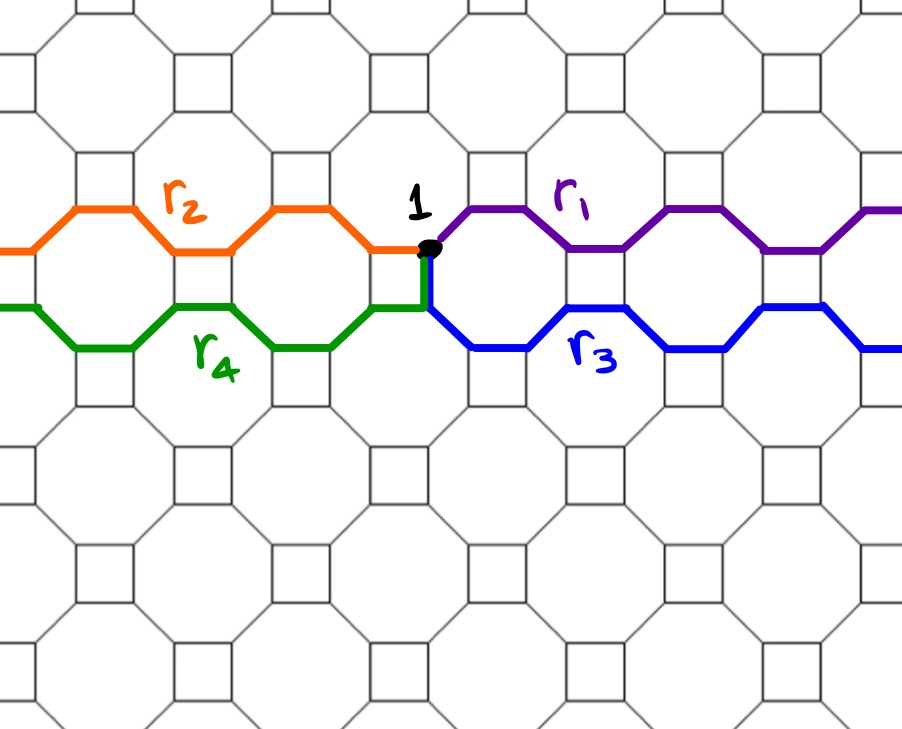}
			\includegraphics[scale=0.127]{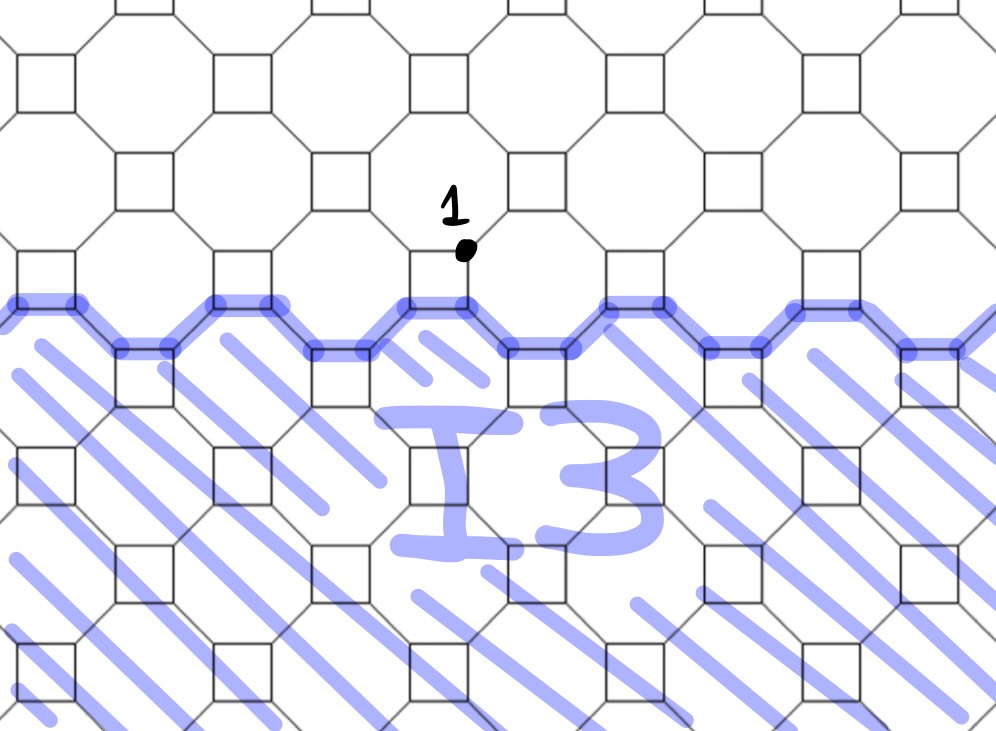}
			\caption{Obtaining cone type $I3$}
			\label{fig:4.5.2}
		\end{center}
	\end{figure}
	Applying the same procedure again we obtain cone type $I3$ (see Figure \ref{fig:4.5.2}), and that way we have found all the cone types at the discrete 1 sphere. Taking translations and intersections of these cones, we find all 14 different (ext.) cone types. The findings can be summarized in the cone type diagram  in Figure \ref{fig:4.5.3}. Using Cannon's system of equations, we arrive at:
	$$ \Delta(z) = \dfrac{(z^2+1)(z+1)^2}{(1-z)^2(1+z+z^2)}.$$	 
	\begin{figure}[h]
		\begin{center}
			\includegraphics[scale=0.16]{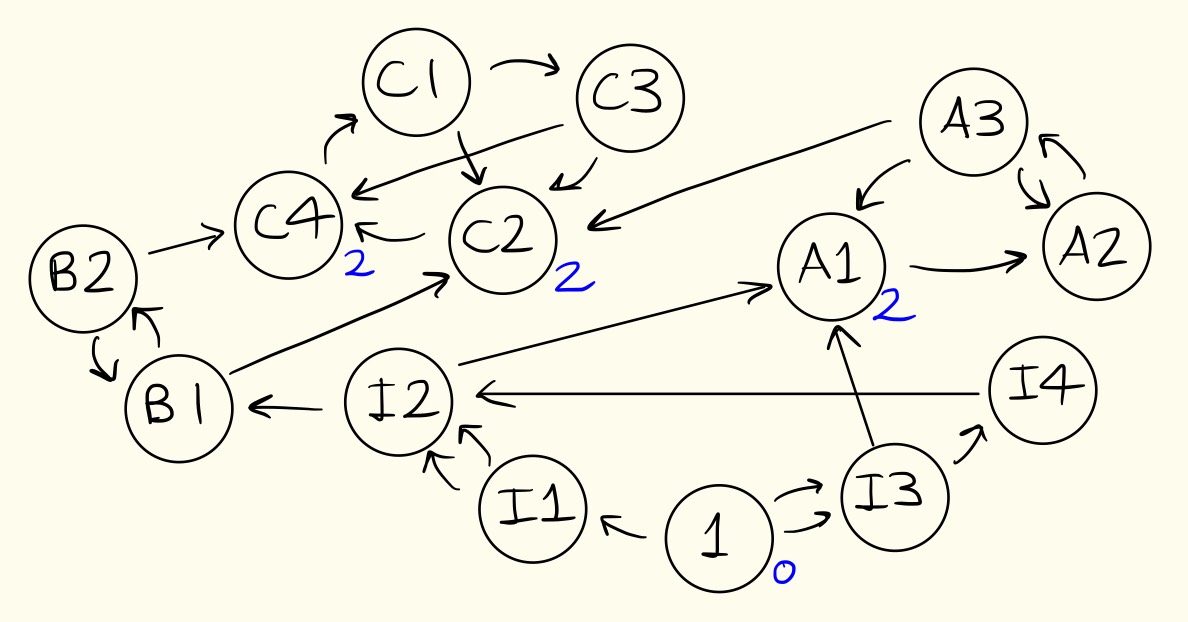}
			\caption{Ext.cone type diagram for $4.8^2$}
			\label{fig:4.5.3}
		\end{center}
	\end{figure}	
	
	\subsection{The $3.12^2$ Tiling}

	Via constructing lines  $(l_k)_{k\in \mathbb{Z}}$ we get geodesic rays $r_1, r_2, r_3, r_4$ which allow us to compute the ext.cones at the 1-sphere, having ext.cone type $I1$ or $I3$ (see Figure \ref{fig:4.6.1}). 
	\begin{figure}[h]
		\begin{center}
			\includegraphics[scale=0.147]{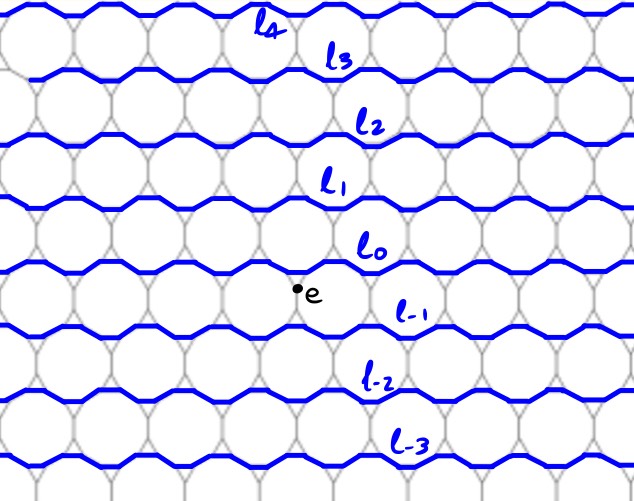}
			\includegraphics[scale=0.147]{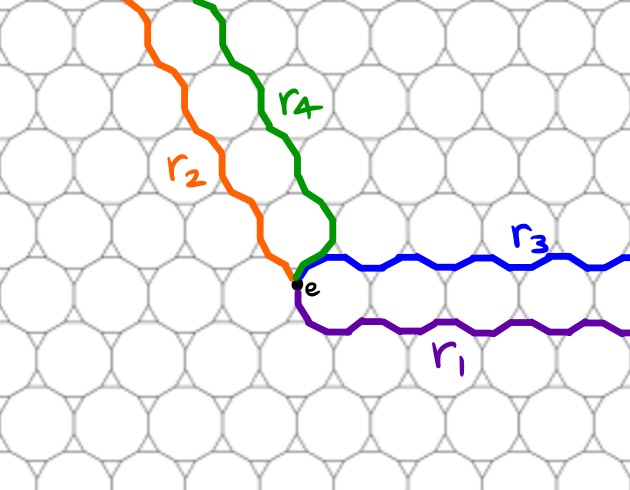}
			\includegraphics[scale=0.147]{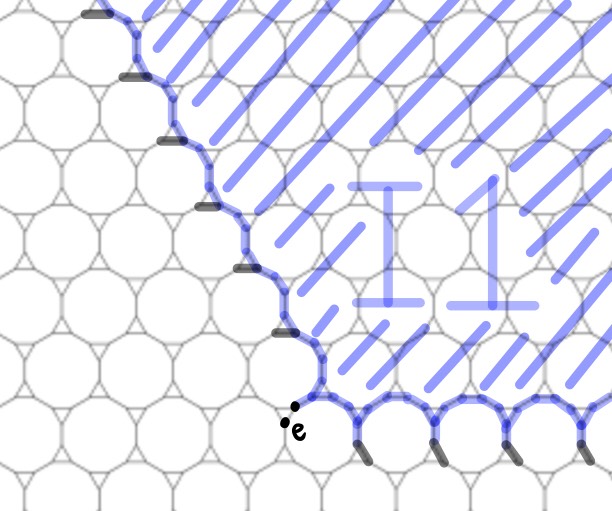}
			\includegraphics[scale=0.147]{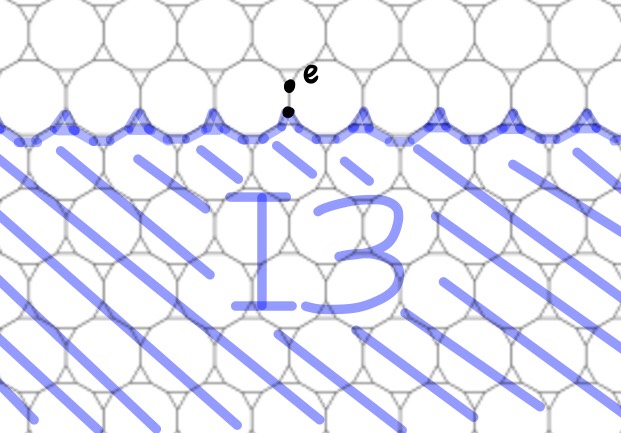}
			\caption{Ext.cone types of the 1-sphere}
			\label{fig:4.6.1}
		\end{center}
	\end{figure}
	Translating and intersecting these cones (Proposition 3.8) while keeping track of the tangent edges added for the extended cones (Proposition 3.9), we manage to find all 28 ext.cone types. 
	\\
	Unfortunately, we cannot display them all, so instead we provide a picture (Figure \ref{fig:4.6.2}) the graph close to the identity where each vertex is colored according to its ext.cone type (the purple colored vertices have ext.cone types $D1,D2,D3,D4$ at the top row and $E1,E2,E3,E4$ at the bottom row). 
	As always, we provide a ext.cone type diagram too (Figure \ref{fig:4.6.2}).
	Cannon's system of equations yields:
	$$ \Delta(z) = \dfrac{-2z^8+4z^7-3z^6+5z^5-z^4+3z^3+z^2+z+1}{(1-z)^2(1+z^2)^2}.$$
	\begin{figure}[h]
		\begin{center}
			\includegraphics[scale=0.137]{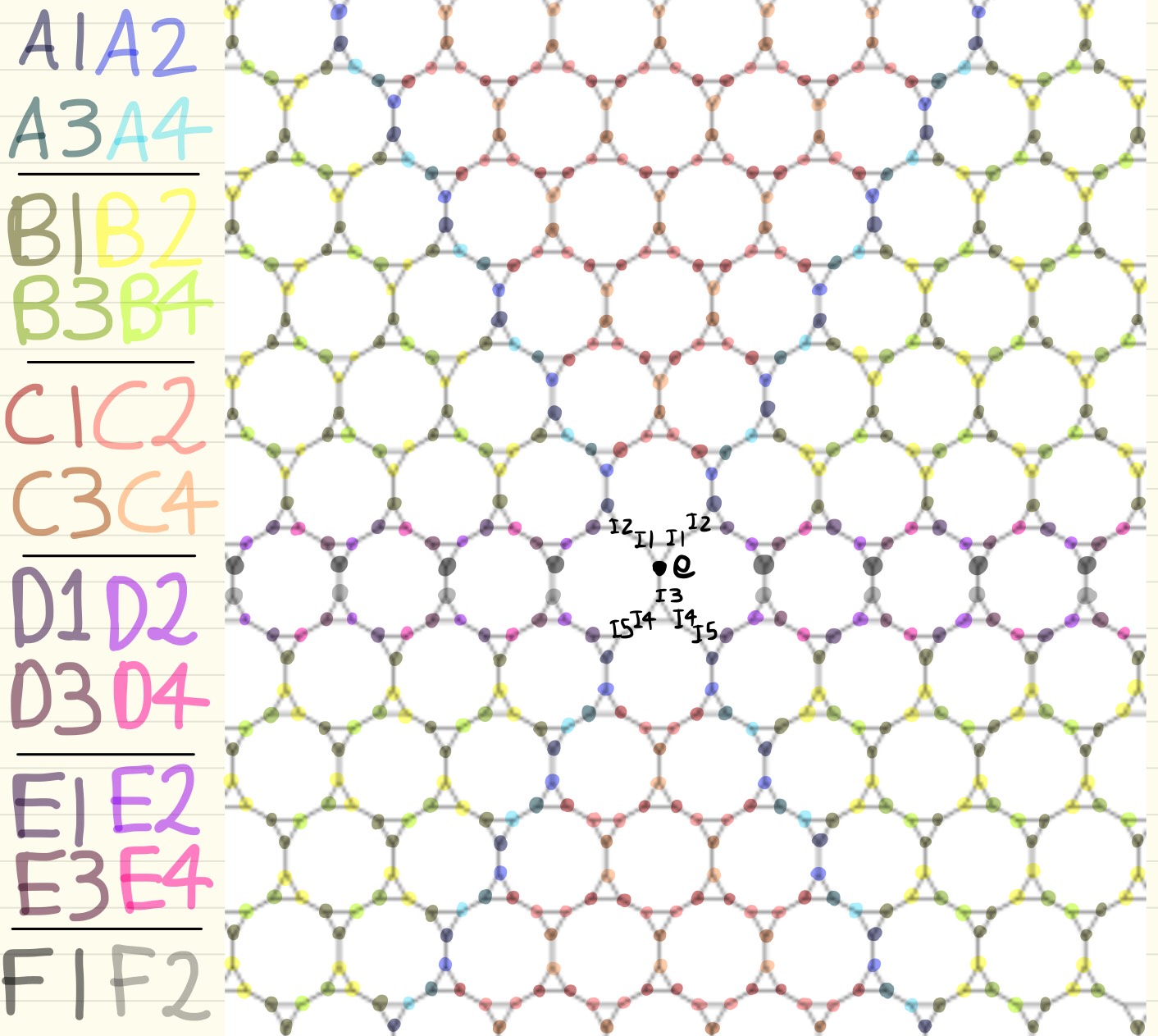}
			\includegraphics[scale=0.137]{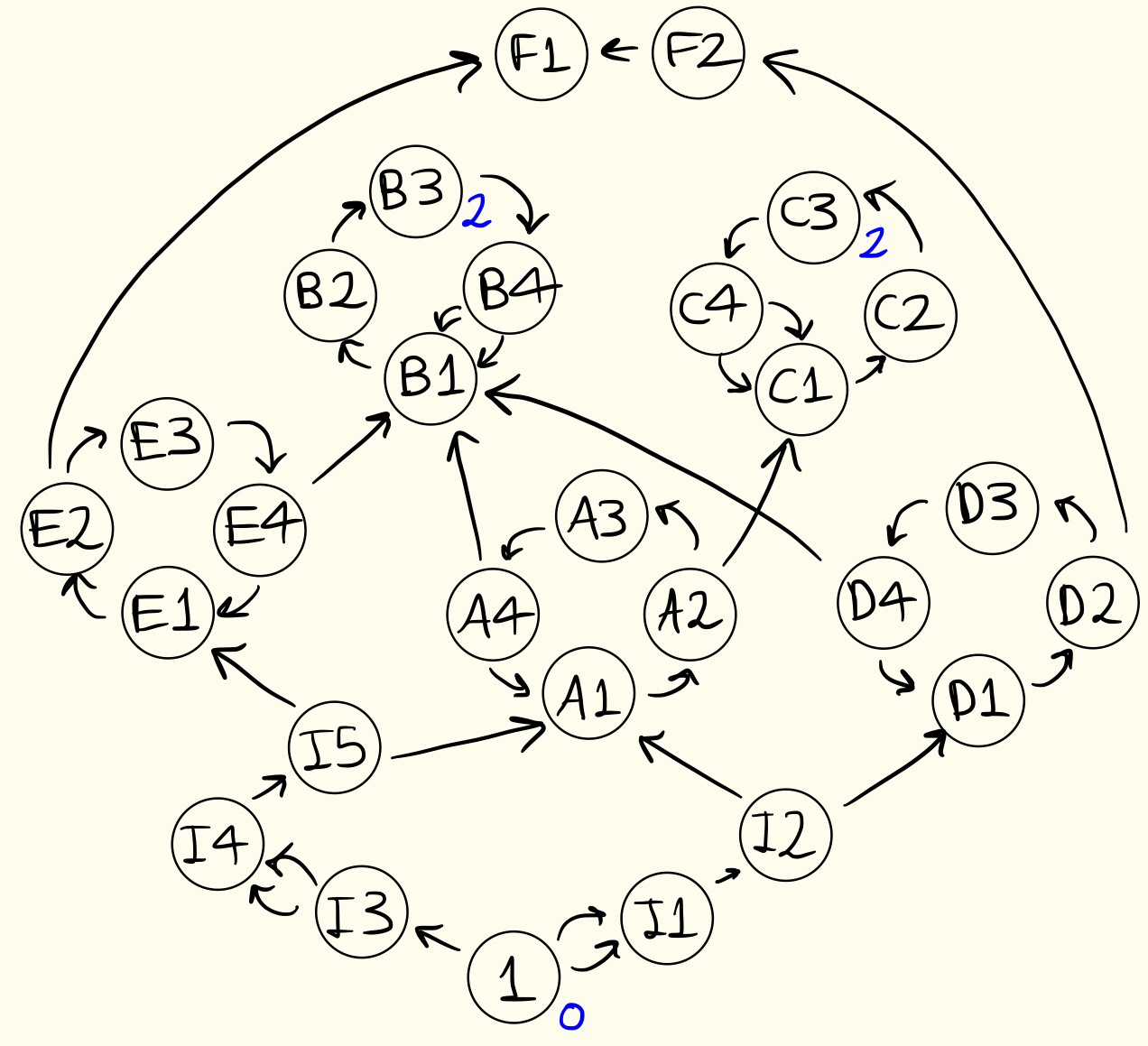}
			\caption{Pretty picture and ext.cone type diagram for $3.12^2$}
			\label{fig:4.6.2}
		\end{center}
	\end{figure}	

	\subsection{The $4.6.12$ Tiling}
	
	This is by far the toughest case, even though this graph is bipartite. In total, we have 42 cone types! Of course we do not display them all. Instead, we display all the cones of the 1-sphere (Figure \ref{fig:4.7.1}), pictures of vertices near the identity labeled according to their cone type (Figure \ref{fig:4.7.2}) and the ext.cone type diagram (Figure \ref{fig:4.7.3}). After solving carefully Cannon's system of linear equations, we obtain:
	$$ \Delta(z) = \dfrac{(z^2+z+1)(z^2-z+1)(z+1)^2}{(1-z^5)(1-z)}.$$
	\begin{figure}[h]
		\begin{center}
			\includegraphics[scale=0.16]{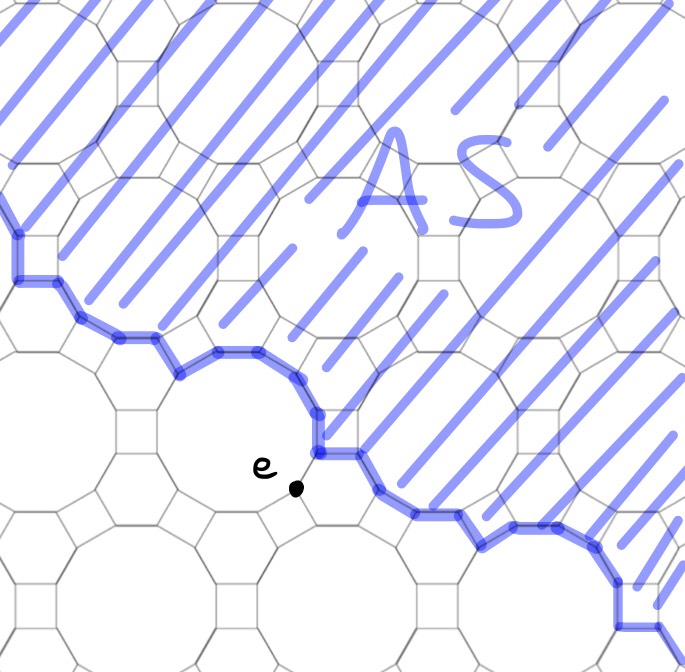}
			\includegraphics[scale=0.16]{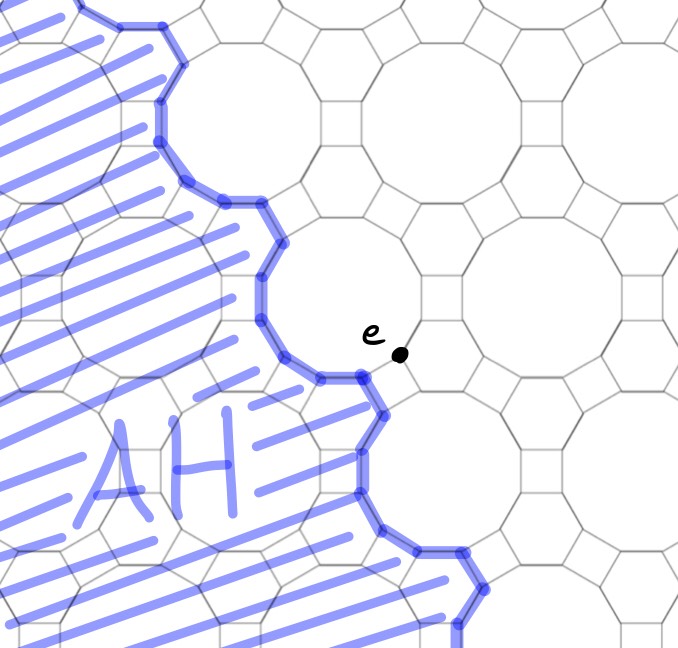}
			\includegraphics[scale=0.16]{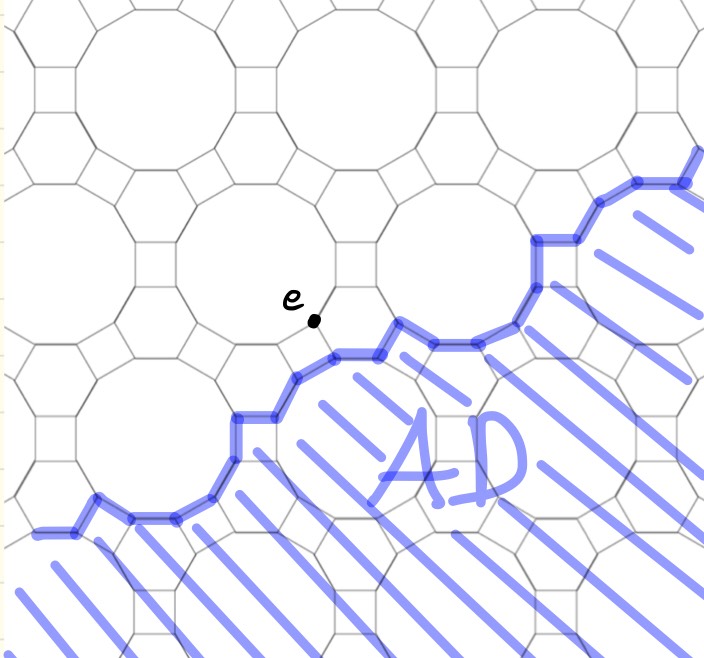}
			\caption{Ext.cone types of the 1-sphere}
			\label{fig:4.7.1}
		\end{center}
	\end{figure}	
	\begin{figure}[h]
		\begin{center}
			\includegraphics[scale=0.15]{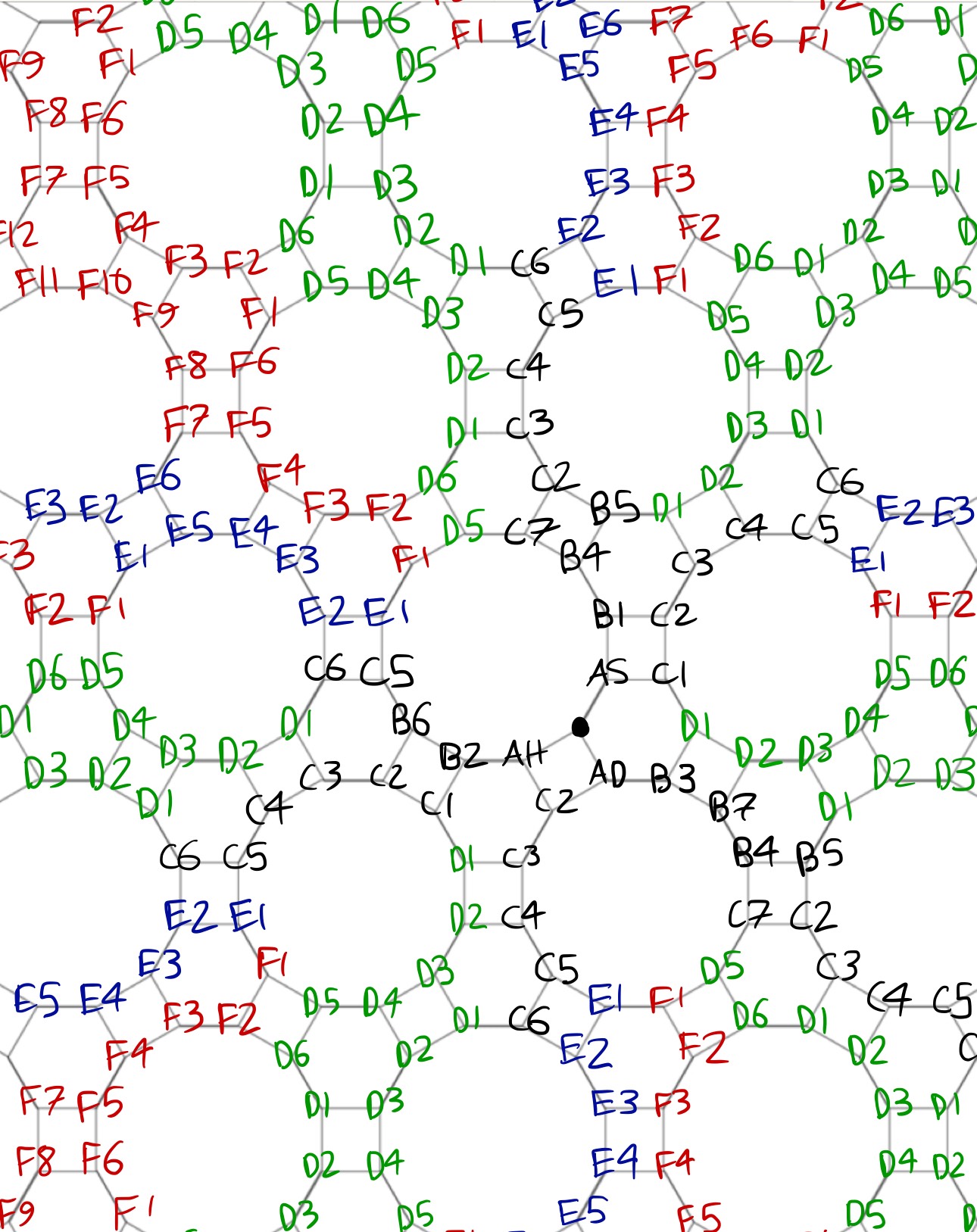}
			\includegraphics[scale=0.15]{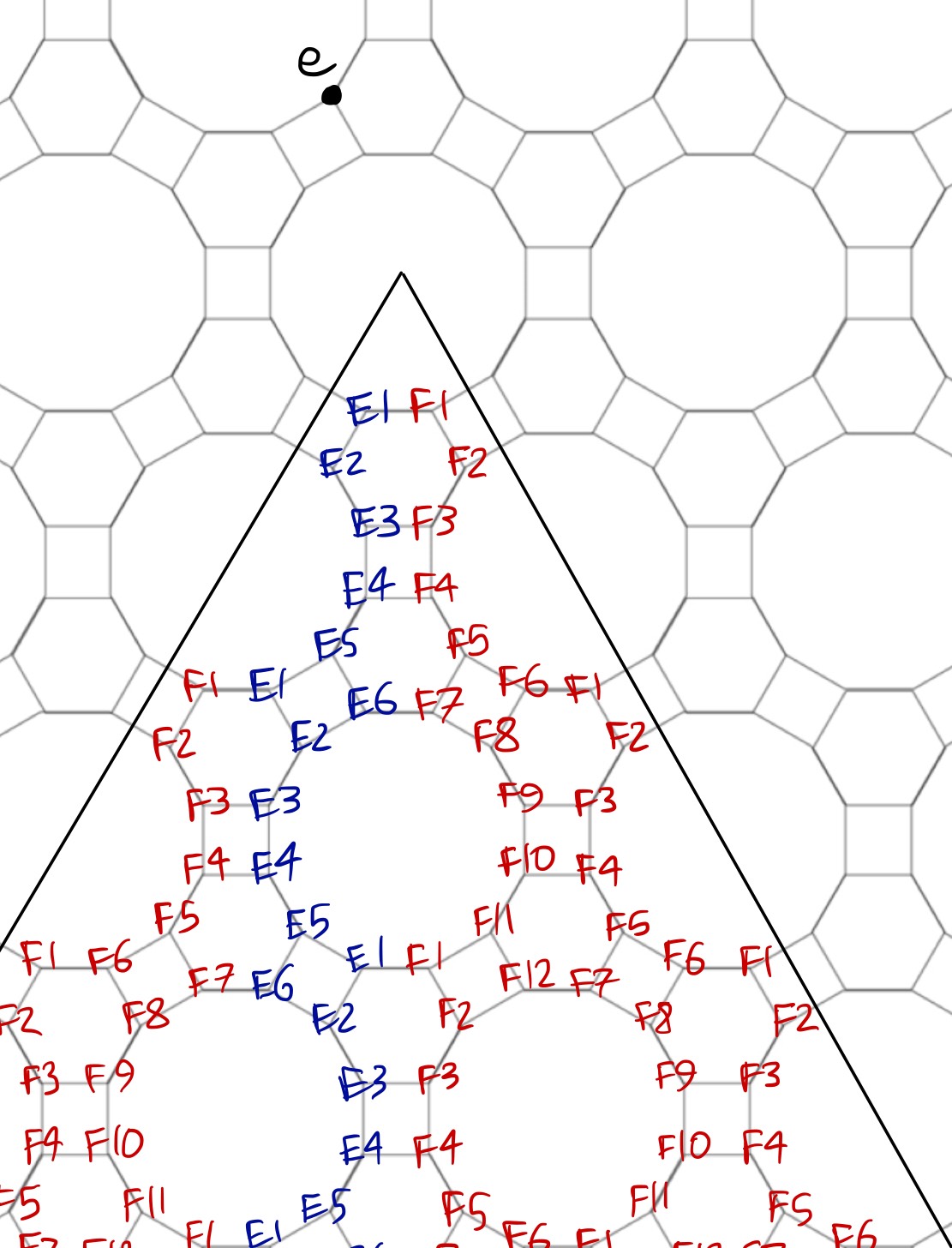}
			\caption{The cone types of vertices near the identity}
			\label{fig:4.7.2}
		\end{center}
	\end{figure}	
	\begin{figure}[h]
		\begin{center}			
			\includegraphics[scale=0.14]{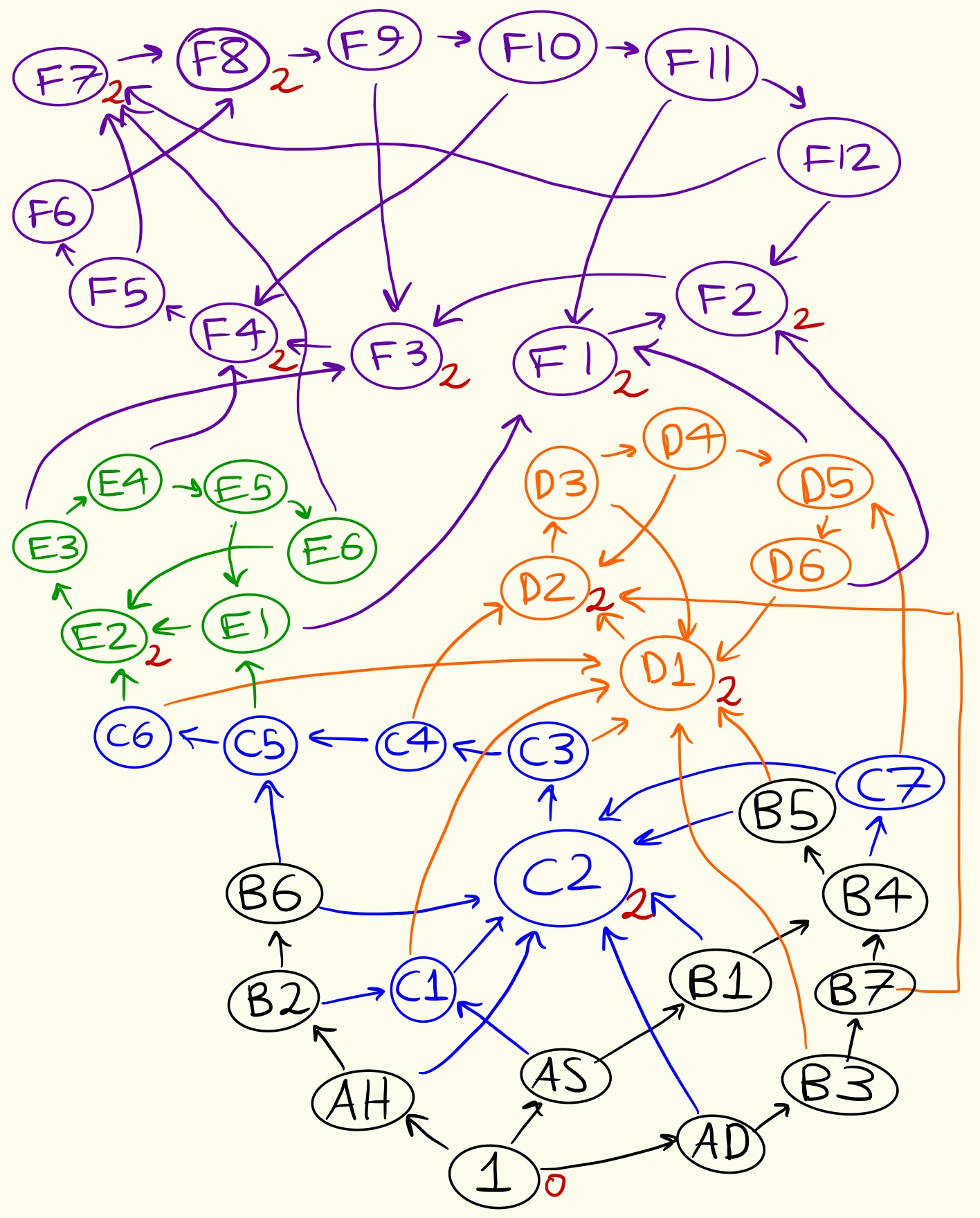}
			\caption{The ext.cone type diagram of $4.6.12$}
			\label{fig:4.7.3}
		\end{center}
	\end{figure}	
	
\section{A Summary of The Results}

	Applying Cannon's method of cone types to each of the 7 non-oriented Cayley graphs,
	we obtain the spherical growth series.
	By taking a partial fraction decomposition of the growth series, 
	we can compute the Taylor expansion of the growth series, i.e. the growth function.
	We take the Taylor expansions of the spherical growth series $\Delta(z)$ 
	and the cumulative growth series $\Gamma(z) = \dfrac{\Delta(z)}{1-z}$.
	We end up with the following tables:\\
	\\ \\ \\ \\ \\ \\ \\
	\\ \\ \\ \\ \\ \\ \\
	\\ \\ \\ \\ \\ \\ \\
	\begin{center}
	\begin{tabular}{ | c | c |}
	\hline
	\hline
	Cayley Graph $\Gamma(G,S)$ & Spherical Growth Series $\Delta(z)$\\
	\hline
	\hline
	$3^6$       & $\dfrac{z^2+4z+1}{(1-z)^2}$ \\
	\hline
	$4^4$       & $\dfrac{(1+z)^2}{(1-z)^2}$  \\
	\hline
	$6^3$       & $\dfrac{z^2+z+1}{(1-z)^2}$ \\
	\hline
	$(3.6)^2$ & $\dfrac{-2z^5+3z^4+6z^3+6z^2+4z+1}{(1-z^2)^2}$ \\
	\hline
	$4.8^2$    & $\dfrac{(z^2+1)(z+1)^2}{(1-z)^2(1+z+z^2)}$ \\
	\hline
	$3.12^2$  & $\dfrac{-2z^8+4z^7-3z^6+5z^5-z^4+3z^3+z^2+z+1}{(1-z)^2(1+z^2)^2}$ \\
	\hline
	$4.6.12$   & $\dfrac{(z^2+z+1)(z^2-z+1)(z+1)^2}{(1-z^5)(1-z)}$ \\
	\hline
	\hline
	\end{tabular}
	\end{center}

%	\begin{center}
%	\begin{tabular}{ | c | c |}
%		\hline
%		\hline
%		Cayley Graph $\Gamma(G,S)$ & Cumulative Growth Series $\Gamma(z)$\\
%		\hline
%		\hline
%		$3^6$       & $\dfrac{z^2+4z+1}{(1-z)^3}$ \\
%		\hline
%		$4^4$       & $\dfrac{(1+z)^2}{(1-z)^3}$  \\
%		\hline
%		$6^3$       & $\dfrac{z^2+z+1}{(1-z)^3}$ \\
%		\hline
%		$(3.6)^2$ & $\dfrac{-2z^5+3z^4+6z^3+6z^2+4z+1}{(1-z)(1-z^2)^2}$ \\
%		\hline
%		$4.8^2$    & $\dfrac{(z^2+1)(z+1)^2}{(1-z)^3(1+z+z^2)}$ \\
%		\hline
%		$3.12^2$  & $\dfrac{-2z^8+4z^7-3z^6+5z^5-z^4+3z^3+z^2+z+1}{(1-z)^3(1+z^2)^2}$ \\
%		\hline
%		$4.6.12$   & $\dfrac{(z^2+z+1)(z^2-z+1)(z+1)^2}{(1-z^5)(1-z)^2}$ \\
%		\hline
%		\hline
%	\end{tabular}
%	\end{center}

\newpage

	\begin{center}
	\begin{tabular}{ | c | c |}
	\hline
	\hline
	Cayley Graph $\Gamma(G,S)$ & Spherical Growth function $\delta(n)$\\
	\hline
	\hline
	$3^6$       &  	$\begin{array}{cc}
					1 & n=0 \\
					6n & n>0\\
				\end{array}$	  		\\
	\hline
	$4^4$       &  	$\begin{array}{cc}
					1 & n=0 \\
					4n & n>0\\
				\end{array}$	  		\\
	\hline
	$6^3$       &  	$\begin{array}{cc}
					1 & n=0 \\
					3n & n>0\\
				\end{array}$			\\
	\hline
	$(3.6)^2$ &  	$\begin{array}{cc}
					1 	& n=0  \\
					4 	& n=1  \\
					5n-2 	& n\equiv0\mod2,n>1\\
					4n+2 	& n\equiv1\mod2,n>1\\
				\end{array}$			\\
	\hline
	$4.8^2$    &	  	$\begin{array}{cc}
					1 				& n=0  \\
					\dfrac{8n}{3} 		& n\equiv0\mod3,n>0\\
					\dfrac{8n+1}{3} 		& n\equiv1\mod3,n>0\\
					\dfrac{8n-1}{3}		& n\equiv2\mod3,n>0\\
				\end{array}$			\\
	\hline
	$3.12^2$  &	  	$\begin{array}{cc}
					1 				& n=0  \\
					3 				& n=1  \\
					4 				& n=2  \\
					\dfrac{5n-4}{4} 		& n\equiv0\mod4,n>2\\
					\dfrac{9n+3}{4} 		& n\equiv1\mod4,n>2\\
					2n+2				& n\equiv2\mod4,n>2\\
					\dfrac{9n-3}{4}		& n\equiv3\mod4,n>2\\
				\end{array}$			\\
	\hline
	$4.6.12$   &	  	$\begin{array}{cc}
					1 				& n=0  \\
					\dfrac{12n}{5} 		& n\equiv0\mod5,n>0  \\
					\dfrac{12n+3}{5} 		& n\equiv1\mod5,n>0\\
					\dfrac{12n+1}{5} 		& n\equiv2\mod5,n>0\\
					\dfrac{12n-1}{5}		& n\equiv3\mod5,n>0\\
					\dfrac{12n-3}{5}		& n\equiv4\mod5,n>0\\
				\end{array}$			\\
	\hline
	\hline
	\end{tabular}
	\end{center}

\newpage

	\begin{center}
	\begin{tabular}{ | c | c |}
	\hline
	\hline
	Cayley Graph $\Gamma(G,S)$ & Cumulative Growth function $\gamma(n)$\\
	\hline
	\hline
	$3^6$       &  	$1+3n(n+1)$	  		\\
	\hline
	$4^4$       &  	$1+2n(n+1)$	  		\\
	\hline
	$6^3$       &  	$1+\dfrac{3}{2}n(n+1)$		\\
	\hline
	$(3.6)^2$ &  	$\begin{array}{cc}
					1 					& n=0  \\
					\dfrac{9n^2+10n-4}{4} 		& n\equiv0\mod2,n>1\\
					\dfrac{9n^2+8n+3}{4}  		& n\equiv1\mod2,n>1\\
				\end{array}$			\\
	\hline
	$4.8^2$    &	  	$\begin{array}{cc}
					\dfrac{3+4n(n+1)}{3} 		& n\equiv0\mod3,n>0\\
					\dfrac{4+4n(n+1)}{3} 		& n\equiv1\mod3,n>0\\
					\dfrac{3+4n(n+1)}{3}		& n\equiv2\mod3,n>0\\
				\end{array}$			\\
	\hline
	$3.12^2$  &	  	$\begin{array}{cc}
					1 					& n=0  \\
					4 					& n=1  \\
					\dfrac{9n^2+10n-8}{8} 		& n\equiv0\mod4,n>2\\
					\dfrac{9n^2+10n-3}{8} 		& n\equiv1\mod4,n>2\\
					\dfrac{9n^2+8n+12}{8}		& n\equiv2\mod4,n>2\\
					\dfrac{9n^2+8n+7}{8}		& n\equiv3\mod4,n>2\\
				\end{array}$			\\
	\hline
	$4.6.12$   &	  	$\begin{array}{cc}
					\dfrac{5+6n(n+1)}{5} 		& n\equiv0\mod5,n>0  \\
					\dfrac{8+6n(n+1)}{5} 		& n\equiv1\mod5,n>0\\
					\dfrac{9+6n(n+1)}{5} 		& n\equiv2\mod5,n>0\\
					\dfrac{8+6n(n+1)}{5}		& n\equiv3\mod5,n>0\\
					\dfrac{5+6n(n+1)}{5}		& n\equiv4\mod5,n>0\\
				\end{array}$			\\
	\hline
	\hline
	\end{tabular}
	\end{center}

\newpage

\bibliographystyle{plain}

\bibliography{bibliography}

\end{document}